\documentclass[11pt,reqno]{amsart}
\pdfoutput=1
\usepackage{amsmath, amsthm, amssymb}
\usepackage{enumitem}
\usepackage{graphicx} 
\usepackage{bbm}

\usepackage{fullpage}
\usepackage{empheq}
\usepackage[many]{tcolorbox}
\usepackage{xcolor}
\usepackage{wasysym}
\usepackage{mathtools}
\usepackage{scalerel}

\usepackage{tikz}
\usetikzlibrary{cd}
\usetikzlibrary{calc}
\usetikzlibrary{decorations,decorations.pathreplacing,decorations.markings,decorations.pathmorphing}
\tikzstyle{snake}=[decorate, decoration={snake, segment length=1mm, amplitude=.5mm}]
\newcommand{\tikzmath}[2][]
{\vcenter{\hbox{\begin{tikzpicture}[#1]#2\end{tikzpicture}}}
}

\tikzset{super thick/.style={line width=3pt}}
\tikzstyle{far>}=[decoration={markings, mark=at position 0.75 with {\arrow{>}}}, postaction={decorate}]
\tikzstyle{mid>}=[decoration={markings, mark=at position 0.55 with {\arrow{>}}}, postaction={decorate}]
\tikzstyle{mid<}=[decoration={markings, mark=at position 0.55 with {\arrow{<}}}, postaction={decorate}]
\tikzset{super thick/.style={line width=3pt}}
\tikzstyle{far>}=[decoration={markings, mark=at position 0.75 with {\arrow{>}}}, postaction={decorate}]
\tikzstyle{mid>}=[decoration={markings, mark=at position 0.55 with {\arrow{>}}}, postaction={decorate}]
\tikzstyle{mid<}=[decoration={markings, mark=at position 0.55 with {\arrow{<}}}, postaction={decorate}]
\tikzstyle{knot}=[preaction={super thick, white, draw}]
\tikzstyle{coupon}=[draw, very thick, rectangle, rounded corners=5pt]
\tikzset{Rightarrow/.style={double equal sign distance,>={Implies},->},
triplecd/.style={-,preaction={draw,Rightarrow}},
quadruplecd/.style={preaction={draw,Rightarrow,
shorten >=0pt
},
shorten >=1pt,
-,double,double
distance=0.2pt}}
\tikzset{
    tripleline/.style args={[#1] in [#2] in [#3]}{
        #1,preaction={preaction={draw,#3},draw,#2}
    }
}
\tikzstyle{triple}=[tripleline={[line width=.15mm,black] in
      [line width=.7mm,white] in
      [line width=1mm,black]}] 
\tikzset{
    quadrupleline/.style args={[#1] in [#2] in [#3] in [#4]}{
        #1,preaction={preaction={preaction={draw,#4},draw,#3}, draw,#2}
    }
}
\tikzstyle{quadruple}=[quadrupleline={[line width=.3mm,white] in
      [line width=.6mm,black] in
      [line width=1.2mm,white] in
      [line width=1.5mm,black]}]

\definecolor{violet}{RGB}{148,0,211}
\definecolor{DarkGreen}{RGB}{0,150,0}
\definecolor{rufous}{HTML}{A81C07}
\definecolor{boysenberry}{HTML}{873260}
\definecolor{OliveGreen}{HTML}{6D712E}
\definecolor{yellow}{RGB}{200, 120, 60}

\usepackage[plainpages=false,hypertexnames=false,pdfpagelabels]{hyperref}
\definecolor{medium-blue}{rgb}{0,0,.8}
\hypersetup{colorlinks, linkcolor={purple}, citecolor={medium-blue}, urlcolor={medium-blue}}
\newcommand{\arxiv}[1]{\href{http://arxiv.org/abs/#1}{\tt arXiv:\nolinkurl{#1}}}
\newcommand{\arXiv}[1]{\href{http://arxiv.org/abs/#1}{\tt arXiv:\nolinkurl{#1}}}

\newcommand{\doi}[1]{\href{http://dx.doi.org/#1}{{\tt DOI:#1}}}

\DeclareMathOperator{\Ad}{Ad}

\DeclareMathOperator{\aux}{aux}

\DeclareMathOperator{\End}{End}

\DeclareMathOperator{\Hom}{Hom}
\DeclareMathOperator{\id}{id}

\DeclareMathOperator{\rev}{rev}

\DeclareMathOperator{\loc}{loc}
\newcommand{\set}[2]{\left\{#1 \middle| #2\right\}}

\newcommand{\Mod}{\mathsf{Mod}}
\newcommand{\Bim}{\mathsf{Bim}}

\newcommand{\SSS}{\mathsf{SSS}}

\def\semicolon{;}
\def\applytolist#1{
    \expandafter\def\csname multi#1\endcsname##1{
        \def\multiack{##1}\ifx\multiack\semicolon
            \def\next{\relax}
        \else
            \csname #1\endcsname{##1}
            \def\next{\csname multi#1\endcsname}
        \fi
        \next}
    \csname multi#1\endcsname}

\def\calc#1{\expandafter\def\csname c#1\endcsname{{\mathcal #1}}}
\applytolist{calc}QWERTYUIOPLKJHGFDSAZXCVBNM;
\def\bbc#1{\expandafter\def\csname bb#1\endcsname{{\mathbb #1}}}
\applytolist{bbc}QWERTYUIOPLKJHGFDSAZXCVBNM;
\def\bfc#1{\expandafter\def\csname bf#1\endcsname{{\mathbf #1}}}
\applytolist{bfc}QWERTYUIOPLKJHGFDSAZXCVBNM;
\def\sfc#1{\expandafter\def\csname s#1\endcsname{{\sf #1}}}
\applytolist{sfc}QWERTYUIOPLKJHGFDSAZXCVBNM;
\def\fc#1{\expandafter\def\csname f#1\endcsname{{\mathfrak #1}}}
\applytolist{fc}QWERTYUIOPLKJHGFDSAZXCVBNM;
\def\rmc#1{\expandafter\def\csname rm#1\endcsname{{\mathrm #1}}}
\applytolist{rmc}QWERTYUIOPLKJHGFDSAZXCVBNM;

\numberwithin{equation}{section}

\theoremstyle{plain}
\newtheorem{thm}[equation]{Theorem}
\newtheorem*{thm*}{Theorem}
\newtheorem{cor}[equation]{Corollary}
\newtheorem{lem}[equation]{Lemma}
\newtheorem{prop}[equation]{Proposition}

\newtheorem*{claim*}{Claim}

\newtheorem{thmalpha}{Theorem}

\theoremstyle{definition}
\newtheorem{defn}[equation]{Definition}

\newtheorem*{trick*}{Trick}
\newtheorem{construction}[equation]{Construction}
\newtheorem{nota}[equation]{Notation}
\newtheorem{fact}[equation]{Fact}
\newtheorem{facts}[equation]{Facts}

\newtheorem{ex}[equation]{Example}

\newtheorem{rem}[equation]{Remark}

\title{Superselection sectors for posets of von Neumann algebras}
\author{\small{Anupama Bhardwaj, Tristen Brisky, Chian Yeong Chuah, Kyle Kawagoe, Joseph Keslin, David Penneys, and Daniel Wallick}}
\date{\today}

\begin{document}

\begin{abstract}
We study a commutant-closed collection of von Neumann algebras acting on a common Hilbert space indexed by a poset with an order-reversing involution.
We give simple geometric axioms for the poset which allow us to construct a braided tensor category of superselection sectors analogous to the construction of Gabbiani and Fr\"ohlich for conformal nets.
For cones in $\mathbb{R}^2$, we weaken our conditions to a bounded spread version of Haag duality and obtain similar results.
We show that intertwined nets of algebras have isomorphic braided tensor categories of superselection sectors.
Finally, we show that the categories constructed here are equivalent to those constructed by Naaijkens and Ogata for certain 2D quantum spin systems.
\end{abstract}

\maketitle

\section{Introduction}
Operator algebras afford a rigorous method to discuss the thermodynamic limit (as the number of sites goes to $\infty$) of a quantum spin system \cite{MR1441540}.
For example, given a $\bbZ^2$ lattice (viewed as a subset of $\bbR^2$) where each site hosts $\bbC^d$ spins, unless we choose a vector at each site, we cannot take the infinite tensor product of the `local' Hilbert spaces.
However, given an arbitrary region $\Lambda\subset \bbR^2$, we get a well-defined algebra of \emph{local operators} $\fA_\Lambda$,
and the inductive limit over all sites gives the UHF algebra of \emph{quasi-local} operators $\fA$.
$$
\fA_\Lambda:= \overline{\bigotimes_{v\in\Lambda} M_d(\bbC)}^{\|\cdot\|}
\qquad\qquad\qquad\qquad
\fA := \overline{\bigotimes_{v\in \bbZ^2} M_d(\bbC)}^{\|\cdot\|}
$$
The assignment $\Lambda\mapsto \fA_\Lambda$ gives a \emph{net of algebras}, where an inclusion of regions $\Lambda\leq \Delta$ gives an inclusion of algebras $\fA_\Lambda \subseteq \fA_\Delta$
and disjoint regions $\Lambda\subset \Sigma^c$ correspond to commuting algebras of local operators $[\fA_\Lambda, \fA_\Sigma]=0$.
These properties are usually called \emph{isotony} and \emph{locality} respectively.

Recent articles \cite{MR2804555, MR2956822, MR3426207} construct a unitary braided tensor category of superselection sectors from a topologically ordered spin system, 
which is equipped with a pure ground state $\omega$ for a family of interactions
whose cone von Neumann algebras 
$$
\fA_\Lambda '' \subset B(L^2(\fA,\omega))
$$
are properly infinite factors which 
satisfy \emph{Haag duality}: 
$$
\fA_\Lambda'' = \fA_{\Lambda^c}'
\qquad\qquad\qquad
\forall\text{ cones $\Lambda\subset \bbR^2$.}
$$
The article \cite{MR4362722} weakens this last property to a technical condition called \emph{approximate Haag duality}.
Both of these approaches are based on \cite{MR165864,MR0297259,MR660538}.

For conformal nets on $S^1$, the article \cite{MR1231644} gives a construction of a unitary braided tensor category of superselection sectors by looking at families of normal representations of the von Neumann algebras $A_I$ associated to intervals on $S^1$ satisfying various axioms, including isotony, locality, and Haag duality.
One can use this same approach to give an equivalent definition of superselection sectors for a quantum spin system.

The purpose of this article is to extend the construction of \cite{MR1231644} to an extremely general setting to obtain a unitary braided tensor category from a very short list of simple axioms.
Our basic object of study is a family of von Neumann subalgebras $\{A_p\}_{p\in\cP}$ of $B(H)$ for a fixed separable Hilbert space indexed by a poset $\cP$ equipped with an order-reversing involution: $p\leq q$ iff $q'\leq p'$ and $p=p''$ for all $p,q\in\cP$.
We require only 3 axioms for these von Neumann algebras:
\begin{itemize}
\item (isotony) If $p\leq q$, then $A_p\subseteq A_q$.
\item (Haag duality) For all $p\in \cP$, $A_p'=A_{p'}$.
\item (absorbing) The representation $H$ is \emph{absorbing} for every $A_p$, i.e., for any other normal separable representation $K$ of $A_p$,
$$
{}_{A_p}H \oplus {}_{A_p}K \cong {}_{A_p}H.
$$
\end{itemize}
We call such an object a $\cP$-\emph{net} of von Neumann algebras.
In \S\ref{sec:AbsorbingReps} below, we discuss absorbing representations, which also go by the name of \emph{maximal} representations in \cite{MR808930}.
Uniqueness of absorbing/maximal representations is used to \emph{localize} superselection sectors at any $p\in\cP$.
In Corollary \ref{cor:PNetsAreProperlyInfinite} we show that assuming Haag duality, the absorbing axiom is equivalent to the condition that each $A_p$ is properly infinite.

Following \cite{MR1231644}, a \emph{superselection sector} of a $\cP$-net of algebras $A$ is a separable Hilbert space $K$ together with a family of normal representations $\pi_p: A_p\to B(K)$ satisfying 3 compatibility axioms:
\begin{itemize}
\item (isotony) If $p\leq q$, then $\pi_q|_{A_p}=\pi_p$.
\item (locality) If $p\leq q'$, then $[\pi_pA_p,\pi_qA_q]=0$.
\item (absorbing) the representation $K$ is \emph{absorbing} for every $\pi_pA_p$.
\end{itemize}
In \S\ref{sec:SuperselectionSectors} below, we show that the absorbing property is equivalent to the condition that each $\pi_p$ is injective.
It is easily seen that superselection sectors form a $\rmW^*$-category which admits orthogonal direct sums, and we discuss when idempotents split in Proposition \ref{prop:IdempotentsSplit} below.
In particular, each $A_p$ being a factor is a sufficient condition.

In order to construct fusion and braiding, we look at the full subcategory $\SSS_p$ of superselection sectors $(H,\pi)$ \emph{localized} at $p\in\cP$, meaning the space of the sector is $H$ and $\pi_{p'}=\id_{A_{p'}}$.
Under the mild \emph{geometric} assumptions \ref{geom:SelfDisjoint}--\ref{geom:ZigZag} on our poset $\cP$, which we summarize in \S\ref{sec:GeometricAxioms} below, the techniques of \cite{MR260294,MR1231644} allow us to construct a strict fusion product and a unitary braiding.

\begin{thmalpha}
\label{thm:Main}
Given a poset $\cP$ satisfying the geometric axioms \ref{geom:SelfDisjoint}--\ref{geom:ZigZag}
and a $\cP$-net of von Neumann algebras $A$,
the superselection sectors $\SSS_p$ localized at $p$ carries a strict tensor product and a unitary braiding.
\end{thmalpha}

Whenever $p\leq q$, there is a strict inclusion braided tensor equivalence $\SSS_p\hookrightarrow\SSS_q$, and by Lemma \ref{lem:Connected} below, given arbitrary $p,q\in \cP$, we always have a braided equivalence $\SSS_p\cong \SSS_q$.

When $\cP$ is the poset of intervals on $S^1$ and $A$ is a conformal net, this construction is equivalent to the one from \cite{MR1231644}.
However, our construction of the braiding stays local to the interval in which we localize, and we do not need to choose any `point at infinity' as in \cite[\S IV]{MR1231644}.

In order to make connections to \cite{MR2804555, MR2956822, MR3426207,MR4362722}, we then focus on the particular poset $\cC$ of cones in $\bbR^2$.\footnote{The techniques of this section can be applied to more general geometric posets, but we specialize to cones in $\bbR^2$ for simplicity.}
We generalize the above construction to a \emph{bounded spread} $\cC$-net which only satisfies \emph{bounded spread Haag duality} \cite{2304.00068}.
However, in order for locality to be preserved under superselection sectors, we add an additional `small generated' axiom.
\begin{itemize}
    \item (bounded spread Haag duality)
    For every cone $\Lambda$, $A_{\Lambda^c}'\subset A_{\Lambda^{+s}}$
    where $\Lambda^{+s}$ is the cone obtained by enlarging $\Lambda$ by distance $s$.
    \item (small generated) 
    For every cone $\Delta$, $A_\Delta$ is generated by the subalgebras $A_\Lambda$ such that $\Lambda\leq \Delta$ and $\Lambda$ is \emph{$\Delta$-small}, i.e., there are cones $\Gamma,\Sigma$ with $\Lambda,\Delta\leq \Gamma$ and $\Lambda, \Delta^c\leq \Sigma$.
\end{itemize}
We define bounded spread superselection sectors similarly as before, but without a locality axiom as it is implied by the small generated axiom.

\begin{thmalpha}
\label{thm:MainBoundedSpread}
Given a bounded spread $\cC$-net of von Neumann algebras $A$,
the superselection sectors $\SSS_\Lambda$ localized at $\Lambda$ carries a strict tensor product and a unitary braiding.
If moreover $A$ satisfies Haag duality, we get the same braided unitary tensor category as in Theorem \ref{thm:Main}.
\end{thmalpha}

Theorem \ref{thm:MainBoundedSpread} is especially useful when we are given two \emph{intertwined} bounded spread $\cC$-nets of algebras $A,B$, i.e., there is a global \emph{intertwining constant} $t>0$ such that
$$
A_{\Lambda} \subseteq B_{\Lambda^{+t}} \subseteq A_{\Lambda^{+2t}}  
\qquad\qquad\qquad\qquad
\forall\,\Lambda\in\cC.
$$

\begin{thmalpha}
Given two $t$-intertwined bounded spread $\cC$-nets of algebras $A,B$,
we have a strict isomorphism of braided unitary strict tensor categories $\SSS(A)_\Lambda \cong \SSS(B)_{\Lambda^{+t}}$.
\end{thmalpha}

This theorem will be applied to intertwined bounded spread $\cC$-nets of algebras coming from the Levin-Wen model in \cite[\S IV.E]{MR4808260} in future work \cite{2DBraidedSpinSystems}.

Finally, as expected, we prove that the categories we construct in this article are braided equivalent to those constructed in \cite{MR2804555, MR3426207,MR4362722,2306.13762}.

\subsection{Geometric poset axioms for a braided tensor category of superselection sectors}
\label{sec:GeometricAxioms}

Here we summarize the list of \emph{geometric axioms} we assume for our poset $\cP$ in order to use the techniques of \cite{MR260294,MR1231644} to construct from our $\cP$-net of algebras $A$ a strict tensor product and a braiding on the $\rmW^*$ category $\SSS_p$ of superselection sectors localized at a chosen $p \in \cP$.
We illustrate each of these axioms for the poset $\cC$ of cones in $\bbR^2$.

For $p,q\in \cP$, we say $p,q$ are \emph{disjoint} if $p\leq q'$,
and we write
$$
p\Cap q := \set{r\in \cP}{r\leq p\text{ and }r\leq q}.
$$
Our first geometric axioms are as follows.
\begin{enumerate}[label=\textup{(GA\arabic*)}, series=geom]
\setcounter{enumi}{-1}
\item 
\label{geom:SelfDisjoint}
\underline{Disjoint complements:}
For every $p\in\cP$, $p\Cap p'=\emptyset$.

By Remark \ref{rem:SelfDisjointness}, the existence of a $\cP$-net of von Neumann algebras as in Definition \ref{defn:PNetOfAlgebras} forces $\cP$ to satisfy \ref{geom:SelfDisjoint}.

\item 
\label{geom:saturate}
\underline{Saturation:}
Every $p\in\cP$ is \emph{saturated}: for any $q\in\cP$, $p\Cap q\neq \emptyset$ or $p\Cap q'\neq \emptyset$ (possibly both).

Every cone in $\bbR^2$ is saturated.

\item 
\label{geom:splitting}
\underline{Splitting:}
Every $p \in \cP$ \emph{splits}: there are disjoint $r,s \leq p$.

Every cone in $\bbR^2$ splits.
$$
\tikzmath{
\clip (-2,-.1) rectangle (2,2);
\draw (135:3cm) -- (0,0) node[above]{$\scriptstyle p$} -- (45:3cm);
\coordinate (a) at (-.1,.3);
\fill[fill=blue!25] ($ (135:2.5cm) + (a) $) -- (a) -- ($ (90:4cm) + (a) $);
\draw ($ (135:2.5cm) + (a) $) -- (a) node[above, xshift=-.2cm,yshift=.2cm]{$\scriptstyle r$} -- ($ (90:4cm) + (a) $);
\coordinate (b) at (.1,.3);
\fill[fill=red!25] ($ (45:2.5cm) + (b) $) -- (b) -- ($ (90:4cm) + (b) $);
\draw ($ (45:2.5cm) + (b) $) -- (b) node[above, xshift=.2cm,yshift=.2cm]{$\scriptstyle s$} -- ($ (90:4cm) + (b) $);
}
$$
We use \ref{geom:saturate} and \ref{geom:splitting} to 
propose a definition of $\circ_p$.

\end{enumerate}

A \emph{zig-zag} between $p,q\in\cP$ is a sequence
$
(z_1,y_1,\dots, z_n,y_n,z_{n+1})\subset\cP
$
with $z_1=p$ and $z_{n+1}=q$
such that
$z_j,z_{j+1}\leq y_j$ for all $j$.

\begin{lem}
\label{lem:Connected}
The axioms \ref{geom:saturate} and \ref{geom:splitting} imply that the poset $\cP$ is \emph{connected}: for any $p,q\in\cP$, there is a zig-zag from $p$ to $q$.
\end{lem}
\begin{proof}
\item[\underline{Step 1:}]
For an arbitrary $p\in\cP$, there is a zig-zag from $p$ to $p'$.
Indeed, choose a splitting $p_1,p_2\leq p$ with $p_1\leq p_2'$, and note that
$
p \leq p \geq p_1 \leq p_2' \geq p'
$
is the desired zig-zag.

\item[\underline{Step 2:}]
If $p\Cap q\neq \emptyset$, then choosing $r\in \cP$ with $r\leq p$ and $r\leq q$, we have $p\leq p \geq r \leq q \geq q$.

\item[\underline{Step 3:}]
If $p\Cap q'\neq \emptyset$, then we can compose a zig-zags from $p$ to $q'$ from Step 2 with the zig-zag from $q'$ to $q$ from Step 1.
\end{proof}

The axioms \ref{geom:saturate} and \ref{geom:splitting} are equivalent to the single axiom \ref{geom:qSmallqIndicator} stated in the lemma below which is how we use \ref{geom:saturate} and \ref{geom:splitting} in practice.
To set up this equivalent axiom, for $p,q\in \cP$, we define the following notions.
\begin{itemize}
\item 
We say $p$ is \emph{$q$-small} if there are $r,s\in\cP$ such that $p,q\leq r$ and $p,q'\leq s$.
(Observe $q$-small is identical to $q'$-small.)
\item
A $\widetilde{p}\leq p$ is called a \emph{$q$-indicator} if $\widetilde{p}\leq q$ or $\widetilde{p}\leq q'$.
\end{itemize}

\begin{lem}
The axioms \ref{geom:saturate} and \ref{geom:splitting} are equivalent to the following axiom.
\begin{enumerate}[label=\textup{(GA1.5)}, series=geomconsequence]
\item
\label{geom:qSmallqIndicator}
\textup{\underline{Small indicators:}
For any $p,q\in\cP$, there is a $q$-small $q$-indicator $\widetilde{p}\leq p$.
}
\end{enumerate}
\end{lem}
\begin{proof}
Suppose $\cP$ satisfies \ref{geom:saturate} and \ref{geom:splitting}.
If $p\Cap q\neq \emptyset$, choose $a\in \cP$ with $a\leq p$ and $a \leq q$.
Otherwise, by \ref{geom:saturate}, $p\Cap q'\neq \emptyset$, so we may choose $a\in \cP$ with $a\leq p$ and $a \leq q'$.
In either case, since $a$ splits by \ref{geom:splitting}, there is a $q$-small $\widetilde{p}\leq a$.   

Conversely, the existence of a $q$-indicator immediately implies \ref{geom:saturate}.
If $\widetilde{p}\leq p$ is a $p'$-small $p'$-indicator, then
the existence of an $r$ containing $\widetilde{p}$ and $p'$ yields a splitting $p_1=\widetilde{p}$ and $p_2=r'$ of $p$, establishing \ref{geom:splitting}.
\end{proof}
Here is a cartoon of cones $p,q$ in $\bbR^2$ with two choices of $\widetilde{p},\widehat{p}$ of $q$-small $q$-indicator.
$$
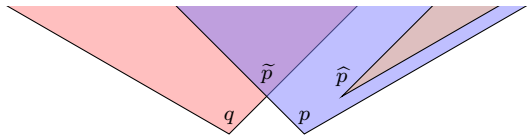
\tikzmath{
\clip (-3,-.05) rectangle (4.3,1.7);
\fill[fill=red!50, opacity=.5] (150:4cm) -- (0,0) -- (45:2.5cm);
\draw (150:4cm) -- (0,0) node[above]{$\scriptstyle q$} -- (45:2.5cm);
\coordinate (b) at (1,0);
\fill[fill=blue!50, opacity=.5] ($ (135:2.5cm) + (b) $) -- (b) -- ($ (30:4cm) + (b) $);
\draw ($ (135:2.5cm) + (b) $) -- (b) node[above]{$\scriptstyle p$} -- ($ (30:4cm) + (b) $);
\coordinate (ptt) at (1.5,.5);
\fill[fill=orange!50, opacity=.5] ($ (45:2.5cm) + (ptt) $) -- (ptt) -- ($ (30:2.5cm) + (ptt) $);
\draw ($ (45:2.5cm) + (ptt) $) -- (ptt) node[above]{$\scriptstyle \widehat{p}$} -- ($ (30:2.5cm) + (ptt) $);
\node at (.5,.8) {$\scriptstyle \widetilde{p}$};
}
$$

\begin{rem}
\label{rem:SmallerqSmallqIndicators}
Note that if $p, q \in \cP$ and $\widetilde{p} \leq p$ is a $q$-small $q$-indicator, then any $\widehat{p} \leq \widetilde{p}$ is also a $q$-small $q$-indicator.  
Indeed, since $\widetilde{p} \leq q$ or $\widetilde{p} \leq q'$, $\widehat{p} \leq q$ or $\widehat{p} \leq q'$, so $\widehat{p}$ is also a $q$-indicator.
Similarly, since $\widetilde{p}$ is $q$-small, we have that there are $r, s \in \cP$ such that $\widetilde{p},q\leq r$ and $\widetilde{p},q'\leq s$.
Since $\widehat{p} \leq \widetilde{p}$, $\widehat{p},q\leq r$ and $\widehat{p},q'\leq s$, so $\widehat{p}$ is $q$-small.
\end{rem}

We require an additional geometric axiom about the existence of zig-zags in order to complete the construction of our fusion product.

\begin{enumerate}[resume*=geom]
\item
\label{geom:ZigZag}
\underline{Zig-zag:}
For every $p,q\in\cP$ and any two choices of $q$-small $q$-indicator $\widetilde{p},\widehat{p}\leq p$ as in \ref{geom:qSmallqIndicator},
there is a zig-zag between $\widetilde{p},\widehat{p}$ contained in $p$ such that each $z_j$ is a $q$-indicator and each $y_j$ is $q$-small.
(Observe this immediately implies each $z_j$ is also $q$-small.)

Here is a cartoon of a zig-zag of cones in $\bbR^2$ contained in $p$ and disjoint from $q$.
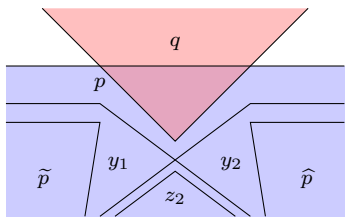
\begin{equation}
\label{eq:ZigZagOfCones}
\tikzmath{
\fill[fill=blue!20] (-2.25,-1) rectangle (2.25,1);
\fill[fill=red!50, opacity=.5] (135:2.5cm) -- (0,0) -- (45:2.5cm);
\draw (135:2.5cm) -- (0,0) -- (45:2.5cm);
\draw (-2.25,1) -- (2.25,1);
\draw (-2.25,.25) -- (-1,.25) -- (-1.2,-1);
\draw (2.25,.25) -- (1,.25) -- (1.2,-1);
\draw (-2.25,.5) -- (-1,.5) -- (1,-1);
\draw (2.25,.5) -- (1,.5) -- (-1,-1);
\draw (.8,-1) -- (0,-.4) -- (-.8,-1);
\node at (-1.75,-.5) {$\scriptstyle \widetilde{p}$};
\node at (1.75,-.5) {$\scriptstyle \widehat{p}$};
\node at (-.75,-.3) {$\scriptstyle y_1$};
\node at (.75,-.3) {$\scriptstyle y_2$};
\node at (0,-.75) {$\scriptstyle z_2$};
\node at (0,1.3) {$\scriptstyle q$};
\node at (-1,.75) {$\scriptstyle p$};
}
\qquad\qquad
(z_1=\widetilde{p},y_1,z_2,y_2,z_3=\widehat{p})
\end{equation}
We use \ref{geom:ZigZag} along with \ref{geom:saturate} and \ref{geom:splitting} to show that $\circ_p$ is a well-defined and strictly associative, 
making $\SSS_p$ a strict $\rmW^*$ tensor category.
Moreover, if $p\leq q$, the `inclusion' $\SSS_p\hookrightarrow \SSS_q$ is a strict tensor equivalence.
By connectivity of $\cP$, for arbitrary $p,q\in\cP$, there is a zig-zag of strict tensor equivalences $\SSS_p\cong \SSS_q$.
\end{enumerate}

The axioms \ref{geom:SelfDisjoint}--\ref{geom:ZigZag} hold for a wide variety of posets including
intervals in $S^1$, cones in $\bbR^2$, disks in $S^2$, cones in $\bbR^3$, and more.
We refer the reader to Appendices \ref{appendix:IntervalsSatisfyGAs} and \ref{appendix:ConesSatisfyGAs} for proofs.

The geometric axioms above are also sufficient to construct a unitary braiding on the strict $\rmW^*$-tensor category $(\SSS_p,\circ_p)$.
To help with the construction, we prove that 
our geometric axioms \ref{geom:SelfDisjoint}--\ref{geom:ZigZag}
imply two other helpful geometric axioms
in
Theorem \ref{thm:MutuallyDisjointZigZagsExist}
and
Proposition \ref{prop:AdmitsReflection} in the sequel.
The axiom \ref{geom:Braid} below is inspired by \cite[Lem.~2.2]{MR260294}.

\begin{enumerate}[resume*=geom]
\item
\label{geom:MutuallyDisjointZigZag}
\underline{Mutually disjoint zig-zag:}
Let $r_1, s_1, r_2, s_2 \in \cP$ with $r_1 \leq s_1'$ and $r_2 \leq s_2'$.  
A \emph{mutually disjoint zig-zag} $(r_1,s_1)\leftrightsquigarrow(r_2,s_2)$ is a pair of zig-zags 
\begin{align*}
&(r_1=x_1, 
\textcolor{blue}{w_1}, \textcolor{red}{x_2},
\dots, x_n, \textcolor{blue}{w_n}, \textcolor{red}{x_{n + 1}=r_2})
\\
&(\textcolor{blue}{s_1=z_1}, \textcolor{red}{y_1}, z_2, \dots, \textcolor{blue}{z_n}, \textcolor{red}{y_n}, z_{n + 1}=s_2)    
\end{align*}
\
such that 
for all $i = 1, \dots, n$, $\textcolor{blue}{w_i}$ is disjoint from $\textcolor{blue}{z_i}$ and $\textcolor{red}{y_i}$ is disjoint from $\textcolor{red}{x_{i + 1}}$.

In Theorem \ref{thm:MutuallyDisjointZigZagsExist}, we prove that the axioms \ref{geom:SelfDisjoint}--\ref{geom:ZigZag}
imply that
for any two splittings $(r_1,s_1), (r_2,s_2)$ of $p$, there is a mutually disjoint zig-zag
$(r_1,s_1)\leftrightsquigarrow(r_2,s_2)$ or $(r_1,s_1)\leftrightsquigarrow(s_2,r_2)$ contained in $p$.

We use \ref{geom:splitting} to propose a definition for the braiding on $\SSS_p$, and we use \ref{geom:MutuallyDisjointZigZag} to show that this definition is independent of the choice of splitting, up to getting the reverse braiding (see Lemma \ref{lem:ReverseBraiding} below).

\item 
\label{geom:Braid}
\underline{Reflection:}
Given a  splitting $(r,s)$ of $p$, a \emph{reflection} is a splitting $(a,b)$ of $p'$
and some $c$ such that $a,r\leq c$ and $b,s\leq c'$. 

In Proposition \ref{prop:AdmitsReflection}, we prove that the axioms \ref{geom:SelfDisjoint}--\ref{geom:ZigZag}
imply that for any $p\in\cP$, 
there is a splitting that admits a reflection.

Every splitting of a cone in $\bbR^2$ admits a reflection.
\begin{equation}
\label{eq:Reflection}
\tikzmath{
\draw (-1.6,0) -- (1.6,0);
\draw (0,-1.6) -- (0,1.6);
\fill[blue!25] (-1.2,.2) rectangle (-.2,1.2);
\draw (-1.2,.2) -- (-.2,.2) -- (-.2,1.2);
\node at (-.4,.4) {$\scriptstyle r$};
\fill[red!25] (1.2,.2) rectangle (.2,1.2);
\draw (1.2,.2) -- (.2,.2) -- (.2,1.2);
\node at (.4,.4) {$\scriptstyle s$};
\fill[green!25] (-1.2,-.2) rectangle (-.2,-1.2);
\draw (-1.2,-.2) -- (-.2,-.2) -- (-.2,-1.2);
\node at (-.4,-.4) {$\scriptstyle a$};
\fill[orange!25] (1.2,-.2) rectangle (.2,-1.2);
\draw (1.2,-.2) -- (.2,-.2) -- (.2,-1.2);
\node at (.4,-.4) {$\scriptstyle b$};
\node at (-1.4,.2) {$\scriptstyle p$};
\node at (-1.4,-.2) {$\scriptstyle p'$};
\node at (1.4,.2) {$\scriptstyle p$};
\node at (1.4,-.2) {$\scriptstyle p'$};
\node at (.2,-1.4) {$\scriptstyle c'$};
\node at (-.2,-1.4) {$\scriptstyle c$};
\node at (.2,1.4) {$\scriptstyle c'$};
\node at (-.2,1.4) {$\scriptstyle c$};
}
\end{equation}

We use \ref{geom:Braid} to show the proposed definition of a braiding for $\SSS_p$ is well-defined and satisfies the braid relations. 
Moreover, the equivalences $\SSS_p\cong \SSS_q$ for arbitrary $p,q\in\cP$ constructed above are all braided equivalences.

\end{enumerate}

\begin{rem}
\label{rem:ReverseBraidingIntro}
If $(r,s)$ is a splitting of $p$ and there is a mutually disjoint zig-zag $(r,s)\leftrightsquigarrow(s,r)$ contained in $p$, then the braiding is symmetric by Lemma \ref{lem:ReverseBraiding} below (see Remark \ref{rem:ReverseBraiding}).
Such a mutually disjoint zig-zag can be built in the poset of disks in $S^2$ from 
$r_1 ,r_2, r_3 ,s_1,s_2,s_3 \leq p$
satisfying $r_i\leq s_i'$ for all $i$ and 
$r_i,r_j \leq s_k$ for $i,j,k$ distinct.
\begin{equation}
\label{eq:SmallMutuallyDisjointZigZagSwap}
\tikzmath{
\draw[fill=red!50, opacity=.5] (30:1) circle (.76cm);
\draw[fill=blue!50, opacity=.5] ($ (1,0) + (150:1) $) circle (.76cm);
\draw[fill=green!50, opacity=.5] ($ (60:1) + (270:1) $) circle (.76cm);
\draw[fill=white] (0,0) node{$\scriptstyle r_1$} circle (.2cm);
\draw[fill=white] (1,0) node{$\scriptstyle r_2$} circle (.2cm);
\draw[fill=white] (60:1) node{$\scriptstyle r_3$} circle (.2cm);
\draw (30:.57735) circle (1.25cm);
\node at ($ (30:1) + (30:.4) $) {$\scriptstyle s_1$};
\node at ($ (1,0) + (150:1) + (150:.4) $) {$\scriptstyle s_2$};
\node at ($ (60:1) + (270:1) + (270:.4) $) {$\scriptstyle s_3$};
\node at ($ (60:1) + (0,.5) $) {$\scriptstyle p$};
}
\qquad\qquad
\begin{aligned}
&(r_1, 
\textcolor{blue}{s_2}, \textcolor{red}{r_3},\textcolor{blue}{s_1}, \textcolor{red}{r_2})
\\
&(\textcolor{blue}{r_2}, \textcolor{red}{s_3}, \textcolor{blue}{r_1}, \textcolor{red}{r_1}, r_1)    
\end{aligned}
\end{equation}
We can also build such a mutually disjoint zig-zag for cones in $\bbR^3$.
\end{rem}

\subsection*{Acknowledgements}
The content of \S\ref{sec:Fusion}-\ref{sec:braiding} and the appendix of this article was the Summer 2024 undergraduate research project of Tristen Brisky and Joseph Keslin,
who were supervised by Anupama Bhardwaj, Chian Yeong Chuah, Kyle Kawagoe, David Penneys, and Daniel Wallick during Summer 2024,
supported by Penneys’ NSF grant DMS 2154389.
Additionally, this work was supported by NSF DMS 1928930 while David Penneys was in residence at the Mathematical Sciences Research Institute/SLMath in Berkeley, California, during Summer 2024.
The authors would like to thank
Andr\'e Henriques, 
Brett Hungar,
Peter Huston,
Corey Jones, 
Pieter Naaijkens, 
Sean Sanford, and
Siddharth Vadnerkar
for helpful conversations.

\section{Nets of algebras and their superselection sectors}

Our main object of study is a poset of von Neumann subalgebras of $B(H)$. 
The main goal is to construct a braided tensor category of \emph{superselection sectors}, which can be viewed as the representation theory of our poset of von Neumann algebras.
To this end, we begin with a general poset, and to define fusion and braiding for our superselection sectors, we will us the \emph{geometric axioms} from \S\ref{sec:GeometricAxioms} above.

For ease of exposition, all Hilbert spaces in this article are assumed to be separable (and infinite dimensional), and all von Neumann algebras are assumed to have separable preduals.

\begin{nota}
For this section, $\cP$ denotes a poset equipped with an order-reversing involution denoted $p\mapsto p'$, i.e., $p\leq q$ implies $q'\leq p'$ and $p''=p$.
For $p,q\in\cP$, we say that $p$ and $q$ are \emph{disjoint} if $p \leq q'$ (equivalently, $q \leq p'$).
We write $p\Cap q:=\set{r\in\cP}{r\leq p\text{ and }r\leq q}$.
\end{nota}

\begin{defn}
\label{defn:PNetOfAlgebras}
A $\cP$-\emph{net} of algebras is a 
Hilbert space $H$ together with an assignment of a von Neumann algebra $A_p\subseteq B(H)$ to each $p\in \cP$ such that 
\begin{itemize}
\item (isotone)
$p\leq q$ implies $A_p\subseteq A_q$, \item (Haag duality)
$A_{p'}=A_p'$ for all $p\in\cP$, and
\item (absorbing) 
for each $p\in\cP$, the representation $H$ for $A_p$ is \emph{absorbing}, defined in \S\ref{sec:AbsorbingReps} below.
\end{itemize}
\end{defn}

While these are a minimal set of axioms for the nets of algebras we consider in this article, we have the following generating property which is useful when discussing locality of superselection sectors (see Lemma \ref{lem:SmallGenerationImpliesIsoneRepsAreLocal} below).

\begin{defn}
We say a $\cP$-net of algebras has the \emph{small generating property} or is \emph{small generated} if for every $q\in\cP$, $A_q$ is generated by 
$\set{A_p}{\text{$p\leq q$ and $p$ is $q$-small}}$.
\end{defn}

\begin{ex}
The small generating property is satisfied by the net of cone algebras coming from a quantum spin system \cite{MR2804555, MR3426207, MR4362722,2306.13762}. 
More specifically, we let $\cC$ denote poset of cones in $\bbR^2$, and we let $\Gamma$ be an infinite planar lattice where each site $v$ hosts $\bbC^{d_v}$ spins.
The \emph{quasilocal algebra} is the AF $\rmC^*$-algebra $\fA \coloneqq \bigotimes_{v \in \Gamma} M_{d_v}(\bbC)$, and for an arbitrary subset $\Lambda \subseteq \Gamma$, we define $\fA_\Lambda \coloneqq \bigotimes_{v \in \Lambda} M_{d_v}(\bbC)$.
For a representation $\pi \colon \fA \to B(H)$ and a cone $\Lambda \in \cC$, we define $\cR_\Lambda \coloneqq \pi(\fA_\Lambda)''$.
(In this context, the notation $\cR_\Lambda$ is most commonly used in the literature.)
For certain representations $\pi$, $\{\cR_\Lambda\}_{\Lambda \in \cC}$ satisfies the axioms of a $\cC$-net of algebras \cite{MR2804555, MR2956822, MR3426207}.  
Furthermore, $\{\cR_\Lambda\}_{\Lambda \in \cC}$ always satisfies the small generating property.  
Indeed, 
for every $v \in \Delta$, there exists $\Lambda \subseteq \Delta$ such that $\Lambda$ is $\Delta$-small and $v \in \Lambda$.
\end{ex}

\subsection{Absorbing representations of von Neumann algebras}
\label{sec:AbsorbingReps}

The notion of an absorbing representation is included in the definition of a $\cP$-net of algebras in order to allow for nets of arbitrary von Neumann algebras, not just nets of properly infinite factors which appear in the AQFT literature.

\begin{defn}
Given a von Neumann algebra $M$ (with separable predual), $\Mod(M)$ denotes the category of separable $M$-modules (normal representations of $M$), i.e., separable Hilbert spaces equipped with normal actions of $M$.
An $M$-module ${}_MH\in \Mod(M)$ is called:
\begin{itemize}
\item
a \emph{generator} if for any non-zero ${}_MK\in\Mod(M)$, $\Hom({}_MH\to {}_MK)\neq 0$,
\item 
\emph{absorbing} if for any other ${}_MK\in\Mod(M)$,
${}_MH\oplus{}_MK \cong {}_MH$, i.e., there is an $M$-linear unitary $u:H\oplus K \to H$ satisfying $u(m\cdot \xi)=m\cdot u\xi$ for all $\xi\in H\oplus K$.
\item 
\emph{self-absorbing} if ${}_MH\oplus{}_MH \cong {}_MH$, and
\item
\emph{maximal} \cite[p100]{MR808930} if for any other normal representation ${}_MK$, there is an $M$-linear isometry $v: K\to H$. 
\end{itemize}
\end{defn}

\begin{rem}
It is well-known that an $M$-module ${}_MH$ is a generator for $\Mod(M)$ if and only if it is \emph{faithful}, i.e., the map $M\to B(H)$ is injective.
For the forward direction, it is clear that generators must be faithful.
Conversely, if ${}_MH$ is faithful, then the central support of $z(p_H)$ in $\End_M(H\oplus K)$ is $1$ for any non-zero $M$-module ${}_MK$.
Thus
$$
\Hom({}_MH \to {}_MK)
=
p_K\End_M(H\oplus K)p_H
\neq 0.
$$
\end{rem}

Observe that an absorbing representation is automatically unique up to unitary isomorphism, as if ${}_MH,{}_MK$ are both absorbing, then
$$
{}_MH \cong {}_MH\oplus {}_MK \cong {}_MK\oplus {}_MH \cong {}_MK.
$$

\begin{lem}
\label{lem:CantorSchroederBernsteinForW*}
If $\cC$ is a $\rmW^*$-category which admits orthogonal direct sums and we have isometries $u:a\to b$ and $v:b\to a$ in $\cC$, then $a\cong b$.
\end{lem}
\begin{proof}
Observe that $p:=\id_a$ and $q:=\id_b$ are projections in the von Neumann algebra $\End_\cC(a\oplus b)$.
The existence of the isometry $u$ implies $p\preceq q$ and the existence of the isometry $v$ implies $q\preceq p$.
We conclude $p\approx q$ in $\End_\cC(a\oplus b)$, i.e., there is a $w\in \End_\cC(a\oplus b)$ such that $ww^\dag = q$ and $w^\dag w=p$.
Then $w=qwp$ is our desired unitary isomorphism $a\to b$.
\end{proof}

\begin{cor}
A maximal representation is unique when it exists.
\end{cor}

\begin{lem}
\label{lem:AbsEquiv}
The following are equivalent for a normal representation ${}_MH$.
\begin{enumerate}
\item ${}_MH$ is a self-absorbing generator.
\item ${}_MH$ is maximal.
\item ${}_MH$ is absorbing.
\end{enumerate}
\end{lem}
\begin{proof}
\item[\underline{(1)$\Rightarrow$(2):}]
This follows from \cite[Prop.~6.5 and Cor.~7.16]{MR808930}.
Indeed, self-absorbing implies ${}_MH\cong {}_MH\otimes \ell^2$, which is a generator with infinite multiplicity and is thus maximal.
\item[\underline{(2)$\Rightarrow$(1):}]
If ${}_MH$ is maximal, then there are $M$-linear isometries
$$
{}_MH\oplus {}_MH
\hookrightarrow 
{}_MH 
\hookrightarrow 
{}_MH\oplus {}_MH,
$$
and thus ${}_MH$ is self-absorbing by Lemma \ref{lem:CantorSchroederBernsteinForW*}.
Moreover, for any non-zero ${}_MK$, there is an $M$-linear isometry ${}_MK\hookrightarrow {}_MH$ whose adjoint is non-zero.
Thus ${}_MH$ is a generator.
\item[\underline{(2)$\Rightarrow$(3):}]
If ${}_MH$ is maximal, then for any ${}_MK\in\Mod(M)$, there are $M$-linear isometries
$$
{}_MH\oplus {}_MK
\hookrightarrow 
{}_MH 
\hookrightarrow 
{}_MH\oplus {}_MK,
$$
and thus ${}_MH$ is absorbing by Lemma \ref{lem:CantorSchroederBernsteinForW*}.
\item[\underline{(3)$\Rightarrow$(2):}]
If ${}_MH$ is absorbing, then for any ${}_MK\in\Mod(M)$, we have $M$-linear isometries
$$
{}_MK 
\hookrightarrow 
{}_MH\oplus {}_MK 
\cong
{}_MH
$$
and thus ${}_MH$ is maximal.
\end{proof}

\begin{prop}
The representation $L^2M\otimes \ell^2$ is absorbing for $M$.
\end{prop}
\begin{proof}
Observe that $L^2M$ is a generator for $\Mod(M)$ as 
$\Hom({}_ML^2M\to {}_MH)$ is dense in $H$ by \cite[Lem.~2]{MR561983}.
(Indeed, $L^2M \cong L^2(M,\phi)$ for any faithful normal state $\phi$.)
Since $L^2M\otimes \ell^2$ is a self-absorbing generator, it is absorbing by Lemma \ref{lem:AbsEquiv}. 
\end{proof}

The following lemma characterizes absorbing representations of von Neumann algebras.

\begin{lem}
\label{lem:absorbing}
A normal representation ${}_MH$ is absorbing if and only if 
it is faithful and $M'\cap B(H)$ is properly infinite.
\end{lem}
\begin{proof}
First, note that $L^2M\otimes \ell^2$ is faithful and $M' \cong JMJ\otimes B(\ell^2)$ is properly infinite.
Conversely, if $M'=\End({}_MH)$ is properly infinite, then ${}_MH\cong {}_MH\otimes \ell^2$ by \cite[Prop.~6.5]{MR808930} and is thus self-absorbing.
Since an $M$-module is a generator if and only if it is faithful, the result follows by Lemma \ref{lem:AbsEquiv}.
\end{proof}

\begin{cor}
\label{cor:PNetsAreProperlyInfinite}
In the definition of a $\cP$-net of algebras, the absorbing property may be equivalently replaced with the property that all $A_p$ are properly infinite.
\end{cor}
\begin{proof}
Haag duality and Lemma \ref{lem:absorbing} imply that every von Neumann algebra $A_p$ in a $\cP$-net of algebras is properly infinite.
Conversely, if each $A_p$ is properly infinite, then $A_p'=A_{p'}$ is properly infinite, so ${}_{A_p}H$ is absorbing by Lemma \ref{lem:absorbing}.
\end{proof}

\begin{rem}
\label{rem:SelfDisjointness}
It follows immediately from Corollary \ref{cor:PNetsAreProperlyInfinite} that if the poset $\cP$ hosts a net of algebras, 
then for all $p\in\cP$, $p\Cap p'=\emptyset$.
Indeed any von Neumann subalgebra of $A_p\cap A_{p'} = Z(A_p)$ is necessarily finite.
\end{rem}

One utility of absorbing representations is the following fact.

\begin{fact}
\label{fact:AbsorbingGivesEndomorphism}
Suppose ${}_MH_N$ and ${}_MK_N$ are $M-N$ bimodules which are absorbing for $N$ such that $H$ is invertible (i.e., $M$ and $N$ are each other's commutants in $B(H)$).
Let $u: H\to K$ be a right $N$-linear unitary, which exists as $H_N,K_N$ are both absorbing.
Observe that for all $m\in M$ and $n\in N$,
$$
u^*\pi_K(m)un 
=  
u^*\pi_K(m)\pi_K(n)u
= 
u^*\pi_K(n)\pi_K(m)u 
= 
n u^*\pi_K(m)u.
$$
Hence $\rho:M\to M$ given by $\rho(m):=u^*\pi_K(m)u \in N'=M$ is a normal unital endomorphism of $M$ such that ${}_MK_N\cong {}_{\rho(M)}H_N$.
\end{fact}

Recall that for a properly infinite von Neumann algebra $M$, the $\rmW^*$-tensor category $\Bim(M)$ of bimodules of $M$ is equivalent to the projection completion of the $\rmW^*$-tensor category $\End(M)$ of normal unital endomorphisms of $M$\footnote{When $M$ is properly infinite, the $\rmW^*$-tensor category $\End(M)$ admits orthogonal direct sums by \cite[Halving Lemma 6.3.3]{MR1468230}.}
via the map 
$$
\End(M) \ni \rho 
\longmapsto
{}_\rho L^2M
\in \Bim(M)
$$
where on the right hand side, the left action is twisted by $\rho$.

\begin{defn}
We call $\rho\in\End(M)$ \emph{absorbing} if the left $M$-module ${}_\rho L^2M$ is absorbing.
\end{defn}

\begin{lem}
\label{lem:EndoInjectiveIffAbs}
An endomorphism of a properly infinite von Neumann algebra is absorbing if and only if it is injective.
\end{lem}
\begin{proof}
Clearly absorbing endomorphisms are always injective by Lemma \ref{lem:absorbing}.

Conversely, if $\rho$ is injective, ${}_{\rho(M)}L^2M$ is a generator.
Since $M$ is properly infinite, ${}_ML^2M \cong {}_ML^2M \otimes \ell^2$, and thus ${}_{\rho(M)}L^2M \cong {}_{\rho(M)}L^2M\otimes \ell^2$.\footnote{We thank Andr\'e Henriques for pointing out this simple, elegant fact.}
Thus ${}_{\rho(M)}L^2M$ is self-absorbing.
We conclude ${}_{\rho(M)}L^2M$ is absorbing by Lemma \ref{lem:AbsEquiv}.
\end{proof}

\begin{cor}
Suppose $M$ is a properly infinite von Neumann algebra.
The composite of absorbing endomorphisms of $M$ is absorbing.
\end{cor}
\begin{proof}
Using the equivalence in Lemma \ref{lem:EndoInjectiveIffAbs}, we simply note that the composite of injective endomorphisms is injective.
\end{proof}

The next lemma is known to experts.
We include a proof for completeness and convenience of the reader.

\begin{lem}
\label{lem:ProperlyInfiniteInclusion}
If $N\subset M$ is an inclusion of von Neumann algebras with $N$ properly infinite, then $M$ is properly infinite.
\end{lem}
\begin{proof}
Fix a non-zero projection $z\in P(Z(M))$.
Since $N$ is properly infinite, by  \cite[Halving Lemma 6.3.3]{MR1468230}, there are projections $p,q\in P(N)$ with $p+q=1$ and $p\approx 1\approx q$.
Then $pz\approx z\approx qz$, so they are all non-zero.
Since $pz\leq z$ and $z-pz=qz\neq 0$, we see that $z$ is infinite.
\end{proof}

\subsection{Superselection sectors}
\label{sec:SuperselectionSectors}

Suppose $\cP$ is a poset equipped with an order-reversing involution $p\mapsto p'$.

\begin{defn}
A \emph{superselection sector} of a $\cP$-net $A$ is a Hilbert space $K$ together with a family of normal representations $\pi_p:A_p\to B(K)$ satisfying the \emph{superselection criteria}: 
\begin{itemize}
\item (isotone) $p\leq q$ implies $\pi_q|_{A_p}=\pi_p$
\item (local) $p\leq q'$ implies $[\pi_pA_p, \pi_qA_q]=0$
\item (absorbing) $K$ is absorbing for every $A_p$, or equivalently, each $\pi_p$ is faithful.\footnote{ 
Clearly $K$ absorbing for $\pi_p(A_p)$ implies each $\pi_p$ is injective by Lemma \ref{lem:absorbing}.
Conversely, since each $A_p$ is properly infinite by Corollary \ref{cor:PNetsAreProperlyInfinite}, if each $\pi_p$ is injective, then $\pi_{p'}(A_{p'})\subset \pi_p(A_p)'$ forces $\pi_p(A_p)'$ to be properly infinite by Lemma \ref{lem:ProperlyInfiniteInclusion}.
We conclude by Lemma \ref{lem:absorbing} that $K$ is absorbing for $A_p$.
}
\end{itemize}
A \emph{morphism} from the superselection sector $(K,\pi)$ to the superselection sector $(L,\sigma)$ is a bounded operator $T: K\to L$ called an \emph{intertwiner} such that
$T\pi_p(x)\xi=\sigma_p(x)T\xi$
for all $\xi\in K$, $x\in A_p$, and $p\in\cP$.
Two superselection sectors are \emph{equivalent} if there is a unitary intertwiner between them.
\end{defn}

Superselection sectors form a $\rmW^*$-category $\SSS$ in the sense of \cite{MR808930}.
Clearly $\SSS$ admits orthogonal direct sums.

\begin{prop}
\label{prop:IdempotentsSplit}
Suppose $(K,\pi)$ is a superselection sector of the $\cP$-net $A$ and $P\in \End(K,\pi)$ is an intertwining projection.
If the central support $z_p(P)\in Z(A_p)$ is 1 for every $p\in\cP$, then $P$ splits orthogonally.
\end{prop}
\begin{proof}
Observe that the maps $\sigma_p: A_p\to B(PK)$ given by $\sigma_p(x):=\pi_p(x)P$ satisfy isotony and locality as $P$ commutes with $\pi_p(A_p)$ for all $p$.
The condition $z_p(P)=1$ is exactly the condition that 
multiplication by $P$ is injective on $A_p'=A_{p'}$.
Hence $\sigma_{p}$ is injective for every $p$, and thus $(PK,\sigma)\in \SSS$.
It is clear that the inclusion isometry $v: PK\hookrightarrow K$ is an intertwiner such that $vv^\dag = P$, so $(PK,\sigma)$ splits $(K,\pi)$.
\end{proof}

We have the following immediate corollary.

\begin{cor}
If each $A_p$ is a factor, then $\SSS$ is unitarily Cauchy complete.
\end{cor}

\begin{rem}
Assuming $(K,\pi)$ is isotone, locality is equivalent to 
\begin{equation}
\label{eq:LocalRepCondition}
[\pi_qA_q, \pi_{q'}A_{q'}]=0
\qquad\qquad
\forall\,q\in\cP.
\end{equation}
\end{rem}

Locality is usually not listed as an axiom for superselection sectors of conformal 
nets essentially because conformal nets satisfy the small generating property. 
We provide a proof in our setup below.

\begin{lem}
\label{lem:SmallGenerationImpliesIsoneRepsAreLocal}
Suppose $A=\{A_p\}_{p\in\cP}$ is a collection of von Neumann subalgebras of $B(H)$ which satisfies isotony, locality, and the small generation property.
Any isotone family of normal representations $\pi_p: A_p\to B(K)$ is automatically local.
\end{lem}
\begin{proof}
Suppose $p\leq q$ is $q$-small.
Let $r\in\cP$ with $p,q'\leq r$.
By isotony of $\pi$ and isotony and locality of $A$,
for all $x\in A_p$ and $y\in A_{q'}$,
\[
\pi_p(x)\pi_{q'}(y)
\underset{\text{(isotony)}}{=}
\pi_r(x)\pi_{r}(y)
=
\pi_r(xy)
\underset{\text{(locality of $A$)}}{=}
\pi_r(yx)
=
\pi_r(y)\pi_{r}(x)
\underset{\text{(isotony)}}{=}
\pi_{q'}(y)\pi_{p}(x).
\]
We then see that
$$
\bigcup_{q\text{-small }p\leq q} \pi_pA_p 
\underset{\text{(isotony)}}{=} 
\bigcup_{q\text{-small }p\leq q} \pi_qA_p 
=
\pi_q\left(\bigcup_{q\text{-small }p\leq q} A_p\right) 
\subseteq 
\left(\pi_{q'}A_q'\right)'
$$
and thus \eqref{eq:LocalRepCondition} holds.
\end{proof}

The absorbing axiom for superselection sectors is included in order to \emph{localize} a superselection sector at any chosen $p\in\cP$.

\begin{defn}
A superselection sector $(H,\pi)$ on the Hilbert space $H$ is said to be \emph{localized} at $p$ if $\pi_{p'}=\id_{A_{p'}}$.
\end{defn}

\begin{facts}
\label{facts:SSSFacts}
Here are some general facts about superselection sectors.
As above, $A$ is a $\cP$-net on the Hilbert space $H$.
\begin{enumerate}[label=\textup{(SSS\arabic*)}]

\item 
\label{SSS:AbilityToLocalize}
Given any superselection sector $(K,\pi)$ of our $\cP$-net and any $p\in\cP$, we can find an equivalent superselection sector $(H,\pi)$ localized at $p$.
We may thus restrict our attention to superselection sectors 
whose underlying Hilbert space is $H$.
\begin{proof}
Since $K$ is absorbing for $A_{p'}$, there is a unitary $u: H\to K$ intertwining the $A_{p'}$-actions.
Hence we can define an equivalent superselection sector on $H$ by setting $\pi:=\Ad(u)\circ \pi$, which is localized in $p$ by construction.
\end{proof}

\item    
\label{SSS:LocalizedPreservesAlgebra}
If the superselection sector $(H,\pi)$ is localized at $p$, then $\pi_p A_p \subseteq A_p$.
\begin{proof}
For all $x\in A_p$ and all $y\in A_{p'}$, we have
$$
\pi_p(x)y
=
\pi_p(x)\pi_{p'}(y)
\underset{(\pi \text{ is local)}}{=}
\pi_{p'}(y)\pi_p(x)
=
y\pi_p(x)
$$
and thus $\pi_p(x) \in A_{p'}' = A_p$ by Haag duality.
\end{proof}

\item 
\label{SSS:LocalizedInLarger}
If $(H,\pi)$ is localized at $p$ and $p\leq q$, then 
$q'\leq p'$, and so $\pi_{q'}=\pi_{p'}|_{A_{q'}}=\id_{A_{q'}}$.
Thus $\pi_q A_q\subseteq A_q$, and $\pi$ is also localized in $q$.

\item
\label{SSS:IntertwinerLocalized}
If $(H,\pi)$ is localized at $p$, $(H,\sigma)$ is localized at $q$, and $p,q\leq r$,
then every intertwiner $u\in B(H)$ between $\pi$ and $\sigma$ lies in $A_r$.
\begin{proof}
For all $x\in A_{r'}$,
$$
ux=u\pi_{r'}(x) =\sigma_{r'}(x)u=xu,
$$
and thus $u\in A_{r'}'=A_r$.
\end{proof}

\end{enumerate}
\end{facts}

\section{Fusion of superselection sectors}
\label{sec:Fusion}
We now construct a family of fusion products $\circ_p$ for superselection sectors, one for each $p\in\cP$ at which we may localize.
To do so, we will end up introducing some geometric axioms for $\cP$ to satisfy that will help with the construction.
For notational simplicity, we will write $\circ$ for $\circ_p$ when no confusion can arise. 

\begin{nota}
\label{nota:LocalizingNotation}
Given a $\cP$-net of algebras $A$, we often consider three  superselection sectors denoted $\pi, \sigma, \tau$ simultaneously.
By \ref{SSS:AbilityToLocalize}, we may localize them at any $p\in\cP$, meaning there exist superselection sectors
$\pi^p, \sigma^p, \tau^p$ localized at $p$
and unitary intertwiners
$u_p,v_p,w_p$ such that
$$
\pi_q(-) = u_p \pi^p_q(-) u_p^{*},
\qquad
\sigma_q(-) = v_p \sigma^p_q(-) v_p^{*},
\qquad
\text{and}
\qquad
\tau_q(-) = w_p \tau^p_q(-) w_p^{*} 
\qquad
\forall\, q\in \cP.
$$
That is, a superscript $p\in\cP$ on a superselection sector denotes that it is localized in $p$,
and a subscript on a unitary denotes that it intertwines a superselection sector to one localized in $p$.
\end{nota}

\subsection{Definition of fusion}

In this section, we define a fusion operation on $\SSS_p$.
We do so by generalizing the following example.

\begin{ex}
Consider the poset $\cP$ of non-empty proper subsets of a two-point set $\{p,p'\}$.
A net of algebras on $\cP$ is a properly infinite von Neumann algebra $A_p\subset B(H)$ with properly infinite commutant $A_{p'}=A_p'$.
A superselection sector is a Hilbert space $K$ with commuting faithful normal actions of $A_p$ and $A_{p'}$.
Since both $H,K$ are absorbing for $A_{p'}$, as in Fact \ref{fact:AbsorbingGivesEndomorphism}, there is an $A_{p'}$-linear unitary $u:H\to K$ and $\rho(x):=u^*\pi_K(x)u$ is an absorbing (injective) endomorphism of $A_p$.
In this way, we see that $\SSS_p$ is equivalent to $\End_{\rm inj}(\sB A_p)$, the category of faithful $*$-endomorphisms of $A_p$ and intertwiners.
This latter category clearly has a tensor product, namely composition.
\end{ex}

To introduce our first geometric axiom, for $p,q\in \cP$, we define the following notions.
\begin{itemize}
\item 
We say $p$ is \emph{$q$-small} if there are $r,s\in\cP$ such that $p,q\leq r$ and $p,q'\leq s$.
(Observe $q$-small is identical to $q'$-small.)
\item
A $\widetilde{p}\leq p$ is called a \emph{$q$-indicator} if $\widetilde{p}\leq q$ or $\widetilde{p}\leq q'$.
\end{itemize}
\begin{itemize}
\item[\text{\ref{geom:qSmallqIndicator}}]
\textup{\underline{Small indicators:}
For any $p,q\in\cP$, there is a $q$-small $q$-indicator $\widetilde{p}\leq p$.
}
\end{itemize}

\begin{construction}
\label{constr:circp}
Fix $p\in\cP$, and suppose we have superselection sectors $\pi,\sigma$ localized at $p$.
For an arbitrary $q\in\cP$, there are 3 cases.
\begin{enumerate}[label=\underline{Case \arabic*:}]
\item 
If $p\leq q'$,
then $\pi_q|_{A_q}=\id_{A_q}=\sigma_q|_{A_q}$,
so we may define
$(\pi\circ \sigma)_q := \id_{A_q}=\pi_q\circ \sigma_q$.
\item 
If $p\leq q$,
then by \ref{SSS:LocalizedInLarger}, 
$\pi_q A_q\subseteq A_q$ and $\sigma_q A_q\subseteq A_q$, so we may define
$(\pi\circ \sigma)_q := \pi_q\circ \sigma_q$.
\item
If neither of the first two cases hold, then we invoke the 
existence of a $q$-small $q$-indicator $\widetilde{p}\leq p$ from \ref{geom:qSmallqIndicator} in order to reduce to Case 1 or Case 2.

We know that $\widetilde{p}\leq q$ or $\widetilde{p}\leq q'$.
(One may choose to prioritize the former rather than the latter if they prefer, but it is not necessary.)
In either case, 
there exist 
superselection sectors $\pi^{\widetilde{p}}, \sigma^{\widetilde{p}}$ localized at $\widetilde{p}\leq p$ and intertwining unitaries 
$u_{\widetilde{p}},v_{\widetilde{p}}\in A_p$
by \ref{SSS:IntertwinerLocalized}
such that
$\pi(-)=u_{\widetilde{p}}\pi^{\widetilde{p}}(-)u_{\widetilde{p}}^*$
and
$\sigma(-)=v_{\widetilde{p}}\sigma^{\widetilde{p}}(-)v_{\widetilde{p}}^*$.

The main idea now is to define the composite $\pi^{\widetilde{p}}_q\circ \sigma^{\widetilde{p}}_q$ of the equivalent superselection sectors 
$\pi^{\widetilde{p}}, \sigma^{\widetilde{p}}$, and then use the operator 
$u_{\widetilde{p}}\pi^{\widetilde{p}}_p(v_{\widetilde{p}})$
as an `intertwiner' to define the composite $(\pi\circ \sigma)_q$.
\begin{enumerate}[label=\underline{Case 3\alph*:}]
\item 
If $\widetilde{p}\leq q'$,
then as in Case 1,
$\pi^{\widetilde{p}}_q = \id_{A_q}=\sigma^{\widetilde{p}}_q$, so we may define
$$
(\pi\circ\sigma)_q(x)
:=
u_{\widetilde{p}}\pi^{\widetilde{p}}_p(v_{\widetilde{p}})
x
\pi^{\widetilde{p}}_p(v_{\widetilde{p}})^*u_{\widetilde{p}}^*
=
\Ad(u_{\widetilde{p}}\pi^{\widetilde{p}}_p(v_{\widetilde{p}}))(x)
\qquad\qquad
\forall\, x\in A_q.
$$

\item 
If $\widetilde{p}\leq q$,
then as in Case 2, 
$\pi^{\widetilde{p}}_q A_q\subseteq A_q$
and
$\sigma^{\widetilde{p}}_q A_q\subseteq A_q$,
so we may define
$$
(\pi\circ\sigma)_q(x)
:=
u_{\widetilde{p}}\pi^{\widetilde{p}}_p(v_{\widetilde{p}})(\pi^{\widetilde{p}}_q\circ \sigma^{\widetilde{p}}_q)(x)
\pi^{\widetilde{p}}_p(v_{\widetilde{p}})^*u_{\widetilde{p}}^*
\qquad\qquad
\forall\, x\in A_q.
$$
\end{enumerate}
\end{enumerate}
\end{construction}

\begin{rem}
Note that in all the cases above, we have 
\begin{equation}
\label{eq:UniformDefOfComposite}
(\pi\circ\sigma)_q(x)
=
u_{\widetilde{p}}\pi^{\widetilde{p}}_p(v_{\widetilde{p}})(\pi^{\widetilde{p}}_q\circ \sigma^{\widetilde{p}}_q)(x)
\pi^{\widetilde{p}}_p(v_{\widetilde{p}})^*u_{\widetilde{p}}^*
\qquad\qquad
\forall\, x\in A_q.
\end{equation}
Indeed, in both Cases 1 and 2, 
$\widetilde{p}=p$, $u_{\widetilde{p}}=1$, and $v_{\widetilde{p}}=1$.
This uniform definition makes things easier to check later.
\end{rem}

\subsection{Fusion is well-defined}

Our next task is to show that the above composition operation $\circ_p$ is well-defined, that is, if in Case 3 we made 2 different choices, then we would get the same $(\pi\circ\sigma)_q$ for all $q\in\cP$.

\begin{rem}
\label{rem:CompositeEquivalence}
The intuition of the proof below comes from considering two endomorphisms $\pi,\sigma$ of a single von Neumann algebra $M$.
Suppose we have endomorphisms 
$\pi^{\widetilde{p}}, \sigma^{\widetilde{p}}$
and
$\pi^{\widehat{p}},\sigma^{\widehat{p}}$
equivalent to $\pi,\sigma$ respectively via unitaries
$u_{\widetilde{p}},v_{\widetilde{p}},u_{\widehat{p}},v_{\widehat{p}}\in M$, i.e., 
$$
\pi(-)
=
u_{\widetilde{p}}\pi^{\widetilde{p}}(-)u_{\widetilde{p}}^*
=
u_{\widehat{p}}\pi^{\widehat{p}}(-)u_{\widehat{p}}^*
\qquad\qquad
\text{and}
\qquad\qquad
\sigma(-)
=
v_{\widetilde{p}}\sigma^{\widetilde{p}}(-)v_{\widetilde{p}}^*
=v_{\widehat{p}}\sigma^{\widehat{p}}(-)v_{\widehat{p}}^*
.
$$
For the unitary $w:=
\pi^{\widehat{p}}(v_{\widehat{p}}^*)u_{\widehat{p}}^*
u_{\widetilde{p}}\pi^{\widetilde{p}}(v_{\widetilde{p}})\in M$,
we have
$$
\pi^{\widehat{p}}(v_{\widehat{p}}^*v_{\widetilde{p}})
u_{\widehat{p}}^*u_{\widetilde{p}}
=
w
=
u_{\widehat{p}}^*u_{\widetilde{p}}
\pi^{\widetilde{p}}(v_{\widehat{p}}^*v_{\widetilde{p}}),
$$
and a straightforward calculation proves that
$$
w(\pi^{\widetilde{p}}\circ \sigma^{\widetilde{p}})(-)w^*
=
(\pi^{\widehat{p}}\circ \sigma^{\widehat{p}})(-).
$$
\end{rem}

In the case of superselection sectors, the difficulty in proving that $\pi\circ_p \sigma$ is well-defined is that we must introduce subscripts $q\in\cP$ on the $\pi$'s in the equations above which do not \emph{a priori} match with the subscripts on the localizing unitaries.
Since these subscripts do not match, we cannot immediately see that the $w$ defined in the above way will work to intertwine our two composites for every $q\in\cP$.
To overcome this obstacle, we introduce a second geometric axiom \ref{geom:ZigZag} in order to make the subscripts agree so that the method from Remark \ref{rem:CompositeEquivalence} can be made to work.

\begin{lem}
\label{lem:FusionWellDefined}
Assuming \ref{geom:SelfDisjoint}--\ref{geom:ZigZag} below, for every $q\in\cP$, $(\pi\circ_p\sigma)_q$ is well-defined, i.e, independent of any choices in Case 3 of Construction \ref{constr:circp}.
\end{lem}
\begin{proof}
Suppose $\widetilde{p},\widehat{p}\leq p$ are $q$-small $q$-indicators, and choose superselection sectors $\pi^{\widetilde{p}}, \sigma^{\widetilde{p}}$ localized at $\widetilde{p}$
and
$\pi^{\widehat{p}},\sigma^{\widehat{p}}$ localized at $\widehat{p}$ 
together with intertwining unitaries 
$u_{\widetilde{p}},v_{\widetilde{p}},u_{\widehat{p}},v_{\widehat{p}} \in A_p$
(by \ref{SSS:IntertwinerLocalized})
such that
$$
\pi_q(\cdot)
=
u_{\widetilde{p}}\pi^{\widetilde{p}}_q(\cdot)u_{\widetilde{p}}^*
=
u_{\widehat{p}}\pi^{\widehat{p}}_q(\cdot)u_{\widehat{p}}^*
\qquad\qquad
\text{and}
\qquad\qquad
\sigma_q(\cdot)
=
v_{\widetilde{p}}\sigma^{\widetilde{p}}_q(\cdot)v_{\widetilde{p}}^*
=v_{\widehat{p}}\sigma^{\widehat{p}}_q(\cdot)v_{\widehat{p}}^*
\quad\qquad
\forall\,q.
$$
We need to prove that 
$$
u_{\widetilde{p}}\pi^{\widetilde{p}}_p(v_{\widetilde{p}})(\pi^{\widetilde{p}}_q\circ \sigma^{\widetilde{p}}_q)(x)
\pi^{\widetilde{p}}_p(v_{\widetilde{p}})^*u_{\widetilde{p}}^*
=
u_{\widehat{p}}\pi^{\widehat{p}}_p(v_{\widehat{p}})(\pi^{\widehat{p}}_q\circ \sigma^{\widehat{p}}_q)(x)
\pi^{\widehat{p}}_p(v_{\widehat{p}})^*u_{\widehat{p}}^*
\qquad\qquad
\forall\, x\in A_q.
$$
This is equivalent to proving 
\begin{equation}
\label{eq:IntertwineComposite}
w(\pi^{\widetilde{p}}_q\circ \sigma^{\widetilde{p}}_q)(x)w^*
=
(\pi^{\widehat{p}}_q\circ \sigma^{\widehat{p}}_q)(x)
\qquad\qquad
\forall\, x\in A_q
\end{equation}
where
\begin{equation*}
\begin{aligned}
w:=
\pi^{\widehat{p}}_p(v_{\widehat{p}}^*)u_{\widehat{p}}^*
u_{\widetilde{p}}\pi^{\widetilde{p}}_p(v_{\widetilde{p}})
&=
\pi^{\widehat{p}}_p(v_{\widehat{p}}^*)u_{\widehat{p}}^*
\pi_p(v_{\widetilde{p}})u_{\widetilde{p}}
=
\pi^{\widehat{p}}_p(v_{\widehat{p}}^*)\pi^{\widehat{p}}_p(v_{\widetilde{p}})
u_{\widehat{p}}^*u_{\widetilde{p}}
=
\pi^{\widehat{p}}_p(v_{\widehat{p}}^*v_{\widetilde{p}})
u_{\widehat{p}}^*u_{\widetilde{p}}
\\
&=
u_{\widehat{p}}^*\pi_p(v_{\widehat{p}}^*)
u_{\widetilde{p}}\pi^{\widetilde{p}}_p(v_{\widetilde{p}})
=
u_{\widehat{p}}^*u_{\widetilde{p}}
\pi^{\widetilde{p}}_p(v_{\widehat{p}}^*)\pi^{\widetilde{p}}_p(v_{\widetilde{p}})
=
u_{\widehat{p}}^*u_{\widetilde{p}}
\pi^{\widetilde{p}}_p(v_{\widehat{p}}^*v_{\widetilde{p}}) \in A_p.
\end{aligned}
\end{equation*}

The main idea of the proof is to show that the assumptions of $q$-indicator and $q$-small imply that it makes sense to say that $w\in A_p$ intertwines the composites
$\pi^{\widetilde{p}}_q\circ \sigma^{\widetilde{p}}_q$ and $\pi^{\widehat{p}}_q\circ \sigma^{\widehat{p}}_q$.
To do so, we introduce another geometric axiom.

\begin{rem}
If we are not careful, we might introduce an axiom that is far too strong to work for all the cases we wish to consider.
For example, if we focused too heavily on intervals in $S^1$ following \cite{MR1231644}, we might be tempted to enforce the following axiom.
\begin{enumerate}[label=($!$)]
\item
\label{geom:StrongConnected}
If $p$ is not contained in $q$ or $q$ is not contained in $p$ (so that we are not in Case 1 or 2 of Construction \ref{constr:circp}), then whenever $\widetilde{p},\widehat{p}\leq p$ and $\widetilde{p},\widehat{p}\leq q'$, there is an $r\leq p$ with $r\leq q'$ such that $\widetilde{p},\widehat{p}\leq r$.
\end{enumerate}
Indeed, \ref{geom:StrongConnected} is satisfied for intervals in $S^1$ (as $p\setminus q$ is connected), but fails for cones in $\bbR^2$ (see \eqref{eq:ZigZagOfCones}).
\end{rem}

Instead, we enforce a much weaker `connectivity' axiom of $\leq$
involving `zig-zags' which allow us to wiggle between elements of $\cP$.

\begin{nota}
A \emph{zig-zag} $\widetilde{p}\leftrightsquigarrow\widehat{p}$ is a sequence
$
(z_1,y_1,\dots, z_n,y_n,z_{n+1})\subset\cP
$
with $z_1=\widetilde{p}$ and $z_{n+1}=\widehat{p}$
such that
$z_j,z_{j+1}\leq y_j$ for all $j$.
We say a zig-zag $(\widetilde{p}=z_1,y_1,\dots, z_n,y_n,z_{n+1}=\widehat{p})$
is 
\begin{itemize}
\item
\emph{contained in} $p\in\cP$
if every $y_i,z_j\leq p$ and
\item 
\emph{disjoint from} $q\in\cP$
if it is contained in $q'$.
\end{itemize}
For an example of a zig-zag of cones in $\bbR^2$, see \eqref{eq:ZigZagOfCones}.
Our zig-zag axiom is as follows.
\begin{enumerate}[label=\textup{\ref{geom:ZigZag}}]
\item 
\underline{Zig-zag:}
For every $p,q\in\cP$ and any two choices of $q$-small $q$-indicator $\widetilde{p},\widehat{p}\leq p$ as in \ref{geom:qSmallqIndicator},
there is a zig-zag between $\widetilde{p},\widehat{p}$ contained in $p$ such that each $z_j$ is a $q$-indicator and each $y_j$ is $q$-small.
(Observe this immediately implies each $z_j$ is also $q$-small.)
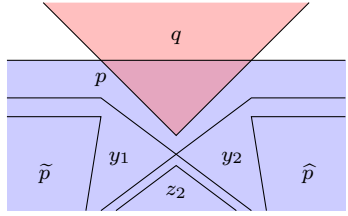
\begin{equation*}
\tag{\ref{eq:ZigZagOfCones}}
\tikzmath{
\fill[fill=blue!20] (-2.25,-1) rectangle (2.25,1);
\fill[fill=red!50, opacity=.5] (135:2.5cm) -- (0,0) -- (45:2.5cm);
\draw (135:2.5cm) -- (0,0) -- (45:2.5cm);
\draw (-2.25,1) -- (2.25,1);
\draw (-2.25,.25) -- (-1,.25) -- (-1.2,-1);
\draw (2.25,.25) -- (1,.25) -- (1.2,-1);
\draw (-2.25,.5) -- (-1,.5) -- (1,-1);
\draw (2.25,.5) -- (1,.5) -- (-1,-1);
\draw (.8,-1) -- (0,-.4) -- (-.8,-1);
\node at (-1.75,-.5) {$\scriptstyle \widetilde{p}$};
\node at (1.75,-.5) {$\scriptstyle \widehat{p}$};
\node at (-.75,-.3) {$\scriptstyle y_1$};
\node at (.75,-.3) {$\scriptstyle y_2$};
\node at (0,-.75) {$\scriptstyle z_2$};
\node at (0,1.3) {$\scriptstyle q$};
\node at (-1,.75) {$\scriptstyle p$};
}
\qquad\qquad
(z_1=\widetilde{p},y_1,z_2,y_2,z_3=\widehat{p})
\end{equation*}
\end{enumerate}
\end{nota}

Clearly \ref{geom:StrongConnected} implies \ref{geom:ZigZag}.
Moreover, \ref{geom:ZigZag} holds for cones in $\bbR^2$.
We sketch a proof in Appendix \ref{appendix:ConesSatisfyGAs} below for the convenience of the reader.

We now resume our proof that $\pi\circ_p \sigma$ is well-defined using our zig-zag axiom \ref{geom:ZigZag}. 
Fix $q\in\cP$, and let
$(\widetilde{p}=z_1,y_1,\dots, z_n,y_n,z_{n+1}=\widehat{p})$
be a zig-zag
contained in $p$
as in \ref{geom:ZigZag}.
To prove that
$$
u_{\widetilde{p}}\pi^{\widetilde{p}}_p(v_{\widetilde{p}})(\pi^{\widetilde{p}}_q\circ \sigma^{\widetilde{p}}_q)(-)
\pi^{\widetilde{p}}_p(v_{\widetilde{p}})^*u_{\widetilde{p}}^*
=
u_{\widehat{p}}\pi^{\widehat{p}}_p(v_{\widehat{p}})(\pi^{\widehat{p}}_q\circ \sigma^{\widehat{p}}_q)(-)
\pi^{\widehat{p}}_p(v_{\widehat{p}})^*u_{\widehat{p}}^*
$$
it suffices to prove that
\begin{equation}
\label{eq:IntertwineComposite-j}
u_{z_j}\pi^{z_j}_p(v_{z_j})(\pi^{z_j}_q\circ \sigma^{z_j}_q)(x)
\pi^{z_j}_p(v_{z_j})^*u_{z_j}^*
=
u_{z_{j+1}}\pi^{z_{j+1}}_p(v_{z_{j+1}})(\pi^{z_{j+1}}_q\circ \sigma^{z_{j+1}}_q)(x)
\pi^{z_{j+1}}_p(v_{z_{j+1}})^*u_{z_{j+1}}^*
\end{equation}
for all $j=1,\dots, n$ and all $x\in A_q$.
Fix a $j \in \{1, \dots, n \}$
and consider the unitary 
$$
w:= \pi^{z_{j+1}}_p(v^*_{z_{j+1}})u^*_{j+1}u_{z_j}\pi^{z_j}_p(v_{z_j}) = u^*_{z_{j+1}}u_{z_j} \pi^{z_j}_p(v^*_{z_{j+1}})\pi^{z_j}_p(v_{z_j}) = u^*_{z_{j+1}}u_{z_j} \pi^{z_j}_p(v^*_{z_{j+1}}v_{z_j}).
$$
\item[\underline{Case 1:}]
Suppose $z_j,z_{j+1}\leq q'$.
Since $y_j$ is $q$-small, there is an $s_j \geq y_j, q'$, and since the zig-zag $(z_1, y_1, \dots, z_n, y_n, z_{n + 1})$ is contained in $p$, $z_j,z_{j+1}\leq y_j\leq p$.
Since $z_j,z_{j+1}\leq q'$,
$$
(\pi^{z_j}_q\circ \sigma^{z_j}_q)(x)
=
(\pi^{z_{j+1}}_q\circ \sigma^{z_{j+1}}_q)(x)
=
x
\qquad\qquad
\forall\, x\in A_q.
$$
To verify \eqref{eq:IntertwineComposite-j}, we must show $wxw^*=x$ for all $x\in A_q$, i.e., $w\in A_q'=A_{q'}$.
Since
$v_{z_{j+1}}^*v_{z_j}$ intertwines $\sigma^{z_j}$ and $\sigma^{z_{j+1}}$
and $u_{z_{j+1}}^*u_{z_j}$
intertwines $\pi^{z_j}$ and $\pi^{z_{j+1}}$,
by \ref{SSS:IntertwinerLocalized}, $u_{z_{j+1}}^*u_{z_j},v_{z_{j+1}}^*v_{z_j}$
each lie in $A_{q'}\cap A_{y_j}\cap A_p\cap  A_{s_j}$.
By isotony,
$$
\pi^{z_j}_p(v_{z_{j+1}}^*v_{z_j})
\underset{(y_j\leq p)}{=}
\pi^{z_j}_{y_j}(v_{z_{j+1}}^*v_{z_j})
\underset{(y_j\leq s_j)}{=}
\pi^{z_j}_{s_j}(v_{z_{j+1}}^*v_{z_j})
\underset{(q'\leq s_j)}{=}
\pi^{z_j}_{q'}(v_{z_{j+1}}^*v_{z_j}) \in A_{q'},
$$
and thus our unitary 
$$
w= u^*_{z_{j+1}}u_{z_j} \pi^{z_j}_p(v^*_{z_{j+1}}v_{z_j})  \in A_q'.
$$

\item[\underline{Case 2:}]
Suppose $z_j,z_{j+1} \leq q$.
Since $y_j$ is $q$-small, there is an $r_j \geq y_j, q$ where $z_j,z_{j+1}\leq y_j\leq p$ by \ref{geom:ZigZag}.
Similar to Case 1, $u_{z_{j+1}}^*u_{z_j},v_{z_{j+1}}^*v_{z_j}$
each lie in $A_{y_j}$.
To verify \eqref{eq:IntertwineComposite-j} above, we note that for $x \in A_q$, 
\begin{align*}
w \pi^{z_j}_{q}(\sigma^{z_j}_q(x)) w^* 
&=
u^*_{z_{j+1}}u_{z_j} \pi^{z_j}_p(v^*_{z_{j+1}}v_{z_j}) \pi^{z_j}_{q}(\sigma^{z_j}_q(x)) \pi^{z_j}_p(v^*_{z_j} v_{z_{j+1}})u^*_{z_j}u_{z_{j+1}}
\\&
=
u^*_{z_{j+1}}u_{z_j} \pi^{z_j}_{y_j}(v^*_{z_{j+1}}v_{z_j}) \pi^{z_j}_{q}(\sigma^{z_j}_q(x)) \pi^{z_j}_{y_j}(v^*_{z_j} v_{z_{j+1}})u^*_{z_j}u_{z_{j+1}} 
&&
(y_j\leq p)
\\&
=
u^*_{z_{j+1}}u_{z_j} \pi^{z_j}_{r_j}(v^*_{z_{j+1}}v_{z_j}) \pi^{z_j}_{r_j}(\sigma^{z_j}_q(x)) \pi^{z_j}_{r_j}(v^*_{z_j} v_{z_{j+1}})u^*_{z_j}u_{z_{j+1}} 
&&
(y_j,q\leq r_j)
\\&= 
u^*_{z_{j+1}}u_{z_j} \pi^{z_j}_{r_j}(v^*_{z_{j+1}}v_{z_j} \sigma^{z_j}_q(x) v^*_{z_j} v_{z_{j+1}})u^*_{z_j}u_{z_{j+1}} 
\\&= 
u^*_{z_{j+1}}u_{z_j} \pi^{z_j}_{r_j}(\sigma^{z_{j+1}}_q(x))u^*_{z_j}u_{z_{j+1}} 
\\&= 
\pi^{z_{j+1}}_{r_j}(\sigma^{z_{j+1}}_q(x)) 
\\&
=
\pi^{z_{j+1}}_{q}(\sigma^{z_{j+1}}_q(x))
&&
(z_{j+1}\leq q\leq r_j).
\end{align*}

\item[\underline{Case 3:}]
Suppose without loss of generality that $z_j \leq q$, $z_{j+1}\leq q'$.
Since $y_j$ is $q$-small, there is an $r_j\geq y_j,q$, and since the zig-zag $(z_1, y_1, \dots, z_n, y_n, z_{n + 1})$ is contained in $p$, 
$z_j,z_{j+1} \leq y_j \leq p$.
In this case, checking \eqref{eq:IntertwineComposite-j} reduces to checking that
$w(\pi^{z_j}_q\circ \sigma^{z_j}_q)(x)w^*=x$ for all $x\in A_q$.
As in the previous cases,  
$u_{z_{j+1}}^*u_{z_j},v_{z_{j+1}}^*v_{z_{j}}\in A_{y_j}\subset A_{r_j}$.
Hence by isotony,
$$
\pi_p(v_{z_{j+1}}^*v_{z_{j}})
\underset{(y_j\leq p)}{=}
\pi_{y_j}(v_{z_{j+1}}^*v_{z_{j}})
\underset{(y_j\leq r_j)}{=}
\pi_{r_j}(v_{z_{j+1}}^*v_{z_{j}}).
$$
Hence we may work completely in $A_{r_j}$ to check \eqref{eq:IntertwineComposite-j}, as it contains both $w$
and $A_q$.
Explicitly,
\begin{align*}
w(\pi^{z_{j}}_q\circ \sigma^{z_{j}}_q)(x)w^*
&=
u_{z_{j+1}}^*u_{z_{j}}
\pi^{z_{j}}_{r_j}(v_{z_{j+1}}^*v_{z_{j}})
(\pi^{z_{j}}_q\circ \sigma^{z_{j}}_q)(x)
\pi^{z_{j}}_{r_j}(v_{z_{j}}^*v_{z_{j+1}})
u_{z_{j}}^*u_{z_{j+1}} 
\\&=
u_{z_{j+1}}^*u_{z_{j}}
\pi^{z_{j}}_{r_j}(
v_{z_{j+1}}^*v_{z_{j}}
\sigma^{z_{j}}_q(x)
v_{z_{j}}^*v_{z_{j+1}}
)
u_{z_{j}}^*u_{z_{j+1}}
\\&=
u_{z_{j+1}}^*u_{z_{j}}
\pi^{z_{j}}_{r_j}(
\sigma^{z_{j+1}}_q(x)
)
u_{z_{j}}^*u_{z_{j+1}}.
\end{align*}
Now since $z_{j+1}\leq q'$, $\sigma^{z_{j+1}}_q = \id_{A_q}$,
so the final term above equals
$$
u_{z_{j+1}}^*u_{z_{j}}
\pi^{z_{j}}_{r_j}(x)
u_{z_{j}}^*u_{z_{j+1}}
\underset{(q\leq r_j)}{=}
u_{z_{j+1}}^*u_{z_{j}}
\pi^{z_{j}}_q(x)
u_{z_{j}}^*u_{z_{j+1}}
=
\pi^{z_{j+1}}_q(x)
=
x
$$
as again $\pi^{z_{j+1}}_q = \id_{A_q}$.
\end{proof}

\begin{lem}
Assuming \ref{geom:SelfDisjoint}--\ref{geom:ZigZag},
$\pi\circ_p \sigma$ is a superselection sector localized at $p$.
\end{lem}
\begin{proof}
\item[\underline{Localized at $p$:}]
Since $p\leq p''$, $(\pi\circ\sigma)_{p'}=\pi_{p'}\circ\sigma_{p'}=\id_{A_{p'}}$.
\item[\underline{Isotone:}]
Suppose $q\leq r$.
It suffices to show that there exists $\widehat{p} \leq p$ that is both a $q$-small $q$-indicator and an $r$-small $r$-indicator.  
Indeed, suppose we have shown this.  
Then for all $x \in A_q$, we have by isotony for $\pi^{\widehat{p}}$ and $\sigma^{\widehat{p}}$ that
$$
  (\pi \circ \sigma)_r(x)  =   u_{\widehat{p}}\pi^{\widehat{p}}_p(v_{\widehat{p}})(\pi^{\widehat{p}}_r\circ \sigma^{\widehat{p}}_r)(x)
\pi^{\widehat{p}}_p(v_{\widehat{p}})^*u_{\widehat{p}}^* =  u_{\widehat{p}}\pi^{\widehat{p}}_p(v_{\widehat{p}})(\pi^{\widehat{p}}_q\circ \sigma^{\widehat{p}}_q)(x)
\pi^{\widehat{p}}_p(v_{\widehat{p}})^*u_{\widehat{p}}^* =   (\pi \circ \sigma)_q(x).
$$

We now show that there exists $\widehat{p} \leq p$ that is both a $q$-small $q$-indicator and an $r$-small $r$-indicator.  
Let $\widetilde{p} \leq p$ be a $q$-small $q$-indicator.
We then let $\widehat{p}\leq \widetilde{p}$ be an $r$-small $r$-indicator.
Since $\widehat{p} \leq \widetilde{p}$ and $\widetilde{p} \leq p$ is a $q$-small $q$-indicator, $\widehat{p} \leq p$ is also a $q$-small $q$-indicator by Remark \ref{rem:SmallerqSmallqIndicators}.

\item[\underline{Local:}]
We prove \eqref{eq:LocalRepCondition} for $q\in\cP$, i.e., $[(\pi\circ \sigma)_qA_q, (\pi\circ \sigma)_{q'}A_{q'}]=0$.
Up to the symmetry $q\leftrightarrow q'$, 
we may assume there is a $q$-small $q$-indicator $\widetilde{p}\leq p$ with $\widetilde{p}\leq q'$.
There then exists an $r\in\cP$ such that $\widetilde{p},q\leq r$.
Let $\pi^{\widetilde{p}},\sigma^{\widetilde{p}},u_{\widetilde{p}},v_{\widetilde{p}}$ be as in Notation \ref{nota:LocalizingNotation} so that
\begin{align*}
(\pi\circ \sigma)_q(x) 
&= 
u_{\widetilde{p}}\pi^{\widetilde{p}}_p(v_{\widetilde{p}})
x
\pi^{\widetilde{p}}_p(v_{\widetilde{p}})^*u_{\widetilde{p}}^*
&&
\forall\,x\in A_q
\\
(\pi\circ \sigma)_{q'}(y) 
&= 
u_{\widetilde{p}}\pi^{\widetilde{p}}_p(v_{\widetilde{p}})(\pi^{\widetilde{p}}_{q'}\circ \sigma^{\widetilde{p}}_{q'})(y)
\pi^{\widetilde{p}}_p(v_{\widetilde{p}})^*u_{\widetilde{p}}^*
&&
\forall\,y\in A_{q'}.
\end{align*}
Now since $\widetilde{p}\leq q'$, $\pi^{\widetilde{p}}$ and $\sigma^{\widetilde{p}}$ are localized in $q'$ by \ref{SSS:LocalizedInLarger}, and thus $\pi^{\widetilde{p}}_{q'}A_{q'}\subseteq A_{q'}$ 
and 
$\sigma^{\widetilde{p}}_{q'}A_{q'}\subseteq A_{q'}$.
It immediately follows that
\begin{align*}
(\pi\circ \sigma)_q(x) 
(\pi\circ \sigma)_{q'}(y) 
&=
u_{\widetilde{p}}\pi^{\widetilde{p}}_p(v_{\widetilde{p}})
x
\pi^{\widetilde{p}}_p(v_{\widetilde{p}})^*u_{\widetilde{p}}^*
u_{\widetilde{p}}\pi^{\widetilde{p}}_p(v_{\widetilde{p}})(\pi^{\widetilde{p}}_{q'}\circ \sigma^{\widetilde{p}}_{q'})(y)
\pi^{\widetilde{p}}_p(v_{\widetilde{p}})^*u_{\widetilde{p}}^*
\\&=
u_{\widetilde{p}}\pi^{\widetilde{p}}_p(v_{\widetilde{p}})
x
(\pi^{\widetilde{p}}_{q'}\circ \sigma^{\widetilde{p}}_{q'})(y)
\pi^{\widetilde{p}}_p(v_{\widetilde{p}})^*u_{\widetilde{p}}^*
\\&=
u_{\widetilde{p}}\pi^{\widetilde{p}}_p(v_{\widetilde{p}})
(\pi^{\widetilde{p}}_{q'}\circ \sigma^{\widetilde{p}}_{q'})(y)
x
\pi^{\widetilde{p}}_p(v_{\widetilde{p}})^*u_{\widetilde{p}}^*
\\&=
u_{\widetilde{p}}\pi^{\widetilde{p}}_p(v_{\widetilde{p}})
(\pi^{\widetilde{p}}_{q'}\circ \sigma^{\widetilde{p}}_{q'})(y)
\pi^{\widetilde{p}}_p(v_{\widetilde{p}})^*u_{\widetilde{p}}^*
u_{\widetilde{p}}\pi^{\widetilde{p}}_p(v_{\widetilde{p}})
x
\pi^{\widetilde{p}}_p(v_{\widetilde{p}})^*u_{\widetilde{p}}^*
\\&=
(\pi\circ \sigma)_{q'}(y) 
(\pi\circ \sigma)_q(x).
\end{align*}

\item[\underline{Absorbing:}]
By the absorbing axiom and Lemma \ref{lem:EndoInjectiveIffAbs} for each of $\pi$ and $\sigma$,
the normal representations $\pi_q,\sigma_q$ of $A_q$ are injective for every $q\in\cP$, as are $\pi^{\widetilde{p}}_q, \sigma^{\widetilde{p}}_q$ for any choice of $q$-small $q$-indicator $\widetilde{p} \leq p$.
We conclude by Lemma \ref{lem:EndoInjectiveIffAbs} that the normal representation $(\pi\circ_p \sigma)_q: A_q\to B(H)$ is absorbing for every $q\in\cP$.
In all cases, the definition of $(\pi\circ_p \sigma)_q$ in \eqref{eq:UniformDefOfComposite} is a unitary conjugate of a composite of injective homomorphisms, which is thus injective.
\end{proof}

\begin{rem}
\label{rem:ChangeTheOutside}
Note that for $q \in \cP$ and for $\widetilde{p} \leq p$ a $q$-small $q$-indicator, we have that 
\begin{align*}
u_{\widetilde{p}}\pi^{\widetilde{p}}_p(v_{\widetilde{p}})(\pi^{\widetilde{p}}_q\circ \sigma^{\widetilde{p}}_q)(-)
\pi^{\widetilde{p}}_p(v_{\widetilde{p}})^*u_{\widetilde{p}}^*
&=
u_{\widetilde{p}}\pi^{\widetilde{p}}_p(v_{\widetilde{p}})u_{\widetilde{p}}^*
u_{\widetilde{p}}(\pi^{\widetilde{p}}_q\circ \sigma^{\widetilde{p}}_q)(-)u_{\widetilde{p}}^*
u_{\widetilde{p}}\pi^{\widetilde{p}}_p(v_{\widetilde{p}})^*u_{\widetilde{p}}^*
\\&=
\pi_p(v_{\widetilde{p}})(\pi_q\circ \sigma^{\widetilde{p}}_q)(-)
\pi_p(v_{\widetilde{p}})^*.
\end{align*}
While this is a very simple observation, it will be helpful to us in our computations.
\end{rem}

\begin{lem}
\label{lem:MorphismTensorProduct}
Given $T \colon \pi \to \check{\pi}$ and $S \colon \sigma \to \check{\sigma}$ morphisms in $\SSS_p$, $T \circ_p S \coloneqq T\pi_p(S) = \check{\pi}_p(S)T$ intertwines $\pi \circ_p \sigma$ and $\check{\pi} \circ_p \check{\sigma}$.  
Hence the map $- \circ_p - \colon \SSS_p \times \SSS_p \to \SSS_p$ is a functor.
\end{lem}

\begin{proof}
The proof proceeds by expanding out $\check{\pi}_p(S)T$ into a more complicated formula, which we will then show intertwines $\pi \circ_p \sigma$ and $\check{\pi} \circ_p \check{\sigma}$.
Let $q \in \cP$, and let $\widetilde{p} \leq p$ be a $q$-small $q$-indicator.  
We choose $\pi^{\widetilde{p}}, \sigma^{\widetilde{p}}, \check{\pi}^{\widetilde{p}}, \check{\sigma}^{\widetilde{p}}$ localized at $\widetilde{p}$ together with unitaries $u_{\widetilde{p}}, v_{\widetilde{p}}, \check{u}_{\widetilde{p}}, \check{v}_{\widetilde{p}}$ satisfying that for all $q \in \cP$, 
\[
\pi_q(-) = u_{\widetilde{p}} \pi^{\widetilde{p}}_q(-) u_{\widetilde{p}}^{*},
\qquad
\sigma_q(-) = v_{\widetilde{p}} \sigma^{\widetilde{p}}_q(-) v_{\widetilde{p}}^{*},
\qquad
\check{\pi}_q(-) = \check{u}_{\widetilde{p}} \check{\pi}^{\widetilde{p}}_q(-) \check{u}_{\widetilde{p}}^{*}
\qquad
\check{\sigma}_q(-) = \check{v}_{\widetilde{p}} \check{\sigma}^{\widetilde{p}}_q(-) \check{v}_{\widetilde{p}}^{*}.
\]
Note that since $\widetilde{p} \leq p$, $u_{\widetilde{p}}, v_{\widetilde{p}}, \check{u}_{\widetilde{p}}, \check{v}_{\widetilde{p}} \in \cP$.
We observe that 
\[
\check{\pi}_p(S)T
=
\check{\pi}_p(\check{v}_{\widetilde{p}}\check{v}_{\widetilde{p}}^*Sv_{\widetilde{p}}v_{\widetilde{p}}^*)T
=
\check{\pi}_p(\check{v}_{\widetilde{p}})
\check{\pi}_p(\check{v}_{\widetilde{p}}^*Sv_{\widetilde{p}})
\check{\pi}_p(v_{\widetilde{p}}^*)T
=
\check{\pi}_p(\check{v}_{\widetilde{p}})
\check{\pi}_p(\check{v}_{\widetilde{p}}^*Sv_{\widetilde{p}})
T\pi_p(v_{\widetilde{p}})^*.
\]
Now, let $x \in A_q$. 
By Remark \ref{rem:ChangeTheOutside}, we have that
\[(\pi \circ_p \sigma)_q(x) 
= 
u_{\widetilde{p}}\pi^{\widetilde{p}}_p(v_{\widetilde{p}})
(\pi^{\widetilde{p}}_q\circ \sigma^{\widetilde{p}}_q)(x)
\pi^{\widetilde{p}}_p(v_{\widetilde{p}})^*u_{\widetilde{p}}^*
=
\pi_p(v_{\widetilde{p}})
(\pi_q\circ \sigma^{\widetilde{p}}_q)(x)
\pi_p(v_{\widetilde{p}})^*.
\]
Therefore, we have that 
\begin{align*}
\check{\pi}_p(S)T(\pi \circ_p \sigma)_q(x)
&=
\check{\pi}_p(\check{v}_{\widetilde{p}})
\check{\pi}_p(\check{v}_{\widetilde{p}}^*Sv_{\widetilde{p}})
T\pi_p(v_{\widetilde{p}})^*
\pi_p(v_{\widetilde{p}})
(\pi_q\circ \sigma^{\widetilde{p}}_q)(x)
\pi_p(v_{\widetilde{p}})^*
\\&=
\check{\pi}_p(\check{v}_{\widetilde{p}})
\check{\pi}_p(\check{v}_{\widetilde{p}}^*Sv_{\widetilde{p}})
T\pi_q( \sigma^{\widetilde{p}}_q(x))
\pi_p(v_{\widetilde{p}})^*
\\&=
\check{\pi}_p(\check{v}_{\widetilde{p}})
\check{\pi}_p(\check{v}_{\widetilde{p}}^*Sv_{\widetilde{p}})
\check{\pi}_q( \sigma^{\widetilde{p}}_q(x))
T\pi_p(v_{\widetilde{p}})^*.
\end{align*}

We now claim that the following equation holds:
\begin{equation}
\label{eq:UsefulEquationInMorphismTensorProduct}
\check{\pi}_p(\check{v}_{\widetilde{p}}^*Sv_{\widetilde{p}})
\check{\pi}_q( \sigma^{\widetilde{p}}_q(x))
=
\check{\pi}_q( \check{\sigma}^{\widetilde{p}}_q(x))
\check{\pi}_p(\check{v}_{\widetilde{p}}^*Sv_{\widetilde{p}}).
\end{equation}
Suppose we have shown this.  
In that case, we have that 
\begin{align*}
\check{\pi}_p(S)T(\pi \circ_p \sigma)_q(x)
&=
\check{\pi}_p(\check{v}_{\widetilde{p}})
\check{\pi}_p(\check{v}_{\widetilde{p}}^*Sv_{\widetilde{p}})
\check{\pi}_q( \sigma^{\widetilde{p}}_q(x))
T\pi_p(v_{\widetilde{p}})^*
\\&=
\check{\pi}_p(\check{v}_{\widetilde{p}})
\check{\pi}_q( \check{\sigma}^{\widetilde{p}}_q(x))
\check{\pi}_p(\check{v}_{\widetilde{p}}^*Sv_{\widetilde{p}})
T\pi_p(v_{\widetilde{p}})^*
\\&=
\check{\pi}_p(v_{\widetilde{p}})
(\check{\pi}_q\circ \check{\sigma}^{\widetilde{p}}_q)(x)
\check{\pi}_p(v_{\widetilde{p}})^*
\check{\pi}_p(\check{v}_{\widetilde{p}})
\check{\pi}_p(\check{v}_{\widetilde{p}}^*Sv_{\widetilde{p}})
T\pi_p(v_{\widetilde{p}})^*
\\&=
(\check{\pi} \circ_p \check{\sigma})_q(x)
\check{\pi}_p(S)T,
\end{align*}
where in the last line we again apply Remark \ref{rem:ChangeTheOutside}.
It therefore suffices to show that \eqref{eq:UsefulEquationInMorphismTensorProduct} holds.  
Observe that $\check{v}_{\widetilde{p}}^*Sv_{\widetilde{p}}$ intertwines $\sigma^{\widetilde{p}}$ and $\check{\sigma}^{\widetilde{p}}$, so $\check{v}_{\widetilde{p}}^*Sv_{\widetilde{p}} \in A_{\widetilde{p}}$.  
Thus, by isotony, we can rewrite \eqref{eq:UsefulEquationInMorphismTensorProduct} as
\[
\check{\pi}_{\widetilde{p}}(\check{v}_{\widetilde{p}}^*Sv_{\widetilde{p}})
\check{\pi}_q( \sigma^{\widetilde{p}}_q(x))
=
\check{\pi}_q( \check{\sigma}^{\widetilde{p}}_q(x))
\check{\pi}_{\widetilde{p}}(\check{v}_{\widetilde{p}}^*Sv_{\widetilde{p}}).
\]
Furthermore, by conjugating both sides of the above equation by $\check{u}_{\widetilde{p}}^*$, we can rewrite \eqref{eq:UsefulEquationInMorphismTensorProduct} as
\[
\check{\pi}^{\widetilde{p}}_{\widetilde{p}}(\check{v}_{\widetilde{p}}^*Sv_{\widetilde{p}})
\check{\pi}^{\widetilde{p}}_q( \sigma^{\widetilde{p}}_q(x))
=
\check{\pi}^{\widetilde{p}}_q( \check{\sigma}^{\widetilde{p}}_q(x))
\check{\pi}^{\widetilde{p}}_{\widetilde{p}}(\check{v}_{\widetilde{p}}^*Sv_{\widetilde{p}}).
\]
Now, since $\widetilde{p} \leq p$ is a $q$-indicator, we have that either $\widetilde{p} \leq q$ or $\widetilde{p} \leq q'$. 
In the first case, the desired result follows since 
\begin{align*}
\check{\pi}^{\widetilde{p}}_{\widetilde{p}}(\check{v}_{\widetilde{p}}^*Sv_{\widetilde{p}})
\check{\pi}^{\widetilde{p}}_q( \sigma^{\widetilde{p}}_q(x))
&=
\check{\pi}^{\widetilde{p}}_{q}(\check{v}_{\widetilde{p}}^*Sv_{\widetilde{p}})
\check{\pi}^{\widetilde{p}}_q( \sigma^{\widetilde{p}}_q(x))
=
\check{\pi}^{\widetilde{p}}_{q}(\check{v}_{\widetilde{p}}^*Sv_{\widetilde{p}} \sigma^{\widetilde{p}}_q(x))
\\&=
\check{\pi}^{\widetilde{p}}_{q}(\check{\sigma}^{\widetilde{p}}_q(x) \check{v}_{\widetilde{p}}^*Sv_{\widetilde{p}})
=
\check{\pi}^{\widetilde{p}}_{\widetilde{p}}(\check{v}_{\widetilde{p}}^*Sv_{\widetilde{p}})
\check{\pi}^{\widetilde{p}}_q( \check{\sigma}^{\widetilde{p}}_q(x)).
\end{align*}
In the second case, we have that $\check{\pi}^{\widetilde{p}}_q ({\sigma}^{\widetilde{p}}_q(x)) = x = \check{\pi}^{\widetilde{p}}_q (\check{\sigma}^{\widetilde{p}}_q(x))$ and $\check{\pi}^{\widetilde{p}}_{\widetilde{p}}(\check{v}_{\widetilde{p}}^*Sv_{\widetilde{p}}) \in A_{\widetilde{p}} \subseteq A_q'$ by Haag duality, so we have that 
\[
\check{\pi}^{\widetilde{p}}_{\widetilde{p}}(\check{v}_{\widetilde{p}}^*Sv_{\widetilde{p}})
\check{\pi}^{\widetilde{p}}_q( \sigma^{\widetilde{p}}_q(x))
=
\check{\pi}^{\widetilde{p}}_{\widetilde{p}}(\check{v}_{\widetilde{p}}^*Sv_{\widetilde{p}}) x
=
x \check{\pi}^{\widetilde{p}}_{\widetilde{p}}(\check{v}_{\widetilde{p}}^*Sv_{\widetilde{p}})
=
\check{\pi}^{\widetilde{p}}_{\widetilde{p}}(\check{v}_{\widetilde{p}}^*Sv_{\widetilde{p}})
\check{\pi}^{\widetilde{p}}_q( \check{\sigma}^{\widetilde{p}}_q(x)).
\]

In order to show that $- \circ_p - \colon \SSS_p \times \SSS_p \to \SSS_p$ is a functor, it remains to show that the exchange relation is satisfied:  
\[
(T\circ_p \id_{\check{\sigma}})(\id_{\pi}\circ_p S)
=
T\pi_p(S)
=
\check{\pi}_p(S)T
=
(\id_{\check{\pi}}\circ_p S) (T\circ_p \id_{\sigma}).
\qedhere
\]
\end{proof}

\subsection{Fusion is strictly associative}

We now wish to prove that our fusion operation $\circ_p$ is strictly associative, similar to how composition of endomorphisms in $\End(M)$ is strictly associative.

\begin{lem}
\label{lem:StrictAssociativity}
Assuming \ref{geom:SelfDisjoint}--\ref{geom:ZigZag}, $\circ_p$ is strictly associative, i.e.,
$(\pi\circ \sigma)\circ \tau =\pi\circ (\sigma\circ \tau)$.
\end{lem}
\begin{proof}
Let $q \in \cP$.  
It suffices to show that $((\pi \circ \sigma) \circ \tau)_q = (\pi \circ (\sigma \circ \tau))_q$.
This proof proceeds by the following steps.
\item[\underline{Step 1:}] Expand $((\pi \circ \sigma) \circ \tau)_q$.

Let $q \in \cP$.  
We choose $\widetilde{p}_1 \leq p$ a $q$-small $q$-indicator.  
We may assume that $\widetilde{p}_1$ is also $p$-small, so that $\widetilde{p}_1$ is a $p$-small $p$-indicator.
We also choose $\widetilde{p}_2 \leq p$ a $\widetilde{p}_1$-small $\widetilde{p}_1$-indicator.  
We claim that the unitary $t \coloneqq v_{\widetilde{p}_1}^* v_{\widetilde{p}_2} u_{\widetilde{p}_1}^* u_{\widetilde{p}_2}(u_{\widetilde{p}_2} \pi^{\widetilde{p}_2}_p(v_{\widetilde{p}_2}))^*$ localizes $\pi \circ \sigma$ to $\widetilde{p}_1$.  
Indeed, since $(\pi \circ \sigma)_{\widetilde{p}_1'}(x) = u_{\widetilde{p}_2} \pi^{\widetilde{p}_2}_{p}(v_{\widetilde{p}_2}) (\pi^{\widetilde{p}_2}_{\widetilde{p}_1'} \circ \sigma^{\widetilde{p}_2}_{\widetilde{p}_1'})(x) (u_{\widetilde{p}_2} \pi^{\widetilde{p}_2}_{p}(v_{\widetilde{p}_2}))^{*}$, we have that for $x \in A_{\widetilde{p}_1'}$, 
\begin{align*}
t(\pi \circ \sigma)_{\widetilde{p}_1'}(x)t^*
&=
v_{\widetilde{p}_1}^* v_{\widetilde{p}_2} 
u_{\widetilde{p}_1}^* u_{\widetilde{p}_2}
(u_{\widetilde{p}_2} \pi^{\widetilde{p}_2}_p(v_{\widetilde{p}_2}))^*
(\pi \circ \sigma)_{\widetilde{p}_1'}(x)
(u_{\widetilde{p}_2} \pi^{\widetilde{p}_2}_p(v_{\widetilde{p}_2}))
u_{\widetilde{p}_2}^* u_{\widetilde{p}_1}
v_{\widetilde{p}_2}^* v_{\widetilde{p}_1}
\\&=
v_{\widetilde{p}_1}^* v_{\widetilde{p}_2} 
u_{\widetilde{p}_1}^* u_{\widetilde{p}_2}
(\pi^{\widetilde{p}_2}_{\widetilde{p}_1'} \circ \sigma^{\widetilde{p}_2}_{\widetilde{p}_1'})(x) 
u_{\widetilde{p}_2}^* u_{\widetilde{p}_1}
v_{\widetilde{p}_2}^* v_{\widetilde{p}_1}
\\&=
v_{\widetilde{p}_1}^* v_{\widetilde{p}_2} 
(\pi^{\widetilde{p}_1}_{\widetilde{p}_1'} \circ \sigma^{\widetilde{p}_2}_{\widetilde{p}_1'})(x) 
v_{\widetilde{p}_2}^* v_{\widetilde{p}_1}
\\&=
v_{\widetilde{p}_1}^* v_{\widetilde{p}_2} 
\sigma^{\widetilde{p}_2}_{\widetilde{p}_1'}(x) 
v_{\widetilde{p}_2}^* v_{\widetilde{p}_1}
\\&=
\sigma^{\widetilde{p}_1}_{\widetilde{p}_1'}(x)
=
x.
\end{align*}
In the third equality, we used that $u_{\widetilde{p}_1}^* u_{\widetilde{p}_2}$ intertwines $\pi^{\widetilde{p}_2}$ and $\pi^{\widetilde{p}_1}$, and in the fourth equality we use that $\pi^{\widetilde{p}_1}$ is localized at $\widetilde{p}_1$.  
Analogous facts for $\sigma$ are used in the last two equalities.  

We now define $(\pi \circ \sigma)^{\widetilde{p}_1}$ by $(\pi \circ \sigma)^{\widetilde{p}_1}_r(x) \coloneqq t(\pi \circ \sigma)_r(x)t^*$ for $r \in \cP$ and $x \in A_r$.
We then have that for $x \in A_q$, 
\begin{align*}
((\pi \circ \sigma) \circ \tau)_q(x) 
&= 
t^{*} (\pi \circ \sigma)^{\widetilde{p}_1}_{p}(w_{\widetilde{p}_1}) 
(\pi \circ \sigma)^{\widetilde{p}_1}_{q}(\tau^{\widetilde{p}_1}_{q}(x)) 
(t^{*} (\pi \circ \sigma)^{\widetilde{p}_1}_{p}(w_{\widetilde{p}_1}))^{*}
\\&=
t^{*} (\pi \circ \sigma)^{\widetilde{p}_1}_{p}(w_{\widetilde{p}_1}) t
t^*(\pi \circ \sigma)^{\widetilde{p}_1}_{q}(\tau^{\widetilde{p}_1}_{q}(x)) t
t^*(\pi \circ \sigma)^{\widetilde{p}_1}_{p}(w_{\widetilde{p}_1}^*)t
\\&=
(\pi \circ \sigma)_p(w_{\widetilde{p}_1})
(\pi \circ \sigma)_{q}(\tau^{\widetilde{p}_1}_{q}(x))
(\pi \circ \sigma)_{p}(w_{\widetilde{p}_1})^*
\\&=
u_{\widetilde{p}_1} \pi^{\widetilde{p}_1}_{p}(v_{\widetilde{p}_1}) 
\pi^{\widetilde{p}_1}_{p}(\sigma^{\widetilde{p}_1}_{p}(w_{\widetilde{p}_1})) 
\pi^{\widetilde{p}_1}_{q}(\sigma^{\widetilde{p}_1}_{q}(\tau^{\widetilde{p}_1}_q(x))) 
\pi^{\widetilde{p}_1}_{p}(\sigma^{\widetilde{p}_1}_{p}(w_{\widetilde{p}_1}))^* 
\pi^{\widetilde{p}_1}_{p}(v_{\widetilde{p}_1})^*u_{\widetilde{p}_1}^{*}
\\&=
\pi_{p}(v_{\widetilde{p}_1}\sigma^{\widetilde{p}_1}_{p}(w_{\widetilde{p}_1})) 
\pi_{q}(\sigma^{\widetilde{p}_1}_{q}(\tau^{\widetilde{p}_1}_q(x))) 
\pi_{p}(v_{\widetilde{p}_1}\sigma^{\widetilde{p}_1}_{p}(w_{\widetilde{p}_1}))^*,
\end{align*}
where we used that $\widetilde{p}_1$ is $p$-small in the fourth equality and used Remark \ref{rem:ChangeTheOutside} in the last equality.
We now have an expanded but simplified formula for $((\pi \circ \sigma) \circ \tau)_q$, namely
\begin{equation}
\label{eq:ThreeWayComposition}
((\pi \circ \sigma) \circ \tau)_q(x) 
=
\pi_{p}(v_{\widetilde{p}_1}\sigma^{\widetilde{p}_1}_{p}(w_{\widetilde{p}_1})) 
\pi_{q}(\sigma^{\widetilde{p}_1}_{q}(\tau^{\widetilde{p}_1}_q(x))) 
\pi_{p}(v_{\widetilde{p}_1}\sigma^{\widetilde{p}_1}_{p}(w_{\widetilde{p}_1}))^*.
\end{equation}
Note the above equation holds for all $q$-small $q$-indicators $\widetilde{p}_1 \leq p$ that are also $p$-small.

\item[\underline{Step 2:}] $(\pi \circ (\sigma \circ \tau))_q$ agrees with \eqref{eq:ThreeWayComposition} for some $q$-small $q$-indicator $\widetilde{p}_1 \leq p$.

Let $\widetilde{p} \leq p$ be a $q$-small $q$-indicator.
By \ref{geom:splitting}, there exist $\widetilde{p}_1, \widetilde{p}_2 \leq \widetilde{p}$ with $\widetilde{p}_1 \leq \widetilde{p}_2'$.  
We first wish find an unitary localizing $\sigma \circ \tau$ at $\widetilde{p}$.  
We claim that $(v_{\widetilde{p}_1}\sigma^{\widetilde{p}_1}_p(w_{\widetilde{p}_1}))^*$ is such a unitary. 
Indeed, $\widetilde{p}_1$ is a $\widetilde{p}'$-indicator since $\widetilde{p}_1 \leq \widetilde{p}''$, and $\widetilde{p}_1$ is $\widetilde{p}'$-small since $\widetilde{p}_1, \widetilde{p}' \leq \widetilde{p}_2'$ and $\widetilde{p}_1, \widetilde{p} \leq \widetilde{p}$. 
We therefore have that for $x \in A_{\widetilde{p}'}$, 
\[
(\sigma \circ \tau)_{\widetilde{p}'}(x) 
= 
v_{\widetilde{p}_1} \sigma^{\widetilde{p}_1}_p(w_{\widetilde{p}_1})
\sigma^{\widetilde{p}_1}_{\widetilde{p}'}(\tau^{\widetilde{p}_1}_{\widetilde{p}'}(x))
\sigma^{\widetilde{p}_1}_p(w_{\widetilde{p}_1}^*)v_{\widetilde{p}_1}^*
=
v_{\widetilde{p}_1} \sigma^{\widetilde{p}_1}_p(w_{\widetilde{p}_1})
x
\sigma^{\widetilde{p}_1}_p(w_{\widetilde{p}_1}^*)v_{\widetilde{p}_1}^*,
\]
where the last equality follows since $\sigma^{\widetilde{p}_1}$ and $\tau^{\widetilde{p}_1}$ are localized in $\widetilde{p}_1 \leq \widetilde{p}$.
We can therefore define $(\sigma \circ \tau)^{\widetilde{p}}$ by 
\[
(\sigma \circ \tau)^{\widetilde{p}}_r(x)
\coloneqq
(v_{\widetilde{p}_1}\sigma^{\widetilde{p}_1}_p(w_{\widetilde{p}_1}))^*
(\sigma \circ \tau)_r(x)
v_{\widetilde{p}_1}\sigma^{\widetilde{p}_1}_p(w_{\widetilde{p}_1})
\]
for $r \in \cP$ and $x \in A_r$.  

Now, since $\widetilde{p}_1 \leq \widetilde{p}$ and $\widetilde{p} \leq p$ is a $q$-small $q$-indicator, $\widetilde{p}_1$ is also a $q$-small $q$-indicator by Remark \ref{rem:SmallerqSmallqIndicators}.  
Therefore, we have that for $x \in A_q$,
\[
(\sigma \circ \tau)_q(x)
=
v_{\widetilde{p}_1}\sigma^{\widetilde{p}_1}_p(w_{\widetilde{p}_1})
(\sigma_q^{\widetilde{p}_1} \circ \tau_q^{\widetilde{p}_1})(x)
(v_{\widetilde{p}_1}\sigma^{\widetilde{p}_1}_p(w_{\widetilde{p}_1}))^*.
\]
We therefore have that 
\(
(\sigma \circ \tau)^{\widetilde{p}}_q(x)
=
(\sigma_q^{\widetilde{p}_1} \circ \tau_q^{\widetilde{p}_1})(x).
\)
We can thus compute $(\pi \circ (\sigma \circ \tau))_q$ as follows: for $x \in A_q$, we have that
\begin{align*}
(\pi \circ (\sigma \circ \tau))_q(x)
&=
u_{\widetilde{p}} \pi^{\widetilde{p}}_p(v_{\widetilde{p}_1} \sigma^{\widetilde{p}_1}_p (w_{\widetilde{p}_1})) \pi^{\widetilde{p}}_q((\sigma \circ \tau)^{\widetilde{p}}_q(x)) (u_{\widetilde{p}} \pi^{\widetilde{p}}_p(v_{\widetilde{p}_1} \sigma^{\widetilde{p}_1}_p (w_{\widetilde{p}_1})))^*
\\&=
\pi_p(v_{\widetilde{p}_1} \sigma^{\widetilde{p}_1}_p(w_{\widetilde{p}_1})) \pi_q((\sigma \circ \tau)^{\widetilde{p}}_q(x)) 
(\pi_p(v_{\widetilde{p}_1} \sigma^{\widetilde{p}_1}_p (w_{\widetilde{p}_1})))^*
\\&=
\pi_p(v_{\widetilde{p}_1} \sigma^{\widetilde{p}_1}_p (w_{\widetilde{p}_1})) \pi_q(\sigma^{\widetilde{p}_1}_{q}(\tau^{\widetilde{p}_1}_{q}(x))) (\pi_p(v_{\widetilde{p}_1} \sigma^{\widetilde{p}_1}_p(w_{\widetilde{p}_1})))^*.
\end{align*}
Note that the second equality follows by Remark \ref{rem:ChangeTheOutside}.
The last line in the above equation is exactly \eqref{eq:ThreeWayComposition}.
\end{proof}

\begin{lem}
Assuming \ref{geom:SelfDisjoint}--\ref{geom:ZigZag}, $(\SSS_p,\circ_p)$ is a strict tensor category. 
\end{lem}
\begin{proof}
By Lemma \ref{lem:FusionWellDefined}, $- \circ_p -$ is well-defined on objects, and by Lemma \ref{lem:MorphismTensorProduct}, $- \circ_p - \colon \SSS_p \times \SSS_p \to \SSS_p$ is a functor. 
By Lemma \ref{lem:StrictAssociativity}, $- \circ_p -$ is strictly associative.  
It remains to check strict unitality.  
We claim that the identity superselection sector $\mathbbm{1} \in \SSS_p$ given by $\mathbbm{1}_q \coloneqq \id_{A_q}$ for $q \in \cP$ is a strict tensor unit.  
Indeed, for each $q \in \cP$, we let $\widetilde{p} \leq p$ be a $q$-small $q$-indicator.  
We then have that 
\begin{align*}
(\pi \circ_p \mathbbm{1})_q(-)
&=
u_{\widetilde{p}} \pi^{\widetilde{p}}_p(1) \pi^{\widetilde{p}}_q(\id_{A_q}(-))  \pi^{\widetilde{p}}_p(1) u^*_{\widetilde{p}}
=
u_{\widetilde{p}} \pi^{\widetilde{p}}_q(-) u^*_{\widetilde{p}} 
=
\pi_q(-)
\\
(\mathbbm{1} \circ_p \pi)_q(-)
&=
u_{\widetilde{p}} \pi^{\widetilde{p}}_q(-) u^*_{\widetilde{p}} 
=
\pi_q(-).
\qedhere
\end{align*}
\end{proof}

\subsection{Equivalence of categories}

We now show that localizing our superselection sectors in different $p\in\cP$ yields equivalent $\rmW^*$ tensor categories.
Suppose $a\leq b$ both satisfy \ref{geom:SelfDisjoint}--\ref{geom:ZigZag}.
By \ref{SSS:LocalizedInLarger}, we may consider
a superselection sector localized at $a$ to be localized at $b$, which gives us a fully faithful functor $\SSS_a\hookrightarrow \SSS_b$, which we call the \emph{inclusion afforded by} \ref{SSS:LocalizedInLarger}.


\begin{thm} 
\label{restrictingequiv}
Suppose $a\leq b$ both satisfy \ref{geom:SelfDisjoint}--\ref{geom:ZigZag}.
The inclusion $\SSS_a\hookrightarrow \SSS_b$ afforded by \ref{SSS:LocalizedInLarger} is a strict tensor equivalence, i.e., the inclusion is a tensor equivalence with trivial tensorator.
\end{thm}
\begin{proof}
This map is clearly fully faithful, and it is essentially surjective by \ref{SSS:AbilityToLocalize}.
It remains to show that the inclusion map is a strict tensor functor.  
To do so, we must show that for $\pi, \sigma \in \SSS_a$, $\pi \circ_a \sigma = \pi \circ_b \sigma$, and for $T \colon \pi \to \check{\pi}$ and $S \colon \sigma \to \check{\sigma}$ morphisms in $\SSS_a$, $T \circ_a S = T \circ_b S$. 
Let $q \in \cP$
and
let $\widetilde{a} \leq a$ be a $q$-small $q$-indicator.
Note that in this case we also have that $\widetilde{a} \leq b$.  
We then have that 
\[
(\pi \circ_a \sigma)_q(-)
=
u_{\widetilde{a}}\pi^{\widetilde{a}}_a(v_{\widetilde{a}})(\pi^{\widetilde{a}}_q\circ \sigma^{\widetilde{a}}_q)(-)
\pi^{\widetilde{a}}_a(v_{\widetilde{a}})^*u_{\widetilde{a}}^*
=
u_{\widetilde{a}}\pi^{\widetilde{a}}_b(v_{\widetilde{a}})(\pi^{\widetilde{a}}_q\circ \sigma^{\widetilde{a}}_q)(-)
\pi^{\widetilde{a}}_b(v_{\widetilde{a}})^*u_{\widetilde{a}}^*
=
(\pi \circ_b \sigma)_q(-).
\]
Now suppose $T \colon \pi \to \check{\pi}$ and $S \colon \sigma \to \check{\sigma}$ are morphisms in $\SSS_a$.  
Then $T, S \in A_a$.
Since $a \leq b$, we have by isotony that 
\[
T \circ_a S
=
T\pi_a(S)
=
T\pi_b(S)
=
T \circ_b S.
\qedhere
\]
\end{proof}

\begin{cor} 
If $p,q\in \cP$ satisfy \ref{geom:SelfDisjoint}--\ref{geom:ZigZag},
then any zig-zag from $p$ to $q$
(one exists by Lemma \ref{lem:Connected})
gives a tensor equivalence
\[
\SSS_p
=
\SSS_{z_1}
\overset{\cong}{\hookrightarrow}
\SSS_{y_1}
\overset{\cong}{\hookleftarrow}
\SSS_{z_2}
\overset{\cong}{\hookrightarrow}
\cdots
\overset{\cong}{\hookleftarrow}
\SSS_{z_n}
\overset{\cong}{\hookrightarrow}
\SSS_{y_n}
\overset{\cong}{\hookleftarrow}
\SSS_{z_{n+1}}
=
\SSS_q.
\pushQED{\qed} 
\qedhere
\popQED
\]
\end{cor}

\section{Braiding of superselection sectors}
\label{sec:braiding}

We now outline our strategy to define a braiding on $\SSS_p$ which is manifestly `localized in $p$' and does not require any choice of a `point at infinity' as in \cite[\S~IV]{MR1231644}.
We choose a splitting $(r,s)$ of $p$ and for $\pi,\sigma\in \SSS_p$, we define a `putative braiding' $\beta_{\pi,\sigma}$ (which ought to exist) by
$$
\pi\circ_p \sigma
\xleftrightarrow{\cong}
\pi^{r}\circ_p \sigma^{s}
\xleftrightarrow{\text{should be }=}
\sigma^{s}\circ_p \pi^{r}
\xleftrightarrow{\cong}
\sigma\circ_r \pi,
$$
where we have made choices of $\pi^r$ localized in $r$ equivalent to $\pi$ and $\sigma^s$ localized in $s$ equivalent to $\sigma$.
We then prove this putative braiding is well-defined.
The main techniques we use are:
\begin{itemize}
\item 
a notion of \emph{mutually disjoint zig-zags} in order to prove $\beta_{\pi,\sigma}$ 
does not depend on the choice of splitting, and
\item 
a \emph{reflection} of a splitting $(r,s)$ of $p$ to a splitting of $p'$ based on \cite[Lem.~2.2]{MR260294} in order to prove $\beta_{\pi,\sigma}$  is an intertwiner.
\end{itemize}

\subsection{Definition of the putative braiding}

\begin{construction}
\label{constr:Braiding}
Suppose $\pi,\sigma\in \SSS_p$.
Choose a splitting $(r, s)$ of $p$.
Choose unitaries $u_r,v_s\in A_p$ which intertwine $\pi$ with $\pi^r\in \SSS_r$ and $\sigma$ with $\sigma^s\in \SSS_s$ respectively, i.e.,
$$
\pi(-) = u_r\pi^r(-)u_r^*
\qquad\qquad\qquad
\sigma(-) = v_s \sigma^s(-) v_s^*.
$$
We define a putative braiding from $\pi\circ_p \sigma$ to $\sigma\circ_p \pi$ by
$$
\beta_{\pi,\sigma} 
\coloneqq 
(v_{s} \circ_p u_{r})(u_{r}^* \circ_p v_{s}^*)
=
v_s\sigma^s_p(u_r) u_r^* \pi_p(v_s)^*.
$$
\end{construction}

\begin{lem}
\label{lem:ReverseBraiding}
Switching the roles of $r,s$ in Construction \ref{constr:Braiding} produces the `reverse' putative braiding
$\beta^{\rev}_{\pi,\sigma}=\beta_{\sigma,\pi}^{-1}$.
\end{lem}
\begin{proof}
It suffices to prove
\(
\beta_{\pi, \sigma}^{-1}
=
u_r \pi^r_p(v_s)v_s^*\sigma_p(u_r)^*,
\)
as the right hand side is the formula one would obtain for $\beta_{\sigma,\pi}$ in Construction \ref{constr:Braiding} after reversing the roles of $r,s$.
Indeed,
\[
v_s\sigma^s_p(u_r) u_r^* \pi_p(v_s)^*
u_r \pi^r_p(v_s)v_s^*\sigma_p(u_r)^*
=
\sigma_p(u_r)v_s u_r^* \pi_p(v_s)^*
\pi_p(v_s)u_r v_s^*\sigma_p(u_r)^*
=
1.
\qedhere
\]
\end{proof}

\begin{lem}
\label{lem:IndependenceOfBraidingOnChoiceOfSectorFixedRS}
The formula for $\beta_{\pi, \sigma}$ does not depend on the choice of $\pi^r$ localized at $r$, $\sigma^s$ localized at $s$, and intertwiners $u_r \colon \pi^r \to \pi$ and $v_s \colon \sigma^s \to \sigma$.
\end{lem}
\begin{proof}
Consider $\pi^r, \check{\pi}^r$ localized at $r$, $\sigma^s, \check{\sigma}^s$ localized at $s$ and intertwiners $u_r \colon \pi^r \to \pi$, $\check{u}_r \colon \check{\pi}^r \to \pi$, $v_s \colon \sigma^s \to \sigma$, $\check{v}_s \colon \check{\sigma}^s \to \sigma$. 
We wish to show that 
\[
v_s \sigma_p^s (u_r) u_r^* \pi_p (v_s)^* = \check{v}_s \check{\sigma}_p^s (\check{u}_r) \check{u}_r^* \check{\pi}_p (\check{v}_s)^*.
\]
Observe that the above equation is equivalent to
\[
\sigma_p (u_r) v_s u_r^* \pi_p (v_s^*) = \sigma_p (\check{u}_r) \check{v}_s \check{u}_r^* \pi_p (\check{v}_s^*),
\]
which is equivalent to 
\[
\sigma_p (\check{u}_r^* u_r) v_s u_r^* \pi_p (v_s^* \check{v}_s) 
=
\check{v}_s\check{u}_r^*.
\]
Now, since $\pi^r, \check{\pi}^r$ are localized in $r$ and $\sigma^s, \check{\sigma}^s$ are localized in $s$, we have that $\check{u}_r^* u_r \in A_r$ and $v_s^* \check{v}_s \in A_s$ by \ref{SSS:IntertwinerLocalized}.
Therefore, since $r \leq s^\prime$, we have
\begin{align*}
\sigma_p (\check{u}_r^* u_r) v_s u_r^* \pi_p (v_s^* \check{v}_s) 
&=
v_s\sigma_p^s (\check{u}_r^* u_r) u_r^* \pi_p (v_s^* \check{v}_s) 
=
v_s\check{u}_r^* u_r u_r^* \pi_p (v_s^* \check{v}_s) 
=
v_s\check{u}_r^*\pi_p (v_s^* \check{v}_s) 
\\&=
v_s\check{\pi}_p^r(v_s^* \check{v}_s)\check{u}_r^*
=
v_s v_s^* \check{v}_s\check{u}_r^*
=
\check{v}_s\check{u}_r^*.
\qedhere
\end{align*}
\end{proof}

\subsection{Independence of the putative braiding up to mutually disjoint zig-zag}

\begin{defn}
Let $r_1, s_1, r_2, s_2 \in \cP$ with $r_1 \leq s_1'$ and $r_2 \leq s_2'$.  
A \emph{mutually disjoint zig-zag} $(r_1, s_1)\leftrightsquigarrow(r_2,s_2)$ is a pair of zig-zags 
$$
\left[
\begin{array}{*{10}c}
r_1=x_1& 
\textcolor{blue}{w_1}& \textcolor{red}{x_2}&
\dots& x_n& \textcolor{blue}{w_n}& \textcolor{red}{x_{n + 1}=r_2}
\\
\textcolor{blue}{s_1=z_1}& \textcolor{red}{y_1}& z_2& \dots& \textcolor{blue}{z_n}& \textcolor{red}{y_n}& z_{n + 1}=s_2
\end{array}
\right]
$$
such that 
for all $i = 1, \dots, n$, $\textcolor{blue}{w_i}$ is disjoint from $\textcolor{blue}{z_i}$ and $\textcolor{red}{y_i}$ is disjoint from $\textcolor{red}{x_{i + 1}}$.
\end{defn}

The existence of a mutually disjoint zig-zag from 
$(r_1, s_1)$ to $(r_2,s_2)$
allows one to first wiggle $r_1=x_1 \leq w_1 \geq x_2 =:\widetilde{r}_1$ while staying disjoint from $s_1=z_1$,
and we can then wiggle $s_1=z_1\leq y_1 \geq z_2=:\widetilde{s}_1$ 
while staying disjoint from $\widetilde{r}_1=x_2$.
We may then repeat this process along the entire mutually disjoint zig-zag.
Observe that mutually disjoint zig-zags can be trivially extended on either side to allow for multiple wiggles for the $r_i$'s keeping $s_i$ constant, and vice versa.

\begin{lem}
\label{lem:MutuallyDisjointZigZagIndependence}
If  $(r_1,s_1)$ and $(r_2,s_2)$ are two splittings in $p$
and
there is a mutually disjoint zig-zag
$(r_1,s_1)\leftrightsquigarrow (r_2,s_2)$ contained in $p$,
then the map $\beta_{\pi,\sigma}$ constructed in Construction \ref{constr:Braiding} above is independent of the choice of splitting $(r_1,s_1)$ versus $(r_2,s_2)$.
\end{lem}
\begin{proof}
 Given unitary intertwiners $u_{r_1}: \pi^{r_1} \to \pi$, $v_{s_1}: \sigma^{s_1} \to \sigma$, and  $u_{r_2}: \pi^{r_2} \to \pi$, $v_{s_2}: \sigma^{s_2} \to \sigma$, we want to show that
 \[
  v_{s_1}\sigma^{s_1}_p(u_{r_1}) u_{r_1}^* \pi_p(v_{s_1})^*
  = 
  v_{s_2}\sigma^{s_2}_p(u_{r_2}) u_{r_2}^* \pi_p(v_{s_2})^*.
 \]
 Note that by induction and by symmetry, it suffices to consider the case $s_1 = s_2 \eqqcolon s$ and $r_1, r_2 \leq r \leq p$ with $r \leq s^\prime$. 
 In that case, the equation becomes
 \[
 v_{s}\sigma^{s}_p(u_{r_1}) u_{r_1}^* \pi_p(v_{s})^*
  = 
  v_{s}\sigma^{s}_p(u_{r_2}) u_{r_2}^* \pi_p(v_{s})^*.
 \]
 Now, $u_{r_1}: \pi^{r_1} \to \pi$ and $u_{r_2}: \pi^{r_2} \to \pi$ are two choices for localizing $\pi$ to $r$ since $r_1, r_2 \leq r$.
Therefore, by Lemma \ref{lem:IndependenceOfBraidingOnChoiceOfSectorFixedRS}, the result follows.
\end{proof}

\begin{rem}
\label{rem:ReverseBraiding}
As stated in Remark \ref{rem:ReverseBraidingIntro} in the introduction,
if there is a mutually disjoint zig-zag $(r,s)\leftrightsquigarrow(s,r)$ in $p$,
then by Lemmas \ref{lem:ReverseBraiding} and \ref{lem:MutuallyDisjointZigZagIndependence}, 
for all $\pi,\sigma\in \SSS_p$,
$\beta_{\pi,\sigma}=\beta^{\rev}_{\pi,\sigma}$, which is exactly the condition that our putative braiding $\beta$ is symmetric.
An explicit example where such a mutually disjoint zig-zag exists can be constructed from disks in $\bbR^2$.
\begin{equation*}
\tag{\ref{eq:SmallMutuallyDisjointZigZagSwap}}
\tikzmath{
\draw[fill=red!50, opacity=.5] (30:1) circle (.76cm);
\draw[fill=blue!50, opacity=.5] ($ (1,0) + (150:1) $) circle (.76cm);
\draw[fill=green!50, opacity=.5] ($ (60:1) + (270:1) $) circle (.76cm);
\draw[fill=white] (0,0) node{$\scriptstyle r_1$} circle (.2cm);
\draw[fill=white] (1,0) node{$\scriptstyle r_2$} circle (.2cm);
\draw[fill=white] (60:1) node{$\scriptstyle r_3$} circle (.2cm);
\draw (30:.57735) circle (1.25cm);
\node at ($ (30:1) + (30:.4) $) {$\scriptstyle s_1$};
\node at ($ (1,0) + (150:1) + (150:.4) $) {$\scriptstyle s_2$};
\node at ($ (60:1) + (270:1) + (270:.4) $) {$\scriptstyle s_3$};
\node at ($ (60:1) + (0,.5) $) {$\scriptstyle p$};
}
\qquad\qquad
\begin{aligned}
&(r_1, 
\textcolor{blue}{s_2}, \textcolor{red}{r_3},\textcolor{blue}{s_1}, \textcolor{red}{r_2})
\\
&(\textcolor{blue}{r_2}, \textcolor{red}{s_3}, \textcolor{blue}{r_1}, \textcolor{red}{r_1}, r_1)    
\end{aligned}
\end{equation*}
\end{rem}

We now show that \ref{geom:SelfDisjoint}--\ref{geom:ZigZag}
imply the existence of mutually disjoint zig-zags.

\begin{facts}
Here are some helpful facts to construct mutually disjoint zig-zags.
\begin{enumerate}[label=(ZZ\arabic*)]
\item 
\label{ZZ:shorten}
Suppose $a,b\leq p$
and $(a=z_1,y_1,\dots, y_n,z_{n+1}=b)$ is a zig-zag contained in $p$.
If there is a $j<n$ with $y_j\Cap b \neq \emptyset$, then 
choosing $\widehat{b}\in y_j\Cap b$, there is a shorter zig-zag 
$(a=z_1,y_1,\dots, y_j,z_{j+1}=\widehat{b})$.

\item 
\label{ZZ:Concatenate}
Suppose  $a,b,c\leq p$
and $(a=z_1,y_1,\dots, y_n,z_{n+1}=b)$ is a zig-zag in $p$ such that $y_j\Cap c \neq \emptyset$ for some $j$.
There choosing any $\widetilde{c}\in y_j\Cap c$, our zig-zag is a concatenation of the zig-zags
$(a=z_1,y_1,\dots, z_j,y_j,\widetilde{c})$
and
$(\widetilde{c},y_j, z_{j+1},y_{j+1},z_{j+2},\dots,y_n,z_n=b)$.

\item 
\label{ZZ:RemoveFromThird}
Suppose $a,b,c\leq p$
and $(a=z_1,y_1,\dots, y_n,z_{n+1}=b)$ is a zig-zag in $p$ such that $y_j\Cap c = \emptyset$ for all $j$.
Then there is a $\widetilde{c}\leq c$ such that our zig-zag is contained in $\widetilde{c}'$.
\begin{proof}
Setting $c_0:=c$, observe that by induction and \ref{geom:saturate}, for all $j=1,\dots, n$, there is a $c_j\leq y_j'\Cap c_{j-1}$.
Setting $\widetilde{c}:=c_n$ works.
\end{proof}

\item 
\label{ZZ:Disjointify}
Suppose $a,b\leq p$ are disjoint
and $(a=z_1,y_1,\dots, y_n,z_{n+1}=b)$ is a zig-zag contained in $p$.
Applying \ref{ZZ:shorten}
and replacing $a,b$ with $\widehat{a}\leq a$ and $\widehat{b}\leq b$ if necessary, we may assume that our zig-zag satisfies 
$y_j\Cap a =\emptyset$ for all $j>1$
and
$y_j\Cap b =\emptyset$ for all $j<n$.
Applying \ref{ZZ:RemoveFromThird}, there are 
$\widetilde{a}\leq a$
and 
$\widetilde{b}\leq b$
so that the zig-zag
$(\widetilde{a},y_1,\dots, y_n,\widetilde{b})$ is a zig-zag contained in $p$
such that
\begin{equation}
\label{eq:MutuallyDisjoint}
\text{$y_j\leq \widetilde{a}'$ for all $j>1$
and
$y_j\leq \widetilde{b}'$ for all $j<n$.}
\end{equation}

\item 
\label{ZZ:MutuallyDisjointExists1}
Suppose $a,b,c\leq p$ are all disjoint
and $(a=z_1,y_1,\dots, y_n,z_{n+1}=b)$ is a zig-zag contained in $p$.
If $y_j\Cap c\neq \emptyset$ for some $j$, then there is a mutually disjoint zig-zag from $(a,c)\leftrightsquigarrow(a,b)\leftrightsquigarrow(c,b)$.
\begin{proof}
First, applying \ref{ZZ:Disjointify} and passing to a sub-zig-zag if necessary, there are $\widetilde{a}\leq a$ and $\widetilde{b}\leq b$ such that \eqref{eq:MutuallyDisjoint} holds.
Suppose 
$y_j\Cap c \neq \emptyset$ for some $j$.
Let 
$j$ be minimal so that $y_j\Cap c \neq \emptyset$
and
$k$ be maximal so that $y_k\Cap c \neq \emptyset$.
Choosing
$\widetilde{c}\leq y_j\Cap c$
and
$\widehat{c}\leq y_k\Cap c$,
we have zig-zags
$(\widetilde{a}=z_1,y_1,\dots, y_j,\widetilde{c})$,
$(\widehat{c},y_k,\dots, y_n,z_n=\widetilde{b})$,
and 
$(\widetilde{c}, c, \widehat{c})$
which concatenate to give a zig-zag from $\widetilde{a}$ to $\widetilde{b}$ satisfying \eqref{eq:MutuallyDisjoint}.
Then
$$
\left[
\begin{array}{*{15}c}
a& a& \widetilde{a}& \widetilde{a}& \dots& \widetilde{a}&\widetilde{a}&a& \widetilde{a}& y_1& \dots&y_j& \widetilde{c}& c& c
\\
c& c& \widehat{c}& y_k&\dots& y_n& \widetilde{b}&b&\widetilde{b}& \widetilde{b}& \dots& \widetilde{b}&\widetilde{b}&b&b
\end{array}
\right]
$$
is a mutually disjoint zig-zag $(a,c)\leftrightsquigarrow(a,b)\leftrightsquigarrow(c,b)$.
\end{proof}

\item 
\label{ZZ:MutuallyDisjointExists2}
Suppose $a,b,c\leq p$ are all disjoint and
$(a=z_1,y_1,\dots, y_n,z_{n+1}=b)$ is a zig-zag contained in $p$.
If $y_j\Cap c = \emptyset$ for all $j$,
then there is a mutually disjoint zig-zag 
$(a,b)\leftrightsquigarrow(c,b)\leftrightsquigarrow(c,a)$
or 
$(a,b)\leftrightsquigarrow(a,c)\leftrightsquigarrow(b,c)$.
\begin{proof}
First, use \ref{ZZ:Disjointify} to find $\widetilde{a}\leq a$ and $\widetilde{b}\leq b$
satisfying \eqref{eq:MutuallyDisjoint}.
We may then use \ref{ZZ:RemoveFromThird} to find a $\widetilde{c}\leq c$ so that our zig-zag is contained in $\widetilde{c}'$.

Now use \ref{geom:ZigZag} to find a zig-zag from $(\widetilde{a}=x_1,w_1,x_2,\dots, x_m,w_m,x_{m+1}=\widetilde{c})$
where each $x_i$ is a $\widetilde{b}$-indicator and each $w_i$ is $\widetilde{b}$-small.
Pick $k$ maximal so that 
$w_k\Cap \widetilde{a}\neq \emptyset$
or
$w_k\Cap \widetilde{b}\neq \emptyset$.
\item[\underline{Case 1:}]
If 
$w_k\Cap \widetilde{a}\neq \emptyset$, 
then
since $w_k$ is $\widetilde{b}$-small,
there is a $\widehat{b}\leq \widetilde{b}$ such that $\widehat{b}\leq w_k'$.
Moreover, by maximality of $k$, by \ref{ZZ:RemoveFromThird}, we may assume $w_j\leq\widehat{b}'$ for all $j\geq k$.
Then choosing $\widehat{a}\in w_k\Cap \widetilde{a}$,
$$
\left[
\begin{array}{*{15}c}
a & a &\widehat{a} &w_k &\dots & w_m &\widetilde{c} &c &\widetilde{c} &\widetilde{c} &\dots &\widetilde{c} &\widetilde{c} &c &c
\\
b&b&\widehat{b}&\widehat{b}&\dots&\widehat{b}&\widehat{b}&b&b&y_n&\dots&y_1&a&a&a
\end{array}
\right]
$$
is a mutually disjoint zig-zag  $(a,b)\leftrightsquigarrow(c,b)\leftrightsquigarrow(c,a)$.

\item[\underline{Case 2:}]
If 
$w_k\Cap \widetilde{b}\neq \emptyset$ 
but
$w_k\Cap \widetilde{a}=\emptyset$,
then
by
\ref{ZZ:RemoveFromThird}, 
there is a $\widehat{a}\leq \widetilde{a}$
such that $w_j\leq \widehat{a}'$ for all $j\geq k$.
Then choosing $\widehat{b}\in w_k\Cap \widetilde{b}$,
$$
\left[
\begin{array}{*{15}c}
a & a &\widehat{a} &\widehat{a} &\dots & \widehat{a} &\widehat{a} &a &\widetilde{a}&y_1&\dots& y_n&\widetilde{b}&b&b
\\
b&b&\widehat{b}&w_k&\dots&w_m&\widetilde{c}&c&\widetilde{c}&\widetilde{c}&\dots &\widetilde{c}&\widetilde{c}&c&c
\end{array}
\right]
$$
is a mutually disjoint zig-zag  $(a,b)\leftrightsquigarrow(a,c)\leftrightsquigarrow(b,c)$.
\end{proof}

\item 
\label{ZZ:TriangleDance}
Suppose $a,b,c\leq p$ are all disjoint.
Combining \ref{ZZ:MutuallyDisjointExists1} and \ref{ZZ:MutuallyDisjointExists2},
we see that $a,b,c$ participate in one of the following mutually disjoint zig-zags.
$$
\tikzmath{
\filldraw (-150:.5cm) node[below, xshift=-.1cm]{$\scriptstyle c$} circle (.05cm);
\filldraw (-30:.5cm) node[below, xshift=.1cm]{$\scriptstyle a$} circle (.05cm);
\filldraw (90:.5cm) node[below]{$\scriptstyle b$} circle (.05cm);
\draw[thick, red, <->] (-45:.5cm) arc(-45:-135:.5cm); 
\draw[thick, red, <->] (105:.5cm) arc(105:195:.5cm); 
\node at (-90:1cm) {$\scriptstyle (a,c)\leftrightarrow(a,b)\leftrightarrow(c,b)$};
}
\qquad\qquad\qquad
\tikzmath{
\filldraw (-150:.5cm) node[below, xshift=-.1cm]{$\scriptstyle c$} circle (.05cm);
\filldraw (-30:.5cm) node[below, xshift=.1cm]{$\scriptstyle a$} circle (.05cm);
\filldraw (90:.5cm) node[below]{$\scriptstyle b$} circle (.05cm);
\draw[thick, red, <->] (-45:.5cm) arc(-45:-135:.5cm); 
\draw[thick, red, <->] (-15:.5cm) arc(-15:75:.5cm); 
\node at (-90:1cm) {$\scriptstyle (b,a)\leftrightarrow(b,c)\leftrightarrow(a,c)$};
}
\qquad\qquad\qquad
\tikzmath{
\filldraw (-150:.5cm) node[below, xshift=-.1cm]{$\scriptstyle c$} circle (.05cm);
\filldraw (-30:.5cm) node[below, xshift=.1cm]{$\scriptstyle a$} circle (.05cm);
\filldraw (90:.5cm) node[below]{$\scriptstyle b$} circle (.05cm);
\draw[thick, red, <->] (105:.5cm) arc(105:195:.5cm); 
\draw[thick, red, <->] (-15:.5cm) arc(-15:75:.5cm); 
\node at (-90:1cm) {$\scriptstyle (a,b)\leftrightarrow(a,c)\leftrightarrow(b,c)$};
}
$$

\end{enumerate}
\end{facts}

\begin{thm}
\label{thm:MutuallyDisjointZigZagsExist}
The axioms \ref{geom:SelfDisjoint}--\ref{geom:ZigZag} imply 
that given any two splittings $(r_1,s_1), (r_2,s_2)$ of $p$, there is a mutually disjoint zig-zag
$(r_1,s_1)\leftrightsquigarrow(r_2,s_2)$ 
or 
$(r_1,s_1)\leftrightsquigarrow(s_2,r_2)$ 
contained in $p$.
\end{thm}
\begin{proof}
To ease the notation, suppose $(r,s)=(r_1,s_1)$ and $(a,b)=(r_2,s_2)$ are two splittings of $p$.
Up swapping $r\leftrightarrow s$, $a\leftrightarrow b$, and the pair $(r,s)\leftrightarrow(a,b)$
(notice that the statement of the theorem is preserved under these swaps!),
we have one of the following 3 cases.

\item[\underline{Case 1:}]
Suppose $r\Cap a \neq \emptyset$ and $s\Cap b\neq \emptyset$.
Choose $x\leq r\Cap a$ and $y\leq s\Cap b$, and note that
$$
\left[
\begin{array}{*{15}c}
r & r & x & a & a
\\
s & s & y & b & b
\end{array}
\right]
$$
is a mutually disjoint zig-zag $(r,s)\leftrightsquigarrow (a,b)$ contained in $p$.

\item[\underline{Case 2:}]
Suppose $r\Cap a = \emptyset$ but $s\Cap b \neq \emptyset$.
There are two subcases based on $r\Cap b$ via \ref{geom:saturate}:
\begin{itemize}
\item 
\underline{Case 2a:}
Suppose $r\Cap b=\emptyset$.
By two uses of \ref{geom:saturate},
we may choose $\widehat{r}\in r\Cap b'\neq \emptyset$ 
and then choose $\widetilde{r}\in \widehat{r}\Cap a'\neq \emptyset$.
Moreover, we may replace $b$ with $\widetilde{b}\in s\Cap b$.
Then $\widetilde{r},a,\widetilde{b}$ are all disjoint
and $\widetilde{b}\leq s$

\item 
\underline{Case 2b:}
Suppose $r\Cap b\neq\emptyset$.
Choose $\widetilde{r}\in r\Cap b$ and $\widetilde{s}\in s\Cap b$.
Then $\widetilde{r},\widetilde{s},a$ are all disjoint and $\widetilde{s}\leq b$.
\end{itemize}
In either case above, we may reduce to the case that $a,b,r$ are
all disjoint and $b\leq s$.
By \ref{ZZ:TriangleDance}, $a,b,r$ participate in a mutually disjoint zig-zag 
$(a,b)\leftrightsquigarrow (r,b)$
or 
$(a,b)\leftrightsquigarrow (b,r)$.
Observe that 
$$
\left[
\begin{array}{*{15}c}
r & r & r
\\
b & s &s
\end{array}
\right]
$$
is a mutually disjoint zig-zag $(r,b)\leftrightsquigarrow (r,s)$.
We can now concatenate to get the desired mutually disjoint zig-zag.

\item[\underline{Case 3:}]
Suppose $r\Cap a = \emptyset= r\Cap b$ and $s\Cap a = \emptyset = s\Cap b$.
Then using 4 instances of \ref{geom:saturate},
we may choose $\widehat{a}\in r'\Cap a\neq \emptyset$
and $\widetilde{a}\in \widehat{a}\Cap s'\neq\emptyset$,
and then we may choose $\widehat{b}\in r'\Cap b\neq \emptyset$
and $\widetilde{b}\in \widehat{b}\Cap s'\neq\emptyset$.
We can thus reduce to the case that
$a,b,r,s$ are all disjoint.
By \ref{ZZ:TriangleDance}, $a,b,r$ participate in a mutually disjoint zig-zag 
$(a,b)\leftrightsquigarrow (r,b)$
or 
$(a,b)\leftrightsquigarrow (b,r)$.
Again by \ref{ZZ:TriangleDance}, $b,r,s$ participate in a mutually disjoint zig-zag 
$(r,b)\leftrightsquigarrow (r,s)$
or 
$(r,b)\leftrightsquigarrow (s,r)$.
We can thus concatenate to get the desired mutually disjoint zig-zag.
\end{proof}

\subsection{The putative braiding is an intertwiner}

For this section, we assume $\cP$ satisfies \ref{geom:SelfDisjoint}--\ref{geom:ZigZag}.

\begin{defn}
Given a splitting $(r,s)$ of $p$,
a \emph{reflection is}
a splitting
$(a,b)$ of $p'$
and some $c$
such that 
$r,a\leq c$
and
$s,b\leq c'$. 
\begin{equation*}
\tag{\ref{eq:Reflection}}
\tikzmath{
\draw (-1.6,0) -- (1.6,0);
\draw (0,-1.6) -- (0,1.6);
\fill[blue!25] (-1.2,.2) rectangle (-.2,1.2);
\draw (-1.2,.2) -- (-.2,.2) -- (-.2,1.2);
\node at (-.4,.4) {$\scriptstyle r$};
\fill[red!25] (1.2,.2) rectangle (.2,1.2);
\draw (1.2,.2) -- (.2,.2) -- (.2,1.2);
\node at (.4,.4) {$\scriptstyle s$};
\fill[green!25] (-1.2,-.2) rectangle (-.2,-1.2);
\draw (-1.2,-.2) -- (-.2,-.2) -- (-.2,-1.2);
\node at (-.4,-.4) {$\scriptstyle a$};
\fill[orange!25] (1.2,-.2) rectangle (.2,-1.2);
\draw (1.2,-.2) -- (.2,-.2) -- (.2,-1.2);
\node at (.4,-.4) {$\scriptstyle b$};
\node at (-1.4,.2) {$\scriptstyle p$};
\node at (-1.4,-.2) {$\scriptstyle p'$};
\node at (1.4,.2) {$\scriptstyle p$};
\node at (1.4,-.2) {$\scriptstyle p'$};
\node at (.2,-1.4) {$\scriptstyle c'$};
\node at (-.2,-1.4) {$\scriptstyle c$};
\node at (.2,1.4) {$\scriptstyle c'$};
\node at (-.2,1.4) {$\scriptstyle c$};
}
\end{equation*}
We say $(r,s)$ \emph{admits a reflection} if such a reflection exists.
\end{defn}

\begin{prop}
\label{prop:AdmitsReflection}
For any $p\in\cP$, there is a splitting of $p$ which admits a reflection.
\end{prop}
\begin{proof}
Suppose $(r_0,s_0)$ is a splitting of $p$.
Since $p'$ splits by \ref{geom:splitting}, there is an $a_0\leq p'$ which is $p$-small.
Thus $a_0$ is a $p$-small $p$-indicator, as is $r_0$.
By \ref{geom:ZigZag} applied to $r_0,a_0\leq s_0'$, there is a zig-zag from $r_0$ to $a_0$ contained in $s_0'$ such that each $z_j$ is an $p$-indicator and each $y_j$ is $p$-small.
Let $j$ be minimal such that $r:=z_j\leq p$ and $a:=z_{j+1}\leq p'$.
Note that $c:=y_j$ is $p$-small, so there is a $b'$ such that $c,p\leq b'$, which implies $b\leq p'$ and $b\leq c'\leq a'$.
We have thus constructed a splitting $r,s:=s_0$ of $p$ and a splitting $a,b$ of $p'$ such that there is a $c$ such that $r,a\leq c$, but $s,b\leq c'$.
\end{proof}

\begin{lem}
\label{lem:LocalizedInDisjointCommute}
Suppose $r, s$ is a splitting of $p$
and
$\pi, \sigma$ are localized at $r, s$ respectively.
\begin{itemize}
\item
Whenever $p \leq q'$, $(\pi \circ_p \sigma)_q = \id_{A_q} = (\sigma \circ_p \pi)_q$.
\item 
$(\pi \circ_p \sigma)_p = (\sigma \circ_p \pi)_p$.
\item
Whenever $p\leq q$, $(\pi \circ_p \sigma)_q = (\sigma \circ_p \pi)_q$.
\end{itemize}
\end{lem}
\begin{proof}
The first statement is obvious.
The third statement follows immediately from the second, since
$(r,s)$ is also a splitting of $q$, and
$$
(\pi \circ_p \sigma)_q 
= 
(\pi \circ_q \sigma)_q 
= 
(\sigma \circ_q \pi)_q 
= 
(\sigma \circ_p \pi)_q.
$$
It remains to prove the second statement.
By Proposition \ref{prop:AdmitsReflection},
there is a splitting $(r_1,s_1)$ of $p$ which admits a reflection.
Up to swapping $r_1\leftrightarrow s_1$,
there is a mutually disjoint zig-zag $(r,s)\leftrightsquigarrow(r_1,s_1)$
by Theorem \ref{thm:MutuallyDisjointZigZagsExist}.
By Lemma \ref{lem:MutuallyDisjointZigZagIndependence},
the map $\beta_{\pi,\sigma}$ is independent of the choice of splitting $(r,s)$ versus $(r_1,s_1)$.
If we select the splitting $(r,s)$, then to define $\beta_{\pi,\sigma}$, we may choose $u=1: \pi^r:=\pi\to \pi$ and $v=1:\sigma^s:=\sigma\to \sigma$, giving
$\beta_{\pi,\sigma}=1\sigma_p(1)1\pi_p(1)^*=1$.
If we select the splitting $(r_1,s_1)$ and we choose arbitrary intertwiners $u:\pi^{r_1}\to \pi$ and $v: \sigma^{s_1}\to \sigma$, we have the formula
$$
1=\beta_{\pi,\sigma}=
v_{s_1}\sigma^{s_1}_p(u_{r_1}) u_{r_1}^* \pi_p(v_{s_1})^*
\qquad
\Longleftrightarrow
\qquad
v_{s_1}\sigma^{s_1}_p(u_{r_1})
=
\pi_p(v_{s_1})u_{r_1}
=
u_{r_1}\pi^{r_1}_p(v_{s_1}).
$$
We now calculate that for all $x\in A_p$,
\begin{align*}
(\pi\circ_p \sigma)_p(x)
&=
\pi_p(\sigma_p(x))
\\&=
u_{r_1}
\pi_{p}^{r_1}(v_{s_1})
\pi^{r_1}_p(\sigma^{s_1}_p(x))
\pi_{p}^{r_1}(v_{s_1})^*
u_{r_1}^*
\\&=
v_{s_1}\sigma^{s_1}_p(u_{r_1})
\pi^{r_1}_p(\sigma^{s_1}_p(x))
\sigma^{s_1}_p(u_{r_1})^*
v_{s_1}^*.
\end{align*}
Now as $(r_1,s_1)$ admits a reflection, there is a splitting $(a,b)$ of $p'$ and $c\in\cP$ such that $r_1,a\leq c$ and $s_1,b\leq c'$.
Choosing $u_a: \pi^a\to \pi^{r_1}$, we have
\begin{align*}
v_{s_1}\sigma^{s_1}_p(u_{r_1})
\pi^{r_1}_p(\sigma^{s_1}_p(x))
\sigma^{s_1}_p(u_{r_1})^*
v_{s_1}^*
&=
v_{s_1}\sigma^{s_1}_p(u_{r_1})
u_a
\pi^{a}_p(\sigma^{s_1}_p(x))
u_a^*
\sigma^{s_1}_p(u_{r_1})^*
v_{s_1}^*
\\&=
v_{s_1}\sigma^{s_1}_q(u_{r_1})
u_a
\sigma^{s_1}_p(x)
u_a^*
\sigma^{s_1}_p(u_{r_1})^*
v_{s_1}^*
\\&=
v_{s_1}\sigma^{s_1}_p(u_{r_1})
u_a
\sigma^{s_1}_p(\pi^a_p(x))
u_a^*
\sigma^{s_1}_p(u_{r_1})^*
v_{s_1}^*.\nonumber
\end{align*}
Now since $a,r_1\leq c\leq b'$, $u_a\in A_c$.
Since $\sigma^{s_1}$ is localized in $s_1\leq c'$ and $c,p\leq b'$, we have
$u_a=\sigma^{s_1}_{c}(u_a)=\sigma_{b'}^{s_1}(u_a)$,
so
\begin{align*}
v_{s_1}\sigma^{s_1}_p(u_{r_1})
u_a
\sigma^{s_1}_p(\pi^a_p(x))
u_a^*
\sigma^{s_1}_p(u_{r_1})^*
v_{s_1}^*
&=
v_{s_1}\sigma^{s_1}_{b'}(u_{r_1})
\sigma^{s_1}_{b'}(u_a)
\sigma^{s_1}_{b'}(\pi^a_{p}(x))
\sigma^{s_1}_{b'}(u_a^*)
\sigma^{s_1}_{b'}(u_{r_1})^*
v_{s_1}^*
\\&=
v_{s_1}\sigma^{s_1}_{b'}(u_{r_1}u_a \pi^a_p(x)u_a^*u_{r_1}^*)
v_{s_1}^*
\\&=
v_{s_1}\sigma^{s_1}_{b'}(\pi_p(x))
v_{s_1}^*
\\&=
\sigma_p(\pi_p(x))\nonumber
\end{align*}
as desired.
\end{proof}

\begin{prop}
Suppose $\pi,\sigma\in \SSS_p$.
Choose a splitting $(r,s)$ of $p$.
The operator $\beta_{\pi,\sigma}$ from Construction \ref{constr:Braiding} is an intertwiner
$\pi\circ_p\sigma \to \sigma\circ_p \pi$.
\end{prop}
\begin{proof}
Fix $q\in\cP$.
By \ref{geom:qSmallqIndicator}, there exists $\widetilde{p} \leq p$ a $q$-small $q$-indicator and by \ref{geom:splitting}, there is a splitting $r, s \leq \widetilde{p}$.
The main idea is to use the isomorphism
\[
(\pi\circ_p \sigma)_q
\xleftrightarrow{\cong}
(\pi^{r}\circ_p \sigma^{s})_q
=
(\pi^{r}\circ_{\widetilde{p}} \sigma^{s})_q
\underset{\text{(Lem.~\ref{lem:LocalizedInDisjointCommute})}}{=}
(\sigma^{r}\circ_{\widetilde{p}} \pi^{s})_q
=
(\sigma^{r}\circ_p \pi^{s})_q
\xleftrightarrow{\cong}
(\sigma\circ_p \pi)_q.
\]
If $u_r: \pi^r \to \pi$ and $v_s: \sigma^s \to \sigma$ are intertwiners, by Lemma \ref{lem:MorphismTensorProduct}, we see explicitly that for $x\in A_q$,
\begin{align*}
    \beta_{\pi, \sigma} (\pi \circ_p \sigma)_q(x)  & = (v_s \circ_p u_r)(u_r^* \circ_p v_s^*) (\pi \circ_p \sigma)_q(x) = (v_s \circ_p u_r)(\pi^r \circ_p \sigma^s)_q(x) (u_r^* \circ_p v_s^*)\\
    & = (v_s \circ_p u_r)(\pi^r \circ_{\widetilde{p}} \sigma^s)_q(x) (u_r^* \circ_p v_s^*) = (v_s \circ_p u_r)(\sigma^s \circ_{\tilde{p}} \pi^r)_q(x) (u_r^* \circ_p v_s^*) \\
    &= (v_s \circ_p u_r)(\sigma^s \circ_p \pi^r)_q(x) (u_r^* \circ_p v_s^*) = (\sigma \circ_p \pi)_q(x) (v_s \circ_p u_r) (u_r^* \circ_p v_s^*) \\
    & = (\sigma \circ_p \pi)_q(x) \beta_{\pi, \sigma}
\end{align*}
proving that $\beta_{\pi, \sigma}$ is an intertwiner. 
\end{proof}

\subsection{Braid relations}

\begin{prop}
The map $\beta_{\pi, \sigma}$ defined in Construction \ref{constr:Braiding} defines a braiding on $\SSS_p$.
\end{prop}

\begin{proof}
We check naturality and monoidality of $\beta$.

\item[\underline{Naturality:}]
Let $T \colon \pi \to \check{\pi}$ and $S \colon \sigma \to \check{\sigma}$ in $\SSS_p$.  
We must show that 
\[
\beta_{\pi, \check{\sigma}}(\id_{\pi} \circ_p S)
=
(S \circ_p \id_\pi)\beta_{\pi, \sigma}
\qquad\qquad
\text{and}
\qquad\qquad
\beta_{\check{\pi}, \sigma}(T \circ_p \id_\sigma)
=
(\id_\sigma \circ_p T)\beta_{\pi, \sigma}.
\]
We verify the second equation; the first is analogous. 
For our splitting $r,s\leq p$, we choose intertwiners
$u_r: \pi^r \to \pi$,
$v_s: \sigma^s \to \sigma$, and
$\check{u}_r: \check{\pi}^r \to \check{\pi}$,
and calculate
\begin{align*}
\beta_{\check{\pi}, \sigma} (T \circ_r \id_\sigma) 
& = 
v_s \sigma_p^s (\check{u}_r) \check{u}_r^* \check{\pi}_p (v_s^*) T 
&&  (\text{definition})\\
& = v_s \sigma_p^s (\check{u}_r) \check{u}_r^* T \pi_p (v_s^*) 
&& (\text{T intertwines $\pi_p$ and $\check{\pi}_p$}) \\
& = 
v_s \sigma_p^s (\check{u}_r) \sigma_r^s (\check{u}_r^* Tu_r) u_r^* \pi_p (v_s^*)
&& (\check{u}_r^* Tu_r \in A_r \ \text{by \ref{SSS:IntertwinerLocalized}}, \sigma_p^s \ \text{localized in } s \leq r^\prime ) \\
& = 
v_s \sigma_p^s (\check{u}_r) \sigma_p^s (\check{u}_r^* Tu_r) u_r^* \pi_p (v_s^*)
&& (A_r \subseteq A_p)\\
& = 
v_s \sigma_p^s (\check{u}_r)  \sigma_p^s (\check{u}_r^*) v_s^* \sigma_p (T u_r) v_s u_r^* \pi_p (v_s^*)  
&& (v_s \ \text{intertwines} \ \sigma^s \ \text{and} \ \sigma)\\
&= \sigma_p(T) v_s \sigma_p^s (u_r) u_r^* \pi_p (v_s^*) 
&& (v_s \ \text{intertwines} \ \sigma^s \ \text{and} \ \sigma)\\
& = (\id_\sigma \circ_p T) \beta_{\pi, \sigma}
&& (\text{definition})
\end{align*}

\item[\underline{Monoidality:}]
Let $\pi, \sigma, \tau \in \SSS_p$.
We must show that 
\[
\beta_{\pi \circ_p \sigma, \tau}
=
(\beta_{\pi, \tau} \circ_p \id_\sigma)(\id_\pi \circ_p \beta_{\sigma, \tau})
\qquad\qquad
\text{and}
\qquad\qquad
\beta_{\pi, \sigma \circ_p \tau}
=
(\id_\sigma \circ_p \beta_{\pi, \tau})(\beta_{\pi, \sigma} \circ_p \id_\tau).
\]
We verify the first equation; the second is analogous. 
For our splitting $r,s\leq p$, we choose intertwiners
$u_r: \pi^r \to \pi$,
$v_r: \sigma^r \to \sigma$, 
$v_s: \sigma^s \to \sigma$, and
$w_s: \tau^s \to \tau$.
Since $\pi^r \circ_p \sigma^r = \pi^r \circ_r \sigma^r$ is localized at $r$, 
\begin{align*}
\beta_{\pi \circ_p \sigma, \tau}  & = w_s \tau_p^s (u_r \pi_p^r (v_r )) \pi_p^r (v_r^*) u_r^* (\pi \circ_p \sigma )_p (w_s^*) && (\text{definition}) \\
& = w_s \tau_p^s (u_r) \tau_p^s (\pi_p^r (v_r)) \pi_p^r (v_r^*) u_r^* \pi_p (\sigma_p(w_s^*)) && (\text{definition}) \\
& = w_s \tau_p^s (u_r)  \pi_p^r (\tau_p^s (v_r)) \pi_p^r( v_r^*) u_r^* \pi_p (\sigma_p (w_s^*)) && (\text{by Lemma \ref{lem:LocalizedInDisjointCommute}}) \\
& = w_s \tau_p^s (u_r) u_r^* \pi_p(w_s)^* \pi_p (w_s \tau_p^s (v_r)v_r^* \sigma_p (w_s^*) ) 
&& (u_r \ \text{intertwines} \ \pi^r \ \text{and} \ \pi)
\\
& =  (\beta_{\pi, \tau} \circ_p \id_\sigma)(\id_\pi \circ_p \beta_{\sigma, \pi}) && (\text{definition})
\qedhere
\end{align*}
\end{proof}

\begin{prop}
\label{prop:BraidedEquivalence}
Suppose $a \leq b$ satisfy \ref{geom:SelfDisjoint}--\ref{geom:ZigZag}.
Then the inclusion $\SSS_a \hookrightarrow \SSS_b$ afforded by \ref{SSS:LocalizedInLarger} is a braided equivalence.
\end{prop}
\begin{proof}
Suppose $\pi, \sigma \in \SSS_a$. 
Let $r, s \leq a$ be a splitting of $a$. 
Then as in Construction \ref{constr:Braiding} with chosen intertwiners
$u_r: \pi^r \to \pi$ and
$v_s: \sigma^s \to \sigma$, we have that the braiding on $\SSS_a$ is given by 
\[
\beta_{\pi, \sigma}^a
=
v_s\sigma^s_a(u_r) u_r^* \pi_a(v_s)^*.
\]
Now, since $a \leq b$, we see $r, s \leq b$ is a splitting of $b$, and thus by isotony, 
\[
\beta_{\pi, \sigma}^a
=
v_s\sigma^s_a(u_r) u_r^* \pi_a(v_s)^*b
=
v_s\sigma^s_b(u_r) u_r^* \pi_b(v_s)^*
=
\beta_{\pi, \sigma}^b.
\qedhere
\]
\end{proof}

\begin{cor}
When $\cP$ satisfies \ref{geom:SelfDisjoint}--\ref{geom:ZigZag},
for all $p,q\in\cP$, we have a zig-zag of braided tensor equivalences
$\SSS_p\cong \SSS_q$.
\end{cor}

\section{Bounded spread Haag duality}
\label{sec:BoundedSpread}

We now weaken the notion of Haag duality to \emph{bounded spread Haag duality} in the spirit of \cite{2304.00068} for nets of algebras on the poset $\cC$ of cones in $\bbR^2$.
When $\Lambda$ is a cone, $\Lambda^{+s}$ means the cone obtained from $\Lambda$ by moving the boundaries of the cone perpendicularly by $s$ to obtain a larger cone.
In particular, for rectangles and cones,
\begin{equation}
\label{eq:ConeComplementSpread}
((\Lambda^{+s})^c)^{+s} = \Lambda^c.
\end{equation}
In order to weaken the Haag duality axiom, we add the small generation property 
to help with verification of locality similar to Lemma \ref{lem:SmallGenerationImpliesIsoneRepsAreLocal}.

\subsection{Bounded spread nets and their superselection sectors}

\begin{defn}
A \emph{bounded spread $\cC$-net of algebras} is a Hilbert space $H$
together with a von Neumann subalgebra $A_\Lambda\subset B(H)$ for each cone $\Lambda\in\cC$ satisfying
\begin{itemize}
\item (isotone)
$\Lambda\leq \Delta$ implies $A_\Lambda\subseteq A_\Delta$, 
\item (locality)
$\Lambda\leq \Delta'$ implies $A_\Lambda\subseteq A_\Delta'$,
\item (bounded spread Haag duality)
there is a global $s\geq 0$ called the \emph{duality spread} such that
$A_{\Lambda^c}' \subseteq A_{\Lambda^{+s}}$ 
(equivalently $A_{\Lambda^{+s}}' \subseteq A_{\Lambda^c}$)
for all $\Lambda\in\cC$,
\item (absorbing) 
for each $\Lambda\in\cP$, the representation $H$ for $A_\Lambda$ is absorbing, equivalently, each $A_\Lambda$ is properly infinite,\footnote{The result of Corollary \ref{cor:PNetsAreProperlyInfinite} still applies.
If $H$ is absorbing for each $A_\Lambda$, then each $A_{\Lambda^c}' \subseteq A_{\Lambda^{+s}}$ is properly infinite by Lemma \ref{lem:ProperlyInfiniteInclusion}.
Conversely, if each $A_\Lambda$ is properly infinite, then by locality, $\Lambda\leq \Delta'$ implies $A_\Lambda \leq A_{\Delta}'$ is again properly infinite by Lemma \ref{lem:ProperlyInfiniteInclusion}.} and
\item (small generated)
for each $\Delta\in\cP$, the von Neumann algebra $A_\Delta$ is generated by
$$
\set{A_\Lambda}{\text{$\Lambda\leq \Delta$ and $\Lambda$ is $\Delta$-small}}.
$$
\end{itemize}
Observe that a bounded spread $\cC$-net of algebras with duality spread $s=0$ is the previous notion of small generated $\cC$-net of algebras.
\end{defn}

\begin{ex}
Many $\cC$-nets of algebras coming from quantum spin systems can be shown to satisfy bounded spread Haag duality by constructing an \emph{intertwined} $\cC$-net of algebras which is known to satisfy Haag duality (see Proposition \ref{prop:IntertwinedProperties} below).
One such example coming from the Levin-Wen model was given in \cite[\S IV.E]{MR4808260} which will be further analyzed in \cite{2DBraidedSpinSystems}.

As an explicit example, if $\pi_0$ is the GNS representation of a product state, then $\pi_0$ satisfies Haag duality by \cite[Lem.~4.3]{MR4426734}, so $\pi_0 \circ \alpha$ satisfies bounded spread Haag duality if $\alpha$ is a quantum cellular automaton.  
Thus certain symmetry protected topological phases 
satisfy bounded spread Haag duality.
This fact will be used in \cite{SymmetryDefectsInfiniteVolume}.
\end{ex}

\begin{defn}
A \emph{(bounded spread) superselection sector} of $A=(H,\cC,A)$ is a Hilbert space $K$ together with normal maps $\pi_\Lambda:A_\Lambda\to B(K)$ satisfying
\begin{itemize}
\item (isotone)
$\Lambda\leq \Delta$ implies $A_\Lambda\subseteq A_\Delta$ and
\item 
(absorbing) 
for each $\Lambda\in\cP$, the representation $(K,\pi_\Lambda)$ for $A_\Lambda$ is absorbing, equivalently, injective.
\end{itemize}
The notion of intertwiner is the same as before.
The superselection sectors of a bounded spread net of algebras is again a $\rmW^*$-category.

Observe that when $A$ is a $\cC$-net of algebras satisfying Haag duality and the small generating property, a bounded spread superselection sector is exactly the previous notion of superselection sector by Lemma \ref{lem:SmallGenerationImpliesIsoneRepsAreLocal}.
\end{defn}

\begin{lem}
\label{lem:IsotonyImpliesBSLocality}
A bounded spread superselection sector $(\pi, K)$ of the bounded spread $\cC$-net of algebras $A = (H, \cC, A)$ satisfies the following locality condition: 
\begin{itemize}
    \item (locality) For every $\Lambda \in \cC$, $[\pi_{\Lambda^{+s}}(A_{\Lambda^c}'), \pi_{\Lambda^c}(A_{\Lambda^c})] = 0$.
\end{itemize}
\end{lem}

\begin{proof}
Let $\Lambda \in \cC$.
By the small-generating property, it suffices to show that for every $\Lambda$-small $\Delta \subseteq \Lambda^c$, 
\[
[\pi_{\Lambda^{+s}}(A_{\Lambda^c}'), \pi_{\Lambda^c}(A_{\Delta})] = 0.
\]
Let $\Delta \subseteq \Lambda^c$ be $\Lambda$-small.  
Then there exists $\Upsilon \in \cC$ such that $\Lambda, \Delta \subseteq \Upsilon$.  
We then have that $\Lambda^{+s}, \Delta \subseteq \Upsilon^{+s}$.  
Therefore, if $x \in A_{\Lambda^c}'$ and $y \in A_\Delta$, we have that 
\[
\pi_{\Lambda^{+s}}(x)\pi_{\Lambda^c}(y)
=
\pi_{\Lambda^{+s}}(x)\pi_{\Delta}(y)
=
\pi_{\Upsilon^{+s}}(x)\pi_{\Upsilon^{+s}}(y)
=
\pi_{\Upsilon^{+s}}(xy)
=
\pi_{\Upsilon^{+s}}(yx)
=
\pi_{\Lambda^c}(y)\pi_{\Lambda^{+s}}(x).
\qedhere
\]
\end{proof}

As before, we say that $(H, \pi) \in \SSS$ is \emph{localized at $\Lambda \in \cC$} if $\pi_{\Lambda^c} = \id_{A_{\Lambda^c}}$.
Note that \ref{SSS:AbilityToLocalize} and \ref{SSS:LocalizedInLarger} still hold, but \ref{SSS:LocalizedPreservesAlgebra} 
and
\ref{SSS:IntertwinerLocalized} respectively become:
\begin{enumerate}[label=(BSSS\arabic*), series=BSSS]
\setcounter{enumi}{1}
\item 
\label{BSSS:LocalizedPreservesAlgebraWithSpread}
If the superselection sector $(H,\pi)$ is localized at $\Lambda$, then $\pi_\Lambda A_\Lambda \subseteq A_{\Lambda^c}' \subseteq A_{\Lambda^{+s}}$.
\setcounter{enumi}{3}
\item 
\label{BSSS:IntertwinerLocalizedWithSpread}
If $(H,\pi)$ is localized at $\Lambda_1$, $(H,\sigma)$ is localized at $\Lambda_2$, and $\Lambda_1, \Lambda_2 \leq \Delta$,
then every intertwiner $u\in B(H)$ between $\pi$ and $\sigma$ lies in $A_{\Delta^c}' \subseteq A_{\Delta^{+s}}$.
\end{enumerate}

As before, we denote by $\SSS_\Lambda$ the full $\rmW^*$-subcategory of (bounded spread) superselection sectors localized at $\Lambda$.

\begin{construction}
If $\pi, \sigma \in \SSS_\Lambda$, we define $(\pi \circ_\Lambda \sigma)_\Delta$ for $\Delta \in \cC$ as follows. 
We pick $\widetilde{\Lambda} \subseteq \Lambda$ such that $\widetilde{\Lambda}^{+2s}$ is a $\Delta$-indicator and $\widetilde{\Lambda}^{+s}$ is $\Delta^{+s}$-small.
We then define
\[
(\pi \circ_\Lambda \sigma)_\Delta(x)
\coloneqq
u_{\widetilde{\Lambda}}\pi^{\widetilde{\Lambda}}_{\Lambda^{+s}}(v_{\widetilde{\Lambda}})
(\pi^{\widetilde{\Lambda}}_{\Delta^{+s}} \circ \sigma^{\widetilde{\Lambda}}_\Delta)(x)
\pi^{\widetilde{\Lambda}}_{\Lambda^{+s}}(v_{\widetilde{\Lambda}})^*u_{\widetilde{\Lambda}}^*
\]
for $x \in A_\Delta$.
Here, as usual, $\pi^{\widetilde{\Lambda}},\sigma^{\widetilde{\Lambda}}$ are localized in $\widetilde{\Lambda}$ and $u_{\widetilde{\Lambda}}, v_{\widetilde{\Lambda}}$ are unitaries satisfying that for every $\Sigma \in \cC$ 
\[
\pi_\Sigma(-)
=
u_{\widetilde{\Lambda}} \pi^{\widetilde{\Lambda}}_\Sigma(-) u_{\widetilde{\Lambda}}^*
\qquad\qquad
\text{and}
\qquad\qquad
\sigma_\Sigma(-)
=
v_{\widetilde{\Lambda}} \sigma^{\widetilde{\Lambda}}_\Sigma(-) v_{\widetilde{\Lambda}}^*.
\]
By a modification of the proof of Lemma \ref{lem:FusionWellDefined}, $(\pi \circ_\Lambda \sigma)_\Delta$ is independent of the choices made.  
\end{construction}

\begin{prop}
If $\pi, \sigma \in \SSS_\Lambda$, then $\pi \circ_\Lambda \sigma$ is a superselection sector localized in $\Lambda$.  
\end{prop}
\begin{proof}
The proofs for isotony and absorbing hold without much modification.  
To see that $\pi \circ_\Lambda \sigma$ is localized in $\Lambda$, let $\widetilde{\Lambda} \leq \Lambda$ satisfy that $\widetilde{\Lambda}^{+2s}$ is a $\Delta$-indicator and $\widetilde{\Lambda}^{+s}$ is $\Delta^{+s}$-small.
Let $x \in A_{\Lambda^c}$.  
We then have that 
\begin{align*}
(\pi \circ_\Lambda \sigma)_{\Lambda^c}(x)
&=
u_{\widetilde{\Lambda}}\pi^{\widetilde{\Lambda}}_{\Lambda^{+s}}(v_{\widetilde{\Lambda}})
(\pi^{\widetilde{\Lambda}}_{(\Lambda^c)^{+s}} \circ \sigma^{\widetilde{\Lambda}}_{\Lambda^c})(x)
\pi^{\widetilde{\Lambda}}_{\Lambda^{+s}}(v_{\widetilde{\Lambda}})^*u_{\widetilde{\Lambda}}^*
&&
\text{(definition)}
\\&=
u_{\widetilde{\Lambda}}\pi^{\widetilde{\Lambda}}_{\Lambda^{+s}}(v_{\widetilde{\Lambda}})
x
\pi^{\widetilde{\Lambda}}_{\Lambda^{+s}}(v_{\widetilde{\Lambda}})^*u_{\widetilde{\Lambda}}^*
&&
\text{($\pi^{\widetilde{\Lambda}}, \sigma^{\widetilde{\Lambda}}$ localized in $\widetilde{\Lambda} \subseteq \Lambda$)}
\\&=
\pi_{\Lambda^{+s}}(v_{\widetilde{\Lambda}})u_{\widetilde{\Lambda}}
x
u_{\widetilde{\Lambda}}^*
\pi_{\Lambda^{+s}}(v_{\widetilde{\Lambda}})^*
&&
\text{($u_{\widetilde{\Lambda}}$ intertwines $\pi^{\widetilde{\Lambda}}$ and $\pi$)}
\\&=
\pi_{\Lambda^{+s}}(v_{\widetilde{\Lambda}})x
\pi_{\Lambda^{+s}}(v_{\widetilde{\Lambda}})^*
&&
\text{($u_{\widetilde{\Lambda}} \in A_{\Lambda^c}'$)}
\\&=
\pi_{\Lambda^{+s}}(v_{\widetilde{\Lambda}})\pi_{\Lambda^c}(x)
\pi_{\Lambda^{+s}}(v_{\widetilde{\Lambda}})^*
&&
\text{($\pi$ localized in $\Lambda$)}
\\&=
\pi_{\Lambda^c}(x)
&&
\text{($v_{\widetilde{\Lambda}} \in A_{\Lambda^c}'$ and Lem.~\ref{lem:IsotonyImpliesBSLocality})}
\\&=
x
&&
\text{($\pi$ localized in $\Lambda$).}
\qedhere
\end{align*}

By modifying the proof of Lemma \ref{lem:MorphismTensorProduct}, we can see that if $T \colon \pi \to \check{\pi}$ and $S \colon \sigma \to \check{\sigma}$ in $\SSS_\Lambda$, then  $T \circ_{\Lambda} S \coloneqq T\pi_{\Lambda^{+s}}(S) = \check{\pi}_{\Lambda^{+s}}(S)T$ intertwines $\pi \circ_\Lambda \sigma$ and $\check{\pi} \circ_\Lambda \check{\sigma}$.  
Similarly, by modifying the proof of Lemma \ref{lem:StrictAssociativity}, one can see that $-\circ_\Lambda -$ is strictly associative.
Therefore $\SSS_\Lambda$ is a strict monoidal category.
Furthermore, by modifying the proof of Theorem \ref{restrictingequiv}, one can easily verify that if $\Lambda \leq \Delta$, then the inclusion $\SSS_\Lambda \hookrightarrow \SSS_\Delta$ afforded by \ref{SSS:LocalizedInLarger} is still a strict monoidal equivalence.

We can define a braiding on $\SSS_\Lambda$ by slightly modifying Construction \ref{constr:Braiding}.  
In particular, we let $\Lambda_1, \Lambda_2$ be a splitting of $\Lambda$ where $\Lambda_1^{+s} \leq \Lambda_2^c$.  
We then choose unitaries $u_{\Lambda_1},v_{\Lambda_2}\in A_{\Lambda^{+s}}$ which intertwine $\pi$ with $\pi^{\Lambda_1}\in \SSS_{\Lambda_1}$ and $\sigma$ with $\sigma^{\Lambda_2}\in \SSS_{\Lambda_2}$ respectively, i.e.,
$$
\pi(-) = u_{\Lambda_1}\pi^{\Lambda_1}(-)u_{\Lambda_1}^*
\qquad\qquad\qquad
\sigma(-) = v_{\Lambda_2} \sigma^{\Lambda_2}(-) v_{\Lambda_2}^*.
$$
We define a braiding from $\pi\circ_\Lambda \sigma$ to $\sigma\circ_\Lambda \pi$ by
$$
\beta_{\pi,\sigma} 
\coloneqq 
(v_{\Lambda_2} \circ_\Lambda u_{\Lambda_1})(u_{\Lambda_1}^* \circ_\Lambda v_{\Lambda_2}^*)
=
v_{\Lambda_2}\sigma^{\Lambda_2}_{\Lambda^{+s}}(u_{\Lambda_1}) u_{\Lambda_1}^* \pi_{\Lambda^{+s}}(v_{\Lambda_2})^*.
$$
The above map can be seen to be a braiding by modifying the proofs in \S\ref{sec:braiding}.
\end{proof}

\subsection{Intertwined nets of algebras}

\begin{defn}
Suppose $\{A_\Lambda\}_{\Lambda\in\cC}$ and $\{B_\Lambda\}_{\Lambda\in\cC}$ are two families of von Neumann subalgebras of $B(H)$.
We say $A,B$ are $t$-\emph{intertwined} if there is a global \emph{intertwining constant} $t\geq 0$ such that
$$
A_{\Lambda} \subseteq B_{\Lambda^{+t}}
\qquad\qquad
\text{and}
\qquad\qquad
B_{\Lambda} \subseteq A_{\Lambda^{+t}}
\qquad\qquad
\forall\,\Lambda\in\cC.
$$
\end{defn}

\begin{ex}
The recent article
\cite[\S IV.E]{MR4808260}
shows that a certain `edge restricted' subalgebra of the von Neumann cone algebras arising from the Levin-Wen string net model for a unitary fusion category $\cF$ is intertwined with the von Neumann cone algebras arising from a 2D braided fusion spin system with generator the canonical Lagrangian algebra in $Z(\cF)$.
This example will be further analyzed in \cite{2DBraidedSpinSystems}.

For an additional example of finite dimensional von Neumann algebras associated to the poset of intervals in $\bbR$, we refer the reader to \cite[\S III.C.2]{2309.13440}.
This is a particular example of Elliott's intertwining technique to construct isomorphisms of AF $\rmC^*$-algebras in \cite{MR0397420}.
\end{ex}

\begin{prop}
\label{prop:IntertwinedProperties}
Suppose $\{A_\Lambda\}_{\Lambda\in\cC}$ and $\{B_\Lambda\}_{\Lambda\in\cC}$ are $t$-intertwined families of von Neumann subalgebras of $B(H)$.
\begin{itemize}
\item 
If $A$ satisfies bounded spread Haag duality, then so does $B$.
\item
If each $A_\Lambda$ is properly infinite, then so is each $B_\Lambda$.
\item
If $H$ is absorbing for each $A_\Lambda$, then $H$ is absorbing for each $B_\Lambda$.
\end{itemize}
\end{prop}
\begin{proof}
\item[\underline{Bounded spread Haag duality:}]
If $A$ has duality spread $s$, then
\begin{align*}
B_{\Lambda^{+s+2t}}'
&\subseteq
A_{\Lambda^{+s+t}}'
&&
(\text{$t$-intertwined})
\\&\subseteq
A_{(\Lambda^{+t})^c}
&&
(\text{$A$ has duality spread $s$})
\\&\subseteq
B_{((\Lambda^{+t})^c)^{+t}}
&&
(\text{$t$-intertwined})
\\&=
B_{\Lambda^c}
&&
(\text{$(\Lambda^{+t})^c)^{+t}=\Lambda^c$ by \eqref{eq:ConeComplementSpread}}).
\end{align*}
\item[\underline{Properly infinite:}]
Since $A_\Lambda\subseteq B_{\Lambda^{+t}}$,
each $B_{\Lambda^{+t}}$ is properly infinite by \ref{lem:ProperlyInfiniteInclusion}.
The result follows.

\item[\underline{Absorbing:}]
By Lemma \ref{lem:absorbing}, $H$ is absorbing for each $A_\Lambda, B_\Lambda$ if and only if each $A_\Lambda',B_\Lambda'$ respectively
is properly infinite.
Since $B_{\Lambda}\subseteq A_{\Lambda^{+t}}$, 
$A_{\Lambda^{+t}}'\subseteq B_{\Lambda}'$, so $B_{\Lambda}'$ is properly infinite by Lemma \ref{lem:ProperlyInfiniteInclusion}.
\end{proof}

\begin{cor}
Suppose $(A,H)$ is a bounded spread $\cC$-net of algebras and $\{B_\Lambda\}_{\Lambda\in\cC}$ is an isotone, local, small generated family of von Neumann subalgebras of $B(H)$.
If $A,B$ are intertwined, then $B$ is also a bounded spread $\cC$-net of algebras.
\end{cor}

\begin{construction}
\label{const:IntertwinedNetsHaveSameSSS}
Suppose $(A,H)$ and $(B,H)$ are $t$-intertwined bounded spread $\cC$-nets of algebras.
For a superselection sector $(K,\pi^A)$ of $A$, we define 
$\pi^B_\Lambda:= \pi^A_{\Lambda^{+t}}|_{B_\Lambda}$.
These normal homomorphisms $\pi^B_\Lambda : B_\Lambda\to B(K)$ 
assemble into a superselection sector of $B$. 
\begin{itemize}
\item (Isotony)
Since $(K,\pi^A)$ is isotone, 
$\pi^A_{\Lambda^{+t}}|_{B_\Lambda}= \pi^A_{\Delta}|_{B_\Lambda}$
for any cone $\Delta$ containing $\Lambda^{+t}$.
Thus if $\Lambda\leq \Delta$, then $\Lambda^{+t}\leq \Delta^{+t}$, so 
$$
\pi^B_{\Delta}|_{B_\Lambda} 
= 
(\pi^A_{\Delta^{+t}}|_{B_\Delta})|_{B_\Lambda}
= 
\pi^A_{\Delta^{+t}}|_{B_\Lambda}
= 
\pi^A_{\Lambda^{+t}}|_{B_\Lambda}
=
\pi^B_{\Lambda}.
$$
\item (Absorbing)
Immediate by \ref{lem:EndoInjectiveIffAbs} as $\pi^B_{\Lambda}=\pi^A_{\Lambda^{+t}}|_{B_\Lambda}$ is injective.
\end{itemize}
Moreover, the map $(K,\pi^A)\mapsto (K,\pi^B)$ is functorial in that every morphism in $\SSS(A)$ is also a morphism in $\SSS(B)$.
\begin{itemize}
\item (Well-defined)
For a morphism $T \colon \pi^A \to \sigma^A$ in $\SSS(A)$,
and $\Lambda \in \cC$, $T\pi^A_\Lambda(-) = \sigma^A_\Lambda(-)T$.
In particular, $T\pi^A_{\Lambda^{+t}}(-) = \sigma^A_{\Lambda^{+t}}(-)T$, and thus $T\pi^B_\Lambda(-) = \sigma^B_\Lambda(-)T$ since $\pi^B_\Lambda$ and $\sigma^B_\Lambda$ are defined by restriction.
Thus $T \colon \pi^B \to \sigma^B$ is also a morphism in $\SSS(B)$.
\end{itemize}
\end{construction}

\begin{thm}
Suppose $(A,H)$ and $(B,H)$ are intertwined bounded spread $\cC$-nets of algebras.
The map $(K,\pi^A)\mapsto (K,\pi^B)$ is an isomorphism of $\rmW^*$-categories $\SSS(A)\cong \SSS(B)$ with inverse $(K,\pi^B)\mapsto (K,\pi^A)$.
\end{thm}
\begin{proof}
Observe that Construction \ref{const:IntertwinedNetsHaveSameSSS} can be applied with the roles of $A,B$ reversed, giving us functors both ways.
Since the hom spaces in $\SSS(A)$ and $\SSS(B)$ are identical, it remains to show $(K,((\pi^A)^B)^A)=(K,\pi^A)$ for any $(K,\pi^A)\in \SSS(A)$.
Indeed, for any $\Lambda\in\cC$, by isotony, 
\[
((\pi^A)^B)^A_\Lambda
=
(\pi^A)^B_{\Lambda^{+t}}|_{A_\Lambda}
=
(\pi^A_{\Lambda^{+2t}}|_{B_{\Lambda^{+t}}})|_{A_\Lambda}
=
\pi^A_{\Lambda^{+2t}}|_{A_\Lambda}
=
\pi^A_\Lambda.
\qedhere
\]
\end{proof}

\begin{thm}
Suppose $(A,H)$ and $(B,H)$ are intertwined bounded spread $\cC$-nets of algebras with intertwining constant $t\geq 0$.
Restricting to $\SSS(A)_\Lambda$, the superselection sectors of $A$ localized at $\Lambda$,
we get a strict unitary braided equivalence $\SSS(A)_{\Lambda} \to \SSS(B)_{\Lambda^{+t}}$
\end{thm}
\begin{proof}
We first define the following notation: for $\pi \in \SSS(A)_\Lambda$, we let $\pi^B$ be the image of $\pi$ under the equivalence $\SSS(A)_{\Lambda} \to \SSS(B)_{\Lambda^{+t}}$.
Now, let $\pi, \sigma \in \SSS(A)_\Lambda$.
Let $\Delta \in \cC$, and let $\widetilde{\Lambda} \leq \Lambda$ satisfy that $\widetilde{\Lambda}^{+2s}$ is a $\Delta$-indicator and $\widetilde{\Lambda}^{+s}$ is $\Delta^{+s}$-small.
We choose $\pi^{\widetilde{\Lambda}},\sigma^{\widetilde{\Lambda}} \in SSS(A)_{\widetilde{\Lambda}}$ and unitaries $u_{\widetilde{\Lambda}}, v_{\widetilde{\Lambda}}$ satisfying that for every $\Sigma \in \cC$ 
\[
\pi_\Sigma(-)
=
u_{\widetilde{\Lambda}} \pi^{\widetilde{\Lambda}}_\Sigma(-) u_{\widetilde{\Lambda}}^*
\qquad\qquad
\text{and}
\qquad\qquad
\sigma_\Sigma(-)
=
v_{\widetilde{\Lambda}} \sigma^{\widetilde{\Lambda}}_\Sigma(-) v_{\widetilde{\Lambda}}^*.
\]
Then we have that 
\[
(\pi \circ_\Lambda \sigma)_\Delta(-)
=
u_{\widetilde{\Lambda}}\pi^{\widetilde{\Lambda}}_{\Lambda^{+s}}(v_{\widetilde{\Lambda}})
(\pi^{\widetilde{\Lambda}}_{\Delta^{+s}} \circ \sigma^{\widetilde{\Lambda}}_\Delta)(-)
\pi^{\widetilde{\Lambda}}_{\Lambda^{+s}}(v_{\widetilde{\Lambda}})^*u_{\widetilde{\Lambda}}^*
\]
Now, $u_{\widetilde{\Lambda}} \colon \pi^B \to (\pi^{\widetilde{\Lambda}})^B$ and $v_{\widetilde{\Lambda}} \colon \sigma^B \to (\sigma^{\widetilde{\Lambda}})^B$, and $(\pi^{\widetilde{\Lambda}})^B, (\sigma^{\widetilde{\Lambda}})^B$ are localized in $\widetilde{\Lambda}^{+t}$.
Therefore, we have that
\[
(\pi \circ_\Lambda \sigma)^B_{\Delta^{+t}}(-)
=
u_{\widetilde{\Lambda}}(\pi^{\widetilde{\Lambda}})^B_{\Lambda^{+s+t}}(v_{\widetilde{\Lambda}})
((\pi^{\widetilde{\Lambda}})^B_{\Delta^{+s+t}} \circ (\sigma^{\widetilde{\Lambda}})^B_{\Delta^{+t}})(-)
(\pi^{\widetilde{\Lambda}})^B_{\Lambda^{+s+t}}(v_{\widetilde{\Lambda}})^*u_{\widetilde{\Lambda}}^*
=
(\pi^B \circ_{\Lambda^{+t}} \sigma^B)_{\Delta^{t}}(-).
\]
This shows that $(\pi \circ_\Lambda \sigma)^B = \pi^B \circ_{\Lambda^{+t}} \sigma^B$.
Similarly, if $T \colon \pi \to \check{\pi}$ and $S \colon \sigma \to \check{\sigma}$ in $\SSS(A)_\Lambda$, then 
\[
T \circ_{\Lambda}^A S
=
T\pi_{\Lambda^{+s}}(S)
=
T\pi^B_{\Lambda^{+s+t}}(S)
=
T \circ_\Lambda^B S.
\]
Therefore, the equivalence $\SSS(A)_\Lambda \to \SSS(B)_{\Lambda^{+t}}$ is a strict monoidal equivalence.

It remains to show that this equivalence is braided.
However, this follows by from the fact that it is the identity on intertwiners.
\end{proof}

\section{Equivalence with Naaijkens superselection sectors}

In this section, we prove that our definition of superselection sectors for the poset of cones in $\bbR^2$ agrees with the definition used to study topologically ordered quantum spin systems in (2+1)-dimensions \cite{MR2804555, MR3426207,MR4362722,2306.13762,2310.19661}.
In this setting, we assume that the net of cone algebras satisfies bounded spread Haag duality.  
We first describe an abstract characterization of the approach used in \cite{MR660538, MR2804555} and show that we obtain a braided monoidal equivalence with our approach.  
We then relate our approach to the approach used in the study of spin systems.

\subsection{Buchholz--Fredenhagen superselection sectors}
\label{subsec:BFBraidedMonoidal}

Let $\cC$ be the poset of cones in $\bbR^2$. 
We let $(A, H)$ be a bounded spread $\cC$-net of algebras. 
We fix a cone $\Omega \in \cC$, which will determine a `forbidden direction.'
Analogously to \cite{MR660538, MR2804555}, we define a $\rmC^*$-algebra 
\[
A^{\aux}
\coloneqq
\overline{\bigcup_{\Lambda : \Lambda \Cap \Omega = \emptyset} A_{\Lambda}}^{\|\cdot\|}.
\]
In this section, we assume that for every $\Lambda \in \cC$, $(A_\Lambda \cap A^{\aux})'' = A_\Lambda \subset B(H)$.

\begin{rem}
One way to ensure that $(A_\Lambda \cap A^{\aux})'' = A_\Lambda$ for every $\Lambda\in \cC$
is if each $A_\Lambda$ is generated as a von Neumann algebra by subalgebras $A_B$ associated to open bounded regions $B \subset \Lambda \subset \bbR^2$.
Observe that this condition implies that $(A_\Lambda \cap A^{\aux})'' = A_\Lambda$ for every $\Lambda \in \cC$ since $A_B \subseteq A^{\aux}$ for every open bounded $B \subset \bbR^2$.
Indeed, for every open bounded $B$, there is a cone $\Delta\in \cC$ with $B\subset \Delta$ and $\Delta\Cap \Omega=\emptyset$.
\end{rem}

We now construct another tensor category of \emph{Buchholz--Fredenhagen} superselection sectors localized
in $\Lambda \in \cC$ with $\Lambda \Cap \Omega = \emptyset$ analogously to the one in \cite[\S 4]{MR660538}. 
Since we will now have multiple notions of superselection sectors, in this section, we write $\SSS^{GF}_\Lambda$ for the category of \emph{Gabbiani--Fr\"ohlich} superselection sectors for $(A,H)$ constructed in \S\ref{sec:Fusion}-\ref{sec:BoundedSpread} above.

\begin{defn}
Let $\pi \colon A^{\aux} \to B(H)$. 
We say that $\pi$ is \emph{localized} in $\Lambda \in \cC$ if $\pi|_{A_{\Lambda^c} \cap A^{\aux}} = \id$.  
We say that $\pi$ is \emph{transportable} if for every cone $\Delta \in \cC$, there exists  $\pi^\Delta \colon A^{\aux} \to B(H)$ that is unitarily equivalent to $\pi$ such that $\pi^\Delta$ is localized in $\Delta$.

A \emph{Buchholz--Fredenhagen superselection sector} of the bounded spread $\cC$-net $(A, H)$
is a map $\pi \colon A^{\aux} \to B(H)$ that is localized in some $\Lambda \in \cC$ and transportable.

We write $\SSS_\Lambda^{BF}$ to denote the category of Buchholz--Fredenhagen superselection sectors localized in $\Lambda$. 
The morphisms in this category are intertwiners, that is, $T \colon \pi \to \sigma$ if $T \in B(H)$ and $T\pi(x) = \sigma(x)T$ for all $x \in A^{\aux}$.
\end{defn}

\begin{facts}
Here are some useful facts about Buchholz--Fredenhagen superselection sectors that mirror Facts \ref{facts:SSSFacts}.
\begin{enumerate}[label=(BF\arabic*)]
\item 
\label{BF:LocalizedPreservesAlgebra}
Suppose $\pi \colon A^{\aux} \to B(H)$ is localized in $\Lambda$.
Then $\pi(A_\Lambda \cap A^{\aux}) \subset A_{\Lambda^{+s}}$.
\begin{proof}
By bounded spread Haag duality, it suffices to show that $\pi(A_\Lambda \cap A^{\aux}) \subset A_{\Lambda^c}'$.  
Let $x \in A_\Lambda \cap A^{\aux}$ and $y \in A_{\Lambda^c} \cap A^{\aux}$.  
Then we have that 
\[
\pi(x)y
=
\pi(x)\pi(y)
=
\pi(xy)
=
\pi(yx)
=
\pi(y)\pi(x)
=
y\pi(x).
\]
Therefore, we have that 
\[
\pi(x)
\in
(A_{\Lambda^c} \cap A^{\aux})'
=
(A_{\Lambda^c} \cap A^{\aux})'''
=
A_{\Lambda^c}'.
\qedhere
\]
\end{proof}
\item 
\label{BF:LocalizedInLarger}
Suppose $\pi \colon A^{\aux} \to B(H)$ is localized in $\Lambda$ and $\Lambda \leq \Delta$.
Then $\pi$ is also localized in $\Delta$.  
\begin{proof}
Suppose $\Sigma \in \cC$ is disjoint from $\Delta$.
Then $\Sigma$ is disjoint from $\Lambda$.
Therefore, since $\pi$ is localized in $\Lambda$, $\pi|_{A_\Sigma \cap A^{\aux}} = \id$.
\end{proof}
\item 
\label{BF:EndoOfAaux}
Let $\Lambda \in \cC$ such that $\Lambda \Cap \Omega = \emptyset$. 
If $\pi \colon A^{\aux} \to B(H)$ is localzed in $\Lambda$, then $\pi(A^{\aux}) \subseteq A^{\aux}$.
\begin{proof}
Let $\Delta \in \cC$ with $\Delta \Cap \Omega = \emptyset$. 
It suffices to show that $\pi(A_\Delta) \subset A^{\aux}$.
Since $\Lambda \Cap \Omega = \Delta \Cap \Omega = \emptyset$, there exists $\Sigma \in \cC$ with $\Sigma \Cap \Omega = \emptyset$ such that $\Lambda, \Delta \leq \Sigma$.  
We then have that $A_\Delta \subset A_\Sigma$ and that $\pi$ is localized in $\Sigma$ by \ref{BF:LocalizedInLarger}.
Since $\Sigma \Cap \Omega = \emptyset$, $A_\Sigma \subset A_{\Sigma^{+s}} \subset A^{\aux}$, so by \ref{BF:LocalizedPreservesAlgebra}, 
\[
\pi(A_\Delta)
\subseteq
\pi(A_\Sigma)
\subset
A_{\Sigma^{+s}}
\subset
A^{\aux}.
\qedhere
\]
\end{proof}

\item 
\label{BF:IntertwinerLocalized}
Let $\Lambda \in \cC$ and $\pi, \sigma \colon A^{\aux} \to B(H)$ be localized in $\Lambda$.
Suppose $T \colon \pi \to \sigma$ is an intertwiner.
Then $T \in A_{\Lambda^{+s}}$.
In particular, if $\Lambda \Cap \Omega = \emptyset$, then $T \in A^{\aux}$.

\begin{proof}
It suffices to show that $T \in A_{\Lambda^c}'$ by bounded spread Haag duality.
Let $x \in A_{\Lambda^c} \cap A^{\aux}$. 
Then since $\pi, \sigma$ are localized in $\Lambda$, we have that
\[
Tx
=
T\pi(x)
=
\sigma(x)T
=
xT.
\]
Therefore, we have that 
\[
T 
\in
(A_{\Lambda^c} \cap A^{\aux})'
=
A_{\Lambda^c}'.
\qedhere
\]
\end{proof}

\item 
\label{BF:NormalAbsorbing}
Suppose $\Lambda \in \cC$ and $\pi \in \SSS_\Lambda^{BF}$.
Then for every $\Delta \in \cC$, $\pi|_{A_\Delta \cap A^{\aux}}$ has a unique normal extension to $A_\Delta$, which is absorbing.

\begin{proof}
Since $\pi \in \SSS_\Lambda^{BF}$, $\pi$ is transportable.
In particular, there exists $\pi^{\Delta^c} \colon A^{\aux} \to B(H)$ localized in $\Delta^c$ such that $\pi^{\Delta^c} \cong \pi$.
Let $u_{\Delta^c} \colon \pi^{\Delta^c} \to \pi$ be a unitary witnessing this equivalence.
We then have that for $x \in A_\Delta \cap A^{\aux}$, 
\[
\pi(x)
=
u_{\Delta^c}\pi^{\Delta^c}(x)u_{\Delta^c}^*
=
u_{\Delta^c}xu_{\Delta^c}^*.
\]
The above formula is clearly normal and injective, so $\pi$ has a unique normal extension to $A_\Delta$, which is injective and thus absorbing by Lemma \ref{lem:EndoInjectiveIffAbs}.
\end{proof}
\end{enumerate}
\end{facts}

\begin{construction}
\label{Const:BFBraidedMonoidal}
Let $\Lambda \in \cC$ such that $\Lambda \Cap \Omega = \emptyset$.
We now define a strict tensor product and a braiding on $\SSS_\Lambda^{BF}$.

The monoidal product on $\SSS_\Lambda^{BF}$ is given by composition of endomorphisms.
In detail, for $\pi, \sigma \in \SSS_\Lambda^{BF}$, $\pi \otimes_{BF} \sigma \coloneqq \pi \circ \sigma$, and for $T\colon \pi_1 \to \pi_2$ and $S \colon \sigma_1 \to \sigma_2$ in $\SSS_\Lambda^{BF}$, $T \otimes_{BF} S \coloneqq T\pi_1(S) = \pi_2(S)T$.
Since $\Lambda \Cap \Omega = \emptyset$, $\sigma(A^{\aux}) \subseteq A^{\aux}$ by \ref{BF:EndoOfAaux}, so $\pi \circ \sigma$ is well-defined.
Similarly, $S \in A^{\aux}$ by \ref{BF:IntertwinerLocalized}, so $T \otimes_{BF} S$ is well-defined.
Therefore, $-\otimes_{BF} -$ endows $\SSS_\Lambda^{BF}$ with the structure of a strict monoidal category.

The braiding on $\SSS_\Lambda^{BF}$ is defined analogously to the braiding in \cite{MR2804555} and in fact agrees with the braiding in \cite{MR4362722}.
Let $\Delta \in \cC$ be a cone such that $\Delta \Cap \Omega = \emptyset$ and such that $\Delta^{+s}$ is disjoint from $\Lambda$.
To fix a convention, we additionally assume that $\Delta$ is counterclockwise of $\Lambda$, meaning that if one rotates $\Lambda$ counterclockwise, the rotated cone intersects $\Delta$ before intersecting $\Omega$.
Given $\pi, \sigma \in \SSS_\Lambda^{BF}$, we pick $\sigma^\Delta \cong \sigma$ localized in $\Delta$ and a unitary $v_\Delta \colon \sigma^{\Delta} \to \sigma$.
The braiding is then defined by 
\[
\beta_{\pi, \sigma}
\coloneqq
(v_\Delta \otimes_{BF} \pi)(\pi \otimes_{BF} v_\Delta^*)
=
v_\Delta \pi(v_\Delta^*).
\]
By modifying the proofs in \S\ref{sec:braiding}, one can see that $\beta_{\pi, \sigma}$ is a braiding that is independent of the choices made.
\end{construction}

\begin{thm}
\label{thm:GF-BFBraidedEquivalence}
Let $\Lambda \in \cC$ such that $\Lambda \Cap \Omega = \emptyset$.
The map $\pi \mapsto \pi^{\aux}$, where $\pi^{\aux} \colon A^{\aux} \to B(H)$ is the unique extension of $\pi$ to $A^{\aux}$ guaranteed by isotony,
is a strict braided equivalence $\SSS_\Lambda^{GF} \to \SSS_\Lambda^{BF}$.
\end{thm}

\begin{proof}
We first show that the map $\pi \to \pi^{\aux}$ is an equivalence of categories.  
Since $(A_\Delta \cap A^{\aux})'' = A_\Delta$ for all $\Delta \in \cC$, $T \in B(H)$ intertwines $\pi, \sigma \in \SSS_\Lambda^{GF}$ if and only if $T$ intertwines $\pi^{\aux}, \sigma^{\aux}$.
Therefore, $\pi \mapsto \pi^{\aux}$ defines a fully faithful functor.  
This functor is essentially surjective by \ref{BF:NormalAbsorbing}.
Specifically, if $\rho \in \SSS_\Lambda^{BF}$, then by \ref{BF:NormalAbsorbing}, for every $\Delta \in \cC$, $\rho|_{A_\Delta \cap A^{\aux}}$ has a unique normal extension $\pi_\Delta$ to $A_\Delta$, which is absorbing.  
We then obtain $\pi \in \SSS_\Lambda^{GF}$, and $\pi^{\aux} = \rho$ by construction.

We now show that this equivalence is strict monoidal.
We first show that for $\pi, \sigma \in \SSS_\Lambda^{GF}$, $(\pi \circ_\Lambda \sigma)^{\aux} = \pi^{\aux} \circ \sigma^{\aux}$.
It suffices to show that for every $\Delta \in \cC$ with $\Delta \Cap \Omega = \emptyset$, $(\pi \circ_\Lambda \sigma)^{\aux}(x) = (\pi^{\aux} \circ \sigma^{\aux})(x)$ for every $x \in A_{\Delta}$.
Let $\Delta \in \cC$ with $\Delta \Cap \Omega = \emptyset$.
Then there exists a $\Sigma \in \cC$ with $\Sigma \Cap \Omega = \emptyset$ such that $\Lambda, \Delta \leq \Sigma$.  
Since $\Lambda \leq \Sigma$, $(\pi \circ_\Lambda \sigma)_{\Sigma} = \pi_{\Sigma^{+s}} \circ \sigma_{\Sigma}$.
In addition, since $\Delta \leq \Sigma$, we have that $A_{\Delta} \subseteq A_{\Sigma}$.  
Thus, for any $x \in A_{\Delta}$, we have that
\[
(\pi \circ_\Lambda \sigma)^{\aux}(x)
=
(\pi \circ_\Lambda \sigma)_{\Sigma}(x)
=
(\pi_{\Sigma^{+s}} \circ \sigma_{\Sigma})(x)
=
(\pi^{\aux} \circ \sigma^{\aux})(x).
\]
Furthermore, if $T \colon \pi \to \check\pi$ and $S \colon \sigma \to \check\sigma$ in $\SSS_\Lambda^{GF}$, we have that 
\[
T\circ_\Lambda S
=
T\pi_{\Lambda^{+s}}(S)
=
T\pi^{\aux}(S)
=
T \otimes_{BF} S,
\]
where the tensor product at the end is in the category $\SSS_\Lambda^{BF}$.
Therefore, the functor $\pi \mapsto \pi^{\aux}$ is a strict monoidal equivalence.

It remains to show that this equivalence is braided.
Let $\pi, \sigma \in \SSS_\Lambda^{GF}$.
Let $\Delta \in \cC$ be a cone such that $\Delta \Cap \Omega = \emptyset$ and such that $\Delta^{+s}$ is disjoint from $\Lambda$.
As before, we additionally assume that $\Delta$ is counterclockwise of $\Lambda$.  
We choose $\sigma^\Delta \cong \sigma$ localized in $\Delta$ and a unitary $v_\Delta \colon \sigma^{\Delta} \to \sigma$.
Also, since $\Delta \Cap \Omega = \Lambda \Cap \Omega = \emptyset$, there exists $\Sigma \in \cC$ with $\Lambda, \Delta \subset \Sigma$ such that $\Sigma \Cap \Omega = \emptyset$.
We then have that 
\[
\beta_{\pi^{\aux}, \sigma^{\aux}}
=
v_\Delta \pi^{\aux}(v_\Delta^*)
=
v_\Delta \pi_{\Sigma^{+s}}(v_\Delta^*).
\]
The expression on the right is exactly the bounded spread version of the braiding defined in Construction \ref{constr:Braiding} on $\SSS_\Sigma^{GF}$.
By Proposition \ref{prop:BraidedEquivalence}, this is the same as the braiding on $\SSS_\Lambda^{GF}$.
Therefore, the functor $\pi \mapsto \pi^{\aux}$ defines a braided equivalence.
\end{proof}

\subsection{Naaijkens superselection sectors}

Let $\Gamma$ be an infinite planar lattice, let $\fA \coloneqq \bigotimes_{j \in \Gamma} M_d(\bbC)$, and let $\fA_{\loc}$ be the finitely supported (local) operators in $\fA$.  
For $\Lambda \subseteq \Gamma$, we define $\fA_\Lambda \coloneqq \bigotimes_{j \in \Lambda} M_d(\bbC)$.
Let $\omega_0 \colon \fA \to \bbC$ be a pure state, and let $(\pi_0, \Omega, \cH_0)$ be its GNS representation.  
We assume that $\pi_0$ satisfies bounded spread Haag duality with respect to cones, that is, there exists a global $s \geq 0$ for any cone $\Lambda$, $\pi_0(\fA_{\Lambda^c})' \subseteq \pi_0(\fA_{\Lambda^{+s}})''$.  
For simplicity, we write $\cR_\Lambda \coloneqq \pi_0(\fA_\Lambda)''$ for a cone $\Lambda$.  
Since $\omega_0$ is a pure state, $\cR_\Lambda$ is a factor for any cone $\Lambda$.  
We additionally assume that $\cR_\Lambda$ is an infinite factor for every cone $\Lambda$.  
This is true for instance when $\omega_0$ is translation invariant, by a standard argument used in \cite[Thm.~5.1]{MR2804555}.  
However, there are other ways to ensure that $\cR_\Lambda$ is always infinite.
For example, $\cR_\Lambda$ is infinite if $\omega_0$ is a gapped ground state for a Hamiltonian with uniformly bounded finite range interactions \cite[Lem.~5.3]{MR4362722}.
In this case, the collection of von Neumann algebras $\cR_\Lambda$, together with the vacuum Hilbert space $\cH_0$, forms a bounded spread $\cC$-net of algebras.
(Note that the representation $\cR_\Lambda \hookrightarrow B(\cH_0)$ is absorbing by Lemma \ref{lem:absorbing}.)

\begin{defn}
A \emph{Naaijkens superselection sector} is a representation $\pi$ of $\fA$ satisfying that for all cones $\Lambda \in \cC$, 
\[
\pi|_{\fA_\Lambda}
\cong
\pi_0|_{\fA_\Lambda}.
\]
\end{defn}

It is well-known that for every $\Lambda \in \cC$ with $\Lambda \Cap \Omega = \emptyset$, the category of Naaijkens superselection sectors is equivalent to $\SSS_\Lambda^{BF}$ (see for instance \cite{MR2804555,MR3135456}).
Furthermore, the monoidal product and braiding that is defined for Naaijkens superselection sectors in \cite{MR2804555, MR3426207, MR4362722, 2306.13762} is exactly the monoidal product and braiding defined in Construction \ref{Const:BFBraidedMonoidal}.
Therefore, the following result is a corollary of Theorem \ref{thm:GF-BFBraidedEquivalence}.

\begin{cor}
Let $\Lambda \in \cC$ such that $\Lambda \Cap \Omega = \emptyset$.
Then $\SSS_\Lambda^{GF}$ is braided monoidally equivalent to the category of Naaijkens superselection sectors localized at a cone $\Lambda$.
\end{cor}


\appendix
\section{Geometric axioms for posets of intervals and cones}

\subsection{Intervals on \texorpdfstring{$S^1$}{S1}}
\label{appendix:IntervalsSatisfyGAs}

We prove the geometric axioms for intervals on $S^1$.
\begin{defn}
An interval $I$ in $S^1$ is a pair of distinct points $x,y\in S^1$ and a choice of `interior'.
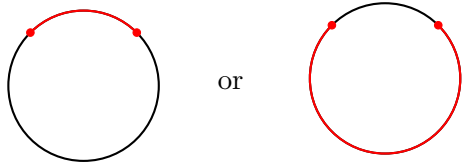
$$
\tikzmath{
\draw[thick] (0,0) circle (1cm);
\filldraw[red] (45:1) circle (.05cm);
\filldraw[red] (135:1) circle (.05cm);
\draw[thick,red] (45:1) arc(45:135:1cm) (135:1);
}
\qquad\text{or}\qquad
\tikzmath{
\draw[thick] (0,0) circle (1cm);
\filldraw[red] (45:1) circle (.05cm);
\filldraw[red] (135:1) circle (.05cm);
\draw[thick,red] (45:1) arc(45:-225:1cm) (135:1);
}
$$
\end{defn}
In this subsection, we take $\mathcal{P}$ to be the poset of intervals on the circle given by inclusion and with an involution given by complements. 
For distinct points $x,y\in S^1$, we write $[x,y]\in\mathcal{P}$ for the interval going clockwise from $x$ to $y$.

It is clear that intervals on $S^1$ satisfy \ref{geom:SelfDisjoint}--\ref{geom:splitting}.

\begin{prop}
    Intervals on the circle satisfy the zig-zag axiom \ref{geom:ZigZag}.
\end{prop}
\begin{proof}
    Consider two $p,q\in\cP$. For $q$-small $q$-indicators $\widetilde{p},\widehat{p}\leq p$, if $r\in \widetilde{p}\cap\widehat{p}$, then $r$ is a $q$-small $q$-indicator. Therefore, $\widetilde{p},\widehat{p}$ are connected by a zig-zag as in $\ref{geom:ZigZag}$ given by $(\widetilde{p},\widetilde{p},r,\widehat{p},\widehat{p})$. Any two $q$-small $q$-indicator intervals in a connected component of the intersection of $p$ and $q$ (or $p$ and $q'$) may be connected in this way. 

    Now suppose there is a boundary point of $q$ which ends in the interior of $p$.  
    Then we can find small enough $q$-small $q$-indicator intervals $\widetilde{p},\widehat{p}\leq p$ on either side of this boundary which end very close to this boundary. Then there is a $q$-small $r\leq p$ straddling the boundary of $q$ with $z_1\in r\cap \widetilde{p}$ and $z_2\in r\cap \widehat{p}$. The following diagram demonstrates this geometry.
$$
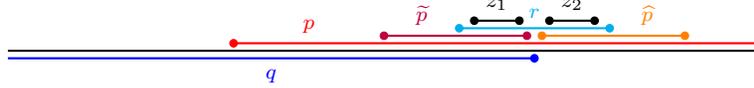
\tikzmath{
\draw[thick] (-5,0) -- (5,0);
\draw[thick, blue] (2,-.1) --node[below]{$\scriptstyle q$} (-5,-.1);
\filldraw[blue] (2,-.1) circle (.05cm);
\draw[thick, red] (-2,.1) --node[above]{$\scriptstyle p$} (0,.1) -- (5,.1);
\filldraw[red] (-2,.1) circle (.05cm);
\draw[thick, purple] (0,.2) --node[above]{$\scriptstyle \widetilde{p}$} (1,.2) -- (1.9,.2);
\filldraw[purple] (0,.2) circle (.05cm);
\filldraw[purple] (1.9,.2) circle (.05cm);
\draw[thick, orange] (2.1,.2) -- (3,.2) --node[above]{$\scriptstyle \widehat{p}$} (4,.2);
\filldraw[orange] (2.1,.2) circle (.05cm);
\filldraw[orange] (4,.2) circle (.05cm);
\draw[thick, cyan] (1,.3) --node[above]{$\scriptstyle r$} (3,.3);
\filldraw[cyan] (1,.3) circle (.05cm);
\filldraw[cyan] (3,.3) circle (.05cm);
\draw[thick] (1.2,.4) --node[above]{$\scriptstyle z_1$} (1.8,.4);
\filldraw[] (1.2,.4) circle (.05cm);
\filldraw[] (1.8,.4) circle (.05cm);
\draw[thick] (2.2,.4) --node[above]{$\scriptstyle z_2$} (2.8,.4);
\filldraw[] (2.2,.4) circle (.05cm);
\filldraw[] (2.8,.4) circle (.05cm);
}
$$
    It is clear that  $(\widetilde{p},\widetilde{p},z_1,r,z_2,\widetilde{p},\widetilde{p})$ is a zig-zag as in \ref{geom:ZigZag}.

    These zig-zags within and between the connected components of $p$ intersect $q$ are enough to establish zig-zags as in \ref{geom:ZigZag} between all $q$-small $q$-indicators in $p$.
\end{proof}

\begin{rem}
Every interval on $S^1$ satisfies the reflection axiom \ref{geom:Braid}, i.e., every splitting admits a reflection.
This can be seen from the diagram below.
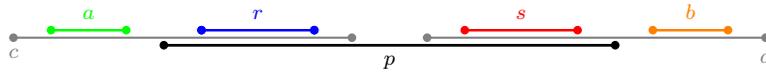
\[
\tikzmath{
\draw[thick, black] (3,0) --node[below]{$\scriptstyle p$} (-3,0) -- (3,0);
\filldraw[black] (-3,0) circle (.05cm);
\filldraw[black] (3,0) circle (.05cm);
\draw[thick, gray] (-5,.1) --node[below]{$\scriptstyle c$} (-5,0.1) -- (-.5,0.1);
\filldraw[gray] (-5,0.1) circle (.05cm);
\filldraw[gray] (-.5,0.1) circle 
(.05cm);
\draw[thick, gray] (5,.1) --node[below]{$\scriptstyle d$} (5,0.1) -- (.5,0.1);
\filldraw[gray] (5,0.1) circle (.05cm);
\filldraw[gray] (.5,0.1) circle (.05cm);
\draw[thick, green] (-3.5,.2) --node[above]{$\scriptstyle a$} (-4.5,0.2) -- (-3.5,0.2);
\filldraw[green] (-4.5,0.2) circle (.05cm);
\filldraw[green] (-3.5,0.2) circle (.05cm);
\draw[thick, orange] (3.5,.2) --node[above]{$\scriptstyle b$} (4.5,0.2) -- (3.5,0.2);
\filldraw[orange] (4.5,0.2) circle (.05cm);
\filldraw[orange] (3.5,0.2) circle (.05cm);
\draw[thick, blue] (-1,.2) --node[above]{$\scriptstyle r$} (-2.5,0.2) -- (-1,0.2);
\filldraw[blue] (-1,0.2) circle (.05cm);
\filldraw[blue] (-2.5,0.2) circle (.05cm);
\filldraw[blue] (-1,0.2) circle (.05cm);
\filldraw[blue] (-2.5,0.2) circle (.05cm);
\draw[thick, red] (1.,.2) --node[above]{$\scriptstyle s$} (2.5,0.2) -- (1,0.2);
\filldraw[red] (1,0.2) circle (.05cm);
\filldraw[red] (2.5,0.2) circle (.05cm);
}
\]
\end{rem}

\subsection{Cones in \texorpdfstring{$\bbR^2$}{R2}}
\label{appendix:ConesSatisfyGAs}

\begin{defn}
A cone $\Lambda$ in $\bbR^2$ is a point $\lambda_0\in \bbR^2$ called the \emph{apex} together with two distinct rays from $\lambda_0$ to $\infty$ and a choice of `interior'. 
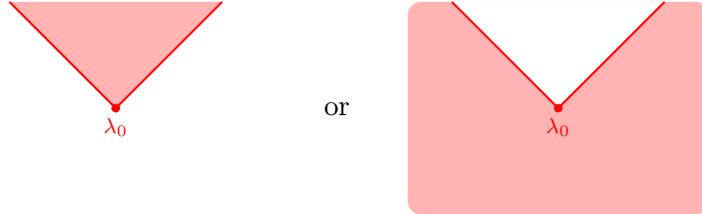
$$
\tikzmath{
\filldraw (0,0) node[below]{$\scriptstyle z$} circle (.05cm);
\fill[fill=white, rounded corners = 5pt] (-2,-1.414) rectangle (2,1.414);
\filldraw[draw=red,thick, fill=red!30] (135:2cm) -- (0,0) -- (45:2cm);
\filldraw[red] (0,0) node[below]{$\scriptstyle \lambda_0$} circle (.05cm);
}
\qquad\text{or}\qquad
\tikzmath{
\filldraw[red] (0,0) node[below]{$\scriptstyle z$} circle (.05cm);
\fill[fill=red!30, rounded corners = 5pt] (-2,-1.414) rectangle (2,1.414);
\filldraw[draw=red, thick, fill=white] (135:2cm) -- (0,0) -- (45:2cm);
\filldraw[red] (0,0) node[below]{$\scriptstyle \lambda_0$} circle (.05cm);
}
$$
\end{defn}

\begin{defn}
To each cone $\Lambda$  in $\bbR^2$, we can associate a canonical interval by
$I(\Lambda):=(\Lambda-\lambda_0)\cap S^1$.
Observe that $I$ is a poset map which preserves the order 2 involutions.
\end{defn}

\begin{ex}
For all cones $\Lambda$ and $s>0$, $I(\Lambda^{+s})=I(\Lambda)$.
\end{ex}

\begin{facts}
We have the following facts regarding cones and their canonical intervals whose proofs are left to the reader.
Let $\Lambda,\Delta$ be cones in $\bbR^2$.
\begin{enumerate}[label=($\Lambda$\arabic*)]
\item
\label{cones:IntersectionEquivalence}
The intersection $\Lambda\cap \Delta$ contains a cone if and only if $I(\Lambda)\cap I(\Delta)$ contains an interval.
That is, $\Lambda\Cap \Delta\neq\emptyset$ if and only if $I(\Lambda)\cap I(\Delta)\neq \emptyset$.
\item 
\label{cones:IntersectionExistence}
For every interval $J\subset I(\Lambda)\cap I(\Delta)$, there is a cone $\Gamma \subset \Lambda\cap \Delta$ with $I(\Gamma)=J$.
\item 
\label{cones:UnionEquivalence}
The union $\Lambda\cup \Delta$ is contained in a cone if and only if $I(\Lambda)\cup I(\Delta)\neq S^1$.
\item
\label{cones:UnionExistence}
For any interval $J\subset S^1\setminus (I(\Lambda)\cup I(\Delta))$,
there is a cone $\Sigma$ containing $\Lambda\cup \Delta$ with $I(\Sigma^c)=J$.
\end{enumerate}
\end{facts}

\noindent
Clearly cones in $\bbR^2$ satisfy \ref{geom:SelfDisjoint}--\ref{geom:splitting}.
We now sketch a proof of the zig-zag axiom \ref{geom:ZigZag}.

For $r>0$, define $C_r$ to be the circle of radius $r$ about $0$ and the interval
$$
I_r(\Lambda):= r^{-1}\cdot (\Lambda \cap C_r)\subset S^1.
$$
Observe that
$I(\Lambda) = \lim_{r\to \infty} I_r(\Lambda)$
in the sense that the endpoints of $I(\Lambda)$ on $S^1$ are the limit of the endpoints of the $I_r(\Lambda)\subset S^1$.
We leave the verification of the following straightforward lemma to the reader.

\begin{lem}
\label{lem:ZoomOut}
For every $r,\varepsilon>0$, there is an $R>r>0$ such that for all cones $\Lambda$ with apex $\lambda\in \overline{B_r(0)}$, 
the endpoints of $I_s(\Lambda)$ and $I(\Lambda)$ are within $\varepsilon$ whenever $s\geq R$.
\qed
\end{lem}

We now give our strategy to prove \ref{geom:ZigZag} for cones from the proof for intervals.
\begin{enumerate}[label=\underline{Step \arabic*:}]
\item 
Suppose $\Lambda,\Delta$ are two cones and $\widetilde{\Lambda},\widehat{\Lambda}$ are two $\Delta$-small $\Delta$-indicators in $\Lambda$.
Let $r>0$ so that the apexes of  $\widetilde{\Lambda},\widehat{\Lambda},\Lambda,\Delta$ lie in $\overline{B_r(0)}$.
Choose $\varepsilon>0$ sufficiently small and $R>r>0$ sufficiently large so that Lemma \ref{lem:ZoomOut} holds for $\widetilde{\Lambda},\widehat{\Lambda},\Lambda,\Delta$.
\item 
Choose an $S>R$ sufficiently large and carefully choose a zig-zag as in \ref{geom:ZigZag} for the intervals $I_{S}(\widetilde{\Lambda}),I_S(\widehat{\Lambda}),I_S(\Lambda),I_S(\Delta)$ on $C_S$.
Call this zig-zag $(I_{S}(\widetilde{\Lambda})=J_1,K_1,J_2,K_2,\dots,J_{n+1}=I_S(\widehat{\Lambda}))$.
\item
Carefully pick apexes for the $J_j$ and $K_k$ on $C_R$ so that the cones they produce form the desired zig-zag for $\widetilde{\Lambda},\widehat{\Lambda}$ as in \ref{geom:ZigZag}.
\end{enumerate}

The $\varepsilon>0$ chosen in Step 1 should be sufficiently small relative to the lengths of the intervals $I(\widetilde{\Lambda}),I(\widehat{\Lambda}),I(\Lambda),I(\Delta)$ and
their complements, and all the subsequent pairwise intersections.
The adverb `carefully' in Steps 2 and 3 means to take into account this $\varepsilon$-tolerance.

\bibliographystyle{alpha}
{\footnotesize{
\bibliography{../../../bibliography/bibliography}
}}
\end{document}